\documentclass[a4paper,11pt]{amsart} 
\usepackage[utf8]{inputenc}
\usepackage{pifont}

\usepackage[english]{babel}
\usepackage{amsmath}
\usepackage{amsthm}
\usepackage{amsfonts,mathrsfs,relsize,bigints,stmaryrd}
\usepackage{amssymb}
\usepackage{anysize}
\usepackage{hyperref}
\usepackage{appendix}
\usepackage{mathtools}
\usepackage{graphicx}
\usepackage{xcolor}
\usepackage{ulem}
\makeatletter  \@addtoreset{equation}{section} \makeatother
\newtheorem{defi}{Definition}[section]
\newtheorem{lem}{Lemma}[section]
\newtheorem{theo}{Theorem}[section]
\newtheorem{cor}{Corollary}[section]
\newtheorem{pro}{Proposition}[section]

\newtheorem{rem}{Remark}[section]

\DeclareMathOperator{\N}{\mathbb{N}}
\DeclareMathOperator{\R}{\mathbb{R}}
\DeclareMathOperator{\C}{\mathbb{C}}
\DeclareMathOperator{\T}{\mathbb{T}}

\allowdisplaybreaks 

{\thanks{C.G. has been supported by the European Research Council ERC-StG-852741 (CAPA), the MINECO--Feder (Spain) research grant number RTI2018--098850--B--I00, and the Junta de Andaluc\'ia (Spain) Project
FQM 954,  and { J. M.  has been partially supported by PID2020-112881GB-I00 and Severo Ochoa and Maria de Maeztu program for centers CEX2020-001084-MMTM2016--75390 (Mineco, Spain)}}


\begin{document}

\title[Time  periodic doubly connected solutions for the 3D quasi-geostrophic model]{Time  periodic doubly connected solutions for \\the 3D quasi-geostrophic model}

\author[C. Garc\'ia]{Claudia Garc\'ia}
\address{ Departament de Matem\`atiques i Inform\`atica, Universitat de Barcelona, 08007 Barcelona, Spain \& Research Unit ``Modeling Nature'' (MNat), Universidad de Granada, 18071 Granada, Spain}
\email{ claudiagarcia@ub.edu}
\author[T. Hmidi]{Taoufik Hmidi}
\address{NYUAD Research Institute, New York University Abu Dhabi, PO BOX 129188, Abu Dhabi, United Arab Emirates. Univ Rennes, CNRS, IRMAR-UMR 6625, F-35000 Rennes, France}
\email{thmidi@univ-rennes1.fr}
\author[J. Mateu]{Joan Mateu}
\address{Departament de Matem\`atiques, Universitat Aut\`onoma de Barcelona, 08193 Bellaterra; Barcelona, Catalonia}
\email{mateu@mat.uab.cat}

\date{\today}

\begin{abstract}
In this paper, we construct time periodic doubly connected  solutions for the 3D quasi-geostrophic model in the patch setting. More specifically, we prove the existence of nontrivial $m$-fold doubly connected rotating patches bifurcating  from a {\it generic} doubly connected revolution shape domain with higher symmetry  $m\geq m_0$ and $m_0$ is large enough. The linearized matrix operator at the equilibrium state is with variable and singular coefficients and its spectral analysis is performed  via  the approach devised in \cite{GHM} where a suitable symmetrization has been introduced.  New difficulties emerge due to the interaction between the surfaces making the spectral problem richer and involved. 
\end{abstract}

\maketitle

\tableofcontents

\section{Introduction}

The quasi-geostrophic motion describes the evolution of a large scale fluid with rapid background rotation and strong stratification. The derivation comes from the primitive equations, that is, the 3D inhomogeneous Euler equations by taking into account the stratification effects and the Earth's rotation. After considering the Boussinesq approximation to the primitive equations and, formally, assuming vanishing Rossby and Froude numbers one arrives to the 3D quasi--geostrophic model describing the evolution of the potential vorticity $q$ which is merely advected by the fluid:
\begin{eqnarray}   \label{equation-model}
       \left\{\begin{array}{ll}
          	\partial_t q+u\partial_1 q+v\partial_2 q=0, & (t,x)\in[0,+\infty)\times\mathbb{R}^3, \\
         	 \Delta \psi=q,& \\
         	 	u=-\partial_2 \psi, \ v=\partial_1 \psi,& \\
         	 q(0,x)=q_0(x).& 
       \end{array}\right.
\end{eqnarray}
We refer to \cite{B-B,Charv,Desjardins-Grenier,Iftimie2,Ped} for the formal derivation of the above model. The second equation involving the standard Laplacian of $\R^3$ can be inverted  using the Green's function  leading to the following representation of the stream function $\psi$,
\begin{equation}\label{Stream-F}
\psi(t,x)=-\frac{1}{4\pi}\bigintsss_{\R^3}\frac{q(t,y)}{|x-y|}dA(y),
\end{equation}
where $dA$ denotes the usual Lebesgue measure. The velocity field $(u,v,0)$  is solenoidal and can be recovered  from $q$ through the Biot--Savart law,
$$
U(t,x):=(u,v)(t,x)=\frac{1}{4\pi}\bigintsss_{\R^3}\frac{(x_1-y_1, x_2-y_2)^{\perp}}{|x-y|^3}q(t,y)dA(y),
$$
{where  $(x_1,x_2)^\perp=(-x_2,x_1)\in\R^2.$} Let us point out that although the velocity field is planar, that is, its third component is vanishing, it strongly depends on the third variable $x_3$. From  the homogeneity of the kernel in  Biot-Savart law, Yudovich theory \cite{Yudovich} can be  implemented as in the two-dimensional case leading to global in time solutions and uniqueness for $q_0\in L^1\cap L^\infty.$
{
Existence can be proved following closely the argument in \cite[Chapter 8]{MB} for the vorticity equation. For the uniqueness one resorts to \cite{NPS}, where one proves uniqueness in $ L^1\cap L^\infty$ for the continuity equation in higher dimensions with velocity field given by the convolution with a kernel whose gradient is a singular integral plus a multiple of a delta function. Our kernel $K$ fulfils this last condition and, furthermore, it is divergence free which concludes that the associated continuity and transport equations are the same. The special form $K=L(\nabla N)$ being $N$ the fundamental solution of the Laplacian  is necessary to adapt the argument in \cite{NPS}.}

This allows to deal in particular with  discontinuous solutions in the patch form for this model. More specifically, if the initial datum takes the patch form \begin{equation}\label{patch-intro}
q_0={\bf 1}_D,
\end{equation}
with $D$ being a bounded domain of $\R^3$, then by virtue of the transport equation of \eqref{equation-model}, the solution keeps this structure for any time, that is,  $q(t,\cdot)={\bf 1}_{D(t)}$. The contour dynamics equation modeling the evolution of the boundary of $D_t$ agrees with
\begin{equation}\label{contour-dynam-intro}
(\partial_t \gamma_{t}-(U,0)(t,\gamma_t))\cdot n(\gamma_t)=0,
\end{equation}
where $\gamma_t$ is a parametrization of the boundary  $\partial D(t)$ and $n(\gamma_t)$ is a normal vector to the boundary at the point $\gamma_t$. 

The motivation to study patch solutions for \eqref{equation-model} comes from the 2D Euler equations. The persistence of the regularity of the boundary in the class $\mathscr{C}^{1,\alpha}$, with $\alpha\in(0,1)$,  { for 2D Euler equation} was obtained in \cite{B-C,Chemin,Serf}, and the ill-posedness for $\mathscr{C}^2$ was recently achieved in \cite{Kiselev}. The dynamics of the boundary is rich and undergoes complex behavior from filamentation to the emergence of ordered structures.  There are  explicit steady patch solutions for the 2D Euler equations such as  Rankine vortex (the circular patch) and the Kirchhoff ellipses \cite{Kirchhoff}. However, a lot of implicit time-periodic solutions have been discovered  along the past few decades using bifurcation theory, desingularization techniques and variational tools. For more details we refer to \cite{ Burbea,CastroCordobaGomezSerrano, DeemZabusky, DelaHozHmidiMateuVerdera, DHHM, G-KVS, Gar-21, GH22, GHS,GPSY, HH21, HHHM-15, HMW, HW21, HmidiHozMateuVerdera, HmidiMateu, H-M, HM16, HmidiMateuVerdera, HMV15} and the references therein. These results have been extended to other 2D active scalar equations such as the generalized surface quasi-geostrophic model or the quasi-geostrophic shallow water equations, see for instance \cite{Cas0-Cor0-Gom, Cas-Cor-Gom, C-C-GS-2,DelaHoz-Hassainia-Hmidi, D-H-R, G-S-19, Hassa-Hmi, HHH-18}.
A recent exploration on time quasi-periodic solutions based on KAM theory  has been set up for  active scalar  equations in \cite{ BHM22,HHM21,HR22,HR21}.  
Coming back to the 3D quasi-geostrophic model, the persistence of the $\mathscr{C}^{1,\alpha}$ boundary regularity, with $\alpha\in(0,1)$, of the patch was recently obtained in \cite{CMOV}. Moreover, this system enjoys a rich set of  stationary  solutions. Actually any revolution shape domain about the vertical axis is a stationary solution.
Hence, it is of great importance to check  for this Hamiltonian system whether  time periodic solutions may exist on the vicinity of some stationary solutions with revolution shape form. Their structures take the patch  form rotating about the vertical axis with constant angular velocity, that is,\begin{equation}\label{evolution-intro}
D(t)=\mathcal{R}_{\Omega t}D(0), \quad \mathcal{R}_{\Omega t}(x_1,x_2,x_3)=(e^{i\Omega t}(x_1,x_2),x_3).
\end{equation}
In the literature we can find explicit solutions in the ellipsoid form and  some numerical experiments on the existence of these structures have been also achieved, see for instance  \cite{D-S-R, Dristchel2, Dristchel,Rein-Drit, Rein}. Very recently, the authors have explored in  \cite{GHM} the existence of time periodic solutions around  simply-connected revolution shape domain. They obtained m-fold non trivial structures around generic domains. In the current paper, we intend to continue this  program  by analyzing a similar problem but with different topological structure. More precisely,  we shall investigate  periodic solutions around doubly connected revolution shape domains.
This is equivalent to find two bounded domains $D_1$ and $D_2$ of $\R^3$, with $D_{2}$ embedded in $D_1$ ($D_2\Subset D_{1}$)  such that
\begin{equation}\label{sol-intro}
q(t,\cdot)={\bf 1}_{D(t)},\quad D(t)=\mathcal{R}_{\Omega t}D, \quad  D=D_{1}\backslash D_{2},
\end{equation}
in a small neighborhood of  the following stationary solution
\begin{equation}\label{stat-intro}
q_0={\bf 1}_{D_0}, \quad D_0=D_{0,1}\backslash D_{0,2}, \quad D_{0,2}\Subset D_{0,1},
\end{equation}
where $D_{0,j}$ is a revolution shape domain with a particular structure. In the same spirit of  \cite{GHM}, the way we parametrize the domain $D$ is a key point for the analysis and spherical coordinates type are privileged, that is for each $j\in\{1,2\},$\begin{align}\label{param-intro}
\partial D_j=&\left\{\big(r_j(\phi,\theta) e^{i\theta},d_j\cos(\phi)\big),\quad (\phi,\theta)\in[0,\pi]\times\T  \right\},
\end{align}
for $d_1>d_2>0$ and with
\begin{align}\label{f-intro}
r_j(\phi,\theta)=&r_{0,j}(\phi)+f_j(\phi,\theta)\quad\hbox{with} \quad f_j(\phi,\theta)=\sum_{n\geqslant 1}f_{j,n}(\phi)\cos(n\theta).
\end{align}
Note that for $f_j=0$ we recover the parametrization of the  revolution shape domain  $D_{0,j}$ covering by this way a large class of stationary solutions of the type \eqref{stat-intro}. Let us point out that due to the above expression for $f_j$ we are assuming that  the horizontal sections of $D_j$ are  at least one-fold.  In addition, we shall work with the natural Dirichlet  boundary conditions
\begin{align*}
r_{0,j}(0)=r_{0,j}(\pi)=f_j(0,\theta)=f_j(\pi,\theta)=0,\quad j=1,2.
\end{align*}
Inserting now the ansatz \eqref{sol-intro} in \eqref{contour-dynam-intro}, we find that the associated relative stream function is constant at the boundary, meaning that,
\begin{align*}
\psi\big(r_j(\phi,\theta)e^{i\theta},d_j\cos(\phi)\big)-\tfrac{\Omega}{2}r_j^2(\phi,\theta)=&m_j(\Omega,f_1,f_2)(\phi),\quad (\phi,\theta)\in[0,\pi]\times\T,
\end{align*}
where $m_j(\Omega,f_1,f_2)$ depends only on $\phi$  
 and given
by  the average in $\theta$ of the left hand side, 
\begin{align}\label{meanT-intro}
m_j(\Omega,f_1,f_2)(\phi)=\frac{1}{2\pi}\int_0^{2\pi}\left\{\psi\big(\gamma_j(\phi,\theta)\big)-\tfrac{\Omega}{2}r^2_j(\phi,\theta)\right\}d\theta.
\end{align}
Let us define the functional ${F}=({F}_1,{F}_2)$ through
\begin{align}\label{nonlinearfunction2-intro}
{F}_j(\Omega,f_1,f_2)(\phi,\theta)=&\psi\big(r_j(\phi,\theta)e^{i\theta},d_j\cos(\phi)\big)-\tfrac{\Omega}{2}r^2_j(\phi,\theta)-m_j(\Omega,f_1,f_2)(\phi),
\end{align}
where the stream function takes in view of \eqref{Stream-F} the form 
\begin{align*}
\psi(Re^{i\theta},z)=&-\frac{d_1}{4\pi}\bigintsss_{0}^{\pi}\bigintsss_0^{2\pi}\bigintsss_0^{r_1(\varphi,\eta)}\frac{\sin(\varphi)rdrd\eta d\varphi}{|(re^{i\eta},d_1\cos(\varphi))-(Re^{i\theta},z)|}\\
&+\frac{d_2}{4\pi}\bigintsss_{0}^{\pi}\bigintsss_0^{2\pi}\bigintsss_0^{r_2(\varphi,\eta)}\frac{\sin(\varphi)rdrd\eta d\varphi}{|(re^{i\eta},d_2\cos(\varphi))-(Re^{i\theta},z)|}\cdot
\end{align*}
Then, the problem of finding rotating doubly connected patches agrees to find the roots of the nonlinear and nonlocal functional $F$, that is, to study the equation 
\begin{equation}\label{Formu-st1-intro}
{F}(\Omega,f_1,f_2)(\phi,\theta)=0, \quad \forall (\phi,\theta)\in[0,\pi]\times\T.
\end{equation}
Since $F$ is vanishing at $\phi\in\{0,\pi\}$ due to the Dirichlet conditions and motivated by the structure of the linearized operator around the stationary solution, we find it convenient  to filter the singularities of the poles at the nonlinear level by working with   the modified functional 
\begin{align}\label{Ftilde-intro}
\tilde{F}(\Omega,f_1,f_2)(\phi,\theta)=&\big(\tilde{F}_1(\Omega,f_1,f_2)(\phi,\theta),\tilde{F}_2(\Omega,f_1,f_2)(\phi,\theta)\big)\\
=&\left(\frac{F_1(\Omega,f_1,f_2)(\phi,\theta)}{r_{0,1}(\phi)}, \frac{F_2(\Omega,f_1,f_2)(\phi,\theta)}{r_{0,2}(\phi)} \right).\nonumber
\end{align}
The aim of this work is to study the existence of nontrivial solutions by perturbing a large family of stationary revolution shape domains of the type \eqref{stat-intro}. For this goal some conditions listed below  on the initial profiles  $\{r_{0,j},j=1,2\}$ are required and denoted in what follows by {\bf (H)}:
\begin{itemize}
\item[{\bf(H1)}] {\it{(Regularity assumption)}} For $j=1,2$,  ${ r_{0,j}}\in \mathscr{C}^{2}([0,\pi])$.
\item[{\bf(H2)}] {\it{(Non-self intersection)}} There exists $C>0$ such that
$$
\forall\,\phi\in[0,\pi],\quad  C^{-1}\sin\phi\leq { r_{0,j}}(\phi)\leq C\sin(\phi).
$$
\item[{\bf(H3)}] {\it{(Equatorial symmetry)}} ${ r_{0,j}}$ is symmetric with respect to $\phi=\frac{\pi}{2}$, i.e., 
$${ r_{0,j}}\left({\pi}-\phi\right)={r_{0,j}}\left(\phi\right),\,\forall\, \phi\in[0,{\pi}].
$$
\item[{\bf(H4)}] {\it (Interfaces separation)} There exists $\delta>0$ such that  
\begin{align*}
 \forall \phi,\varphi\in [0,\pi],\quad { (r_{0,1}(\phi)-r_{0,2}(\varphi))^2}+(d_1\cos(\phi)-d_2\cos(\varphi))^2\geqslant \delta,
 \end{align*}
\end{itemize}
supplemented with the strong condition
\begin{equation}\label{assump-omega-intro}
\sup_{\phi\in[0,{\pi}]} \tfrac{{U_0^2\big(r_{0,2}(\phi),0,d_2\cos(\phi)\big)}}{r_{0,2}(\phi)}:=\overline\Omega_2<
\inf_{\phi\in[0,{\pi}]} \tfrac{U_0^2\big(r_{0,1}(\phi),0,d_1\cos(\phi)\big)}{r_{0,1}(\phi)}:=\overline\Omega_1,
\end{equation}
where $U_0^2$ is the second component of the velocity field associated to the revolution shape domain.

We shall come back later in Section \ref{Pro-assu} to these assumptions with some comments on their usefulness.
The main result of this paper, is the following informally stated theorem. For a more detailed and complete  version we refer to Theorem \ref{theorem}.
\begin{theo}\label{theo-intro}
Let $r_{0,j}$ satisfy  ${\bf (H)}$ and \eqref{assump-omega-intro},
{ for $j=1,2.$}
 Then there is   $m_0\in\N$ such that for any $m\geq m_0$ there exists a curve of nontrivial rotating doubly connected solutions to \eqref{sol-intro} taking   the form \eqref{param-intro}--\eqref{f-intro}, where the horizontal sections of the domains $D_j$  are $m$-fold symmetric.
 \end{theo}
 Before outlining the main ideas of the proof, we should emphasize that the set of domains subject to the conditions {\bf (H)} together with the spectral one  \eqref{assump-omega-intro} is non empty. This point will be carefully  discussed later in Section \ref{sec-examples} where we show in particular their validity for ellipsoid revolution shape domains or any of their smooth perturbations. Another class of well-separated revolution shape domains  is  also provided.

In the following we shall sketch the main ideas of  the proof of Theorem \ref{theo-intro} and give some insights  on   the assumption \eqref{assump-omega-intro} where it is needed. Basically, we implement bifurcation theory by applying  Crandall-Rabinowitz theorem to $\tilde{F}$. This was successfully performed in \cite{GHM}  with patches surrounded by only one surface around smooth revolution shapes.  The goal is to find nontrivial  roots for the nonlinear functional $\tilde{F}$ introduced in \eqref{Ftilde-intro} as a bifurcation from the trivial one $(0,0)$. The linearized operator  of $\tilde{F}$ around the trivial solution  is with variable coefficients and   no longer scalar as with only one surface due to the interaction between the two surfaces. Therefore the spectral study sounds  at this level extremely   subtle and very  sensitive to the geometry of the stationary domains. One of the  delicate key point in the spectral study is to identify a set of parameters $\Omega$ where the kernel is simple together with the transversality assumption. We succeed to construct such kind of set using the largest eigenvalues of some stratified $2\times2$ matrix self-adjoint compact operators. Indeed, 
using  Fourier series expansion $\displaystyle{h_j(\phi,\theta)=\sum_{n\geqslant 1}h_{j,n}(\phi)\cos(n\theta)}$, one finds in view of \eqref{Ftilde-intro}  the following expression for the linearized operator,
\begin{align}\label{lin-op-intro}
\partial_{f_1,f_2}\tilde{F}_i(\Omega,0,0)(h_1,h_2)(\phi,\theta)=&(-1)^{i-1}\nu_{i,\Omega}(\phi)\sum_{n\geqslant 1}\Big(h_{i,n}(\phi)-\mathcal{T}^n_{i,\Omega}(h_{1,n},h_{2,n})(\phi)\Big)\cos(n\theta),
\end{align}
where $\nu_{i,\Omega}$ is related to the surface of  the stationary patch according to the formula,
\begin{align}\label{nu-streamfunction-intro}
\nu_{\Omega,i}(\phi)=(-1)^{i-1}\left(\tfrac{1}{r_{0,i}(\phi)} ({\partial_1} \psi)\big({r_{0,i}(\phi),0, }d_i\cos(\phi)\big)-\Omega\right).
\end{align}
For more details, see Remark \ref{rema-1}.
Notice that, with respect to some aspects, we find it more convenient in view of Lemma \ref{nuOmega12} to recast the operator $\mathcal{T}^n_{i,\Omega}$ in terms of Gauss hypergeometric functions as follows
\begin{align}\label{Tnomega2-intro}
\mathcal{T}^{n}_{i,\Omega} (h_1,h_2)(\phi)=&d_1(-1)^{i-1}\int_0^{\pi} \tfrac{H_{i,1}^n(\phi,\varphi)}{\nu_{i,\Omega}(\phi)}h_1(\varphi)d\varphi+d_2(-1)^{i}\int_0^\pi \tfrac{H_{i,2}^n(\phi,\varphi)}{\nu_{i,\Omega}(\phi)}h_2(\varphi)d\varphi,
\end{align}
with
\begin{align*}
H_{i,j}^n(\phi,\varphi)=&\tfrac{2^{2n-1}\left(\frac12\right)^2_{n}}{(2n)!}\frac{\sin(\varphi){r_{0,i}^{n-1}(\phi)r_{0,j}^{n+1}(\varphi)}}{\left[R_{i,j}(\phi,\varphi)\right]^{n+\frac12}} F_n\left(\tfrac{4{r_{0,i}(\phi)r_{0,j}(\varphi)}}{R_{i,j}(\phi,\varphi)}\right),
\end{align*}
where $F_n$ refers to the Gauss hypergeometric function 
$$
F_n(x)=F\left(n+\tfrac{1}{2},n+\tfrac{1}{2};2n+1;x\right), \quad  x\in[0,1), 
$$
and $R_{i,j}$ is defined as
\begin{align*}
R_{i,j}(\phi,\varphi)=&({ r_{0,i}(\phi)+r_{0,j}(\varphi)})^2+(d_i\cos(\phi)-d_j\cos(\varphi))^2.
\end{align*}
By using \eqref{lin-op-intro}, we observe that solving  the kernel equation of $\partial_{f_1,f_2}\tilde{F}_i(\Omega,0,0)$ is equivalent to the following eigenvalue problem
$$
\exists n\geqslant 1,\exists (h_{1},h_2)\neq(0,0)\quad\hbox{s.t}\quad \quad\mathcal{T}_{i,\Omega}^n(h_{1},h_2)(\phi)=h_i(\phi),  \forall  i=1,2.
$$
{Therefore, we can ensure that the kernel is non trivial if $1$ is an eigenvalue of $\mathcal{T}_{i,\Omega}^n$ for some $n\geqslant1.$ In that way, the kernel study reduces to an eigenvalue problem and one has to perform a deep study about the eigenvalues of $\mathcal{T}_{i,\Omega}^n$ in order to see whether $1$ can be in the discrete spectrum by suitable adjustment of $\Omega$.  The key point to conduct the spectral study is to extract first  a compact self-adjoint   structure of $\mathcal{T}_{i,\Omega}^n$ with respect to a suitable   Hilbert space. For this aim, we introduce the following Borel measure}
\begin{equation}\label{measure-intro}
d\mu_j(\varphi)=\sin(\varphi)r_{0,j}^2(\varphi)\nu_{j,\Omega}(\varphi)d\varphi,\quad j=1,2.
\end{equation}
To get a positive measure we need  to assume the spectral assumption \eqref{assump-omega-intro} and impose the constraint  $\Omega\in(\overline\Omega_2,\overline\Omega_1)$. It follows that  $\mathcal{T}_{i,\Omega}^n$ can be written in the form
\begin{align}\label{TnomegaPP2-intro}
\mathcal{T}_{1,\Omega}^n (h_1,h_2)(\phi)=&d_1\int_0^{\pi} K_{1,1}^n(\phi,\varphi)h_1(\varphi)d\mu_1(\varphi)-d_2\int_0^\pi {K_{1,2}^n(\phi,\varphi)}h_2(\varphi)d\mu_2(\varphi),\\
\nonumber\mathcal{T}_{2,\Omega}^n (h_1,h_2)(\phi)=& d_2\int_0^{\pi}K_{2,2}^n(\phi,\varphi) h_2(\varphi)d\mu_2(\varphi)-d_1\int_0^\pi K_{2,1}^n(\phi,\varphi)h_1(\varphi)d\mu_1(\varphi),
\end{align}
with
\begin{align}\label{K-kernel-intro}
K_{i,j}^n(\phi,\varphi)=\frac{H_{i,j}^n(\phi,\varphi)}{\nu_{i,\Omega}(\phi)\nu_{j,\Omega}(\varphi)\sin(\varphi){r_{0,j}^2}(\varphi)}.
\end{align}
Now, we consider the vectorial weighted Hilbert space $\mathbb{H}_{\Omega}$ endowed  with the inner product
\begin{align*}
\nonumber H=(h_1,h_2), \widehat{H}=(\widehat{h}_1,\widehat{h}_2),\quad \langle H,\widehat{H}\rangle_{\Omega}&=\int_0^\pi h_1(\varphi)\widehat{h}_1(\varphi)d\mu_1(\varphi)+{\tfrac{d_2}{d_1}}\int_0^\pi h_2(\varphi)\widehat{h}_2(\varphi)d\mu_2(\varphi).
\end{align*}

Let us point out that the kernel $K_{i,j}^n$ is symmetric and one may check according to Lemma \ref{lem-ssym} that $\mathcal{T}_{i,\Omega}^n:\mathbb{H}_{\Omega}\to \mathbb{H}_{\Omega}$ is a self-adjoint Hilbert-Schmidt operator. As a consequence of the spectral theorem, the eigenvalues form a countable family of real numbers and the largest eigenvalue denoted in what follows by  $\lambda_n(\Omega)$ will play a central role in our study.  Actually, we show that it is simple and the components of any  associated nonzero 
 eigenfunctions $(h_1,h_2)$  are with  constant and  opposite signs. We emphasize that this property  can be formally interpreted in the framework of  the classical  Krein-Rutman Theorem (KRT). In fact,  we can show that the convex cone $\mathscr{C}^+:=\big\{(h_1,h_2)\in \mathbb{H}_{\Omega} , h_1\geqslant0, h_2\leqslant 0\big\}$ is stable under the action of $\mathcal{T}_{i,\Omega}^n$, that is, $\mathcal{T}_{i,\Omega}^n(\mathscr{C}^+)\subset \mathscr{C}^+$, which follows from the positivity of all   the kernels in \eqref{K-kernel-intro}. However this approach fails with the function space $\mathbb{H}_{\Omega}$ because the cone  $\mathscr{C}^+$ is not solid due to its empty  interior,  and this is an essential property in (KRT). To remedy to this defect a possibility is to strengthen the regularity of  the function spaces and work for instance in the H\"older class $\mathscr{C}^{1,\alpha}(0,\pi)$ subject to the Dirichlet boundary. This approach could work provided that some technical issues are fixed.  However and instead of  this implicit approach,  we shall   check the aforementioned spectral structure for the largest eigenvalue  through direct and more constructive arguments, see Proposition \ref{prop-operatorV2}.  It is worthy to point out that we are also able to track some aspects on the dynamics of  the largest eigenvalue with respect o $n$ and $\Omega$. We get in particular   through some  monotonicity properties that     for large $n$, there exists   a unique $\Omega_n$ such that $\lambda_n(\Omega_n)=1$. As a byproduct we  deduce that the kernel for $\mathcal{T}_{i,\Omega_n}^n$ is one-dimensional and the monotonicity of the sequence $\{\Omega_n, n\geqslant n_0\}$ allows in turn to get that the complete linearized operator restricted to the  $m$ fold symmetry  is also one-dimensional. We observe that the sequence $\{\Omega_n, n\geqslant n_0\}$ is strictly increasing and converges to $\overline\Omega_1$ introduced in \eqref{assump-omega-intro}.

Finally, let us mention that all the spectral analysis is performed in the weaker space $L^2_{\mu_\Omega}$ but the nonlinear functional $\tilde{F}$ will be defined in a more regular H\"older spaces, and then we will need to perform a bootstrap argument to prove that the eigenfunctions (the elements of the kernel) are in fact enough smooth  and  satisfy the  Dirichlet boundary conditions, for more details see Section \ref{Reg-EigenF}. As to the transversality assumption required in Crandall-Rabinowitz theorem, it will be analyzed in Proposition \ref{prop-transversal}. It is more delicate to establish compared to the case of one interface raised in \cite{GHM} because the quadratic form induced by the operator $\mathcal{T}_{i,\Omega_n}^n$ is not elliptic but hyperbolic. This issue is fixed using the specific structure of the normalized  eigenfunction denoted by $h^\star_n=(h_{1,n}^\star,h_{2,n}^\star)$ by showing that for large $n$ the mass is almost concentrated on the first component $h_{1,n}^\star$.   Concerning the regularity of the nonlinear functional $\tilde{F}$, it will be studied in Section \ref{sec-regularity} through some arguments relevant to potential theory applied to operators with kernels deformed via  spherical type coordinates. This allows to achieve the proof of the general statement in Theorem \ref{theorem}.

{The paper is organized as follows. In Section \ref{Sec-contour} we give the derivation of the functional $\tilde{F}$ together with a discussion about the assumptions {\bf (H)} and the definition of the function spaces. Later in Section \ref{Linearized operator}, we shall explore the expression of the linearized operator around the trivial solution  in terms of  Gauss Hypergeometric functions. Section \ref{sec-spectral} is the main important one, which aims to extract enough properties on the  spectral problem associated to the operator   $\mathcal{T}_{i,\Omega}^n$. We shall in particular  perform  a rich  study on  the  largest eigenvalue and detail its behavior and asymptotic with respect to the bifurcation parameter $\Omega$ and the mode $n$. The regularity of the associated eigenfunctions will also be carefully analyzed through a  bootstrap argument. The Fredholm structure of the linearized operator and the transversality condition will be  studied in this section.  As to Section \ref{sec-regularity}, it will be devoted to the regularity aspects of the nonlinear functional  $\tilde{F}$. In Section \ref{result}  we state  the  main result { of} the current paper together with a complete proof and  a discussion around  the validity of the spectral assumption \eqref{assump-omega-intro}. Finally, Appendices \ref{Ap-spfunctions} and \ref{Ap-bif} provide useful properties about the Gauss Hypergeometric function and the statement of the Crandall-Rabinowitz theorem.}

\section{Contour dynamics equations}\label{Sec-contour}
We shall focus on the derivation of special time periodic solutions  to \eqref{equation-model}. Let us consider 
 an initial data $ q_0={\bf{1}}_{D}$ with $D$ being a bounded  domain   of $\R^3$, called a  patch. Then this structure is conserved by the motion and one gets for any time $t\geq0$
\begin{align}\label{VP}
q(t,x)={\bf{1}}_{D(t)}(x),
\end{align}
for some bounded domain $D(t)$.  This domain is propelled  by its own  induced velocity  and the boundary is subject to   the contour dynamics equation formulated by Deem and Zabusky for Euler equations in two dimensions \cite{DeemZabusky}.  For the current model, we refer to \cite{GHM} for the equation derivation when the boundary of $\partial D(t)$ is given by only one smooth connected surface. The same approach works  for domains with finitely many connected components of the boundary.  Here, we shall only discuss the case of two interfaces corresponding to $D_t=D_{1,t}\backslash D_{2,t}$ with the strict embedding $D_{2,t}\Subset D_{1,t}$ and the boundary of each piece $D_{j,t}$ is a connected smooth  surface that can be parametrized    by  a smooth function $\gamma_{j,t}: (\phi,\theta)\in\T^2\mapsto \gamma_{j,t}(\phi,\theta)\in \R^3$, with $\T$ the one--dimensional torus. Since the boundary is transported by the flow, then we end up with 
\begin{equation}\label{Gene-eq}
\big(\partial_t\gamma_{j,t}-(U,0)(t,\gamma_{j,t})\big)\cdot n(\gamma_{j,t})=0,
\end{equation}
where  $n(\gamma_{j,t})$ is a normal vector to the boundary at the point $\gamma_{j,t}$ and $(U,0)=(u,v,0)$ is the velocity field generated by the patch { $q_t=1_{D(t)}$} through Biot-Savart law 
\begin{align}\label{BioS-Form} 
\nonumber (U,0)(t,x\big)&=\frac{1}{4\pi}\int_{D_{1,t}}\frac{(x-y)^\perp}{|x-y|^3}dy-\frac{1}{4\pi}\int_{D_{2,t}}\frac{(x-y)^\perp}{|x-y|^3}dy\\
&=\frac{1}{4\pi}\int_{\partial D_{1,t}}\frac{n^\perp(y)}{|x-y|}d\sigma(y)-\frac{1}{4\pi}\int_{\partial D_{2,t}}\frac{n^\perp(y)}{|x-y|}d\sigma(y),
\end{align}
where we use the notation  $x^\perp=(-x_2,x_1,0)\in\R^3$ for $x=(x_1,x_2,x_3)\in\R^3$ and  $d\sigma$ denotes the Lebesgue  surface measure of $\partial D_{j,t}$. We want to emphasize  one of the main feature of the three dimensions related to  the structure of  stationary solutions. In fact,  this class   is  very rich compared to the two dimensional case as it was stated in \cite{GHM}. More precisely, we established the following result. 
\begin{lem}[\cite{GHM}]\label{Prop-stat}
Let $D$ be a revolution shape smooth domain, that is, $D$ is invariant by any  rotation around the vertical $x_3$-axis, then $q(t)={\bf{1}}_{D}$ defines a stationary solution for \eqref{equation-model}.
\end{lem}

\subsection{Rigid  time-periodic patches}
In what follows we intend to describe the equations for rotating patches which are rigid time periodic solutions where the the initial patch is rotating uniformly around the vertical  $x_3-$axis  that will be its axis of symmetry.
We shall in particular derive two equivalent equations   using the velocity field and the stream function.

Assume that the solution  is given by  a  patch $q(t)={\bf{1}}_{D_t}$ rotating uniformly  about the $x_3-$axis with a constant angular velocity $\Omega\in\R.$ This means that   $D_t=\mathcal{R}_{\Omega t}D$, with $\mathcal{R}_{\Omega t}$ being the rotation of angle $\Omega t$ around the vertical axis. Then inserting this ansatz   into the equation \eqref{Gene-eq} yields to the stationary equation governing the domain $D$
$$
\big((U,0)(x)-\Omega x^\perp\big)\cdot \vec{n}(x)=0, \quad \forall\, x\in \partial D.
$$
This equation 
means also that all the  horizontal sections $D_{x_3}:=\{ y\in\R^2,\, (y,x_3)\in D\}$ rotate with the same  angular velocity $\Omega$. Therefore these sections satisfy the constraint
\begin{equation}\label{Eq-sect1}
\big(U(x)-\Omega x^\perp\big)\cdot \vec{n}_{D_{x_3}}(x_h)=0, \quad x_h=(x_1,x_2)\in \partial D_{x_3}, \ x_3\in\R,
\end{equation}
where $\vec{n}_{D_{x_3}}$ denotes a normal vector to the planar curve $\partial D_{x_3}$. Next we shall specify these equations in the particular case of doubly connected domains seen before along the derivation of \eqref{BioS-Form}. In this case, we shall impose   $D=D_1\backslash D_2$ with the strict embedding $D_2\Subset D_1$ and the boundary of each piece $D_j$ is a connected smooth  surface that can be parametrized as follows \begin{align}\label{param}
\partial D_j=&\left\{\gamma_j(\phi,\theta):=\big(r_j(\phi,\theta) e^{i\theta},d_j\cos(\phi)\big);\,  0\leqslant \theta\leqslant 2\pi, 0\leqslant \phi\leq \pi\right\},
\end{align}
where $r_j:[0,\pi]\times\mathbb{T}\rightarrow [0,\infty)$ satisfies the Dirichlet boundary condition 
$$r_j(0,\theta)=r_j(\pi,\theta)=0
$$ and periodic in the variable $\theta$. The real numbers   $d_1$ and $d_2$ are fixed ans satisfies $d_1>d_2$. We assume that the two surfaces don't intersect, which is guaranteed if and only if
$$
\inf_{\phi,\varphi\atop \theta,\eta}\big|r_1(\phi,\theta)e^{i\theta}-r_2(\varphi,\eta)e^{i\eta}|^2+\big|d_1\cos(\phi)-d_2\cos(\varphi)\big|^2>0.
$$ 
This condition is satisfied provided that
 \begin{align*}
 \forall \phi,\varphi\in [0,\pi],\,\forall\,\theta,\eta\in\mathbb{T},\quad (r_1(\phi,\theta)-r_2(\varphi,\eta))^2+(d_1\cos(\phi)-d_2\cos(\varphi))^2\geqslant \delta,
 \end{align*}
 for some $\delta>0.$  This is similar to the assumption {(\bf{H4})} imposed to the initial profiles $r_{0,j},$ where the lower bound constants are not necessary the  same.
Notice that the horizontal sections ${D_{x_3}}$ are delimited by two closed curves indexed by $\phi$  and a normal vector is given by 
\begin{align*}
\vec{n}(r_j(\phi,\theta)e^{i\theta})&=i\partial_\theta\gamma_j(\phi,\theta)\\
&=\big( i\partial_{\theta}r_j(\phi,\theta)-r_j(\phi,\theta)\big)e^{i\theta}.
\end{align*}
Therefore, the equations  \eqref{Eq-sect1} reduce to the following system
\begin{align}\label{vel-form-ini}
\textnormal{Re}\left[\left(U\big(\gamma_j(\phi,\theta)\big)-i\Omega r_j(\phi,\theta)e^{i\theta}\right)\big(i\partial_{\theta}r_j(\phi,\theta)+r_j(\phi,\theta)\big)e^{-i\theta}\right]=0,
\end{align}
for $j=1,2$ and $(\phi,\theta)\in[0,\pi]\times\T$.
Next, we intend to establish an  explicit form of $U$ using the parametrization $\gamma_j$. Indeed, using the second identity in \eqref{BioS-Form} we find
\begin{align}\label{U}
U\big(\gamma_1(\phi,\theta)\big)=&\frac{1}{4\pi}\bigintsss_{-d_1}^{d_1}\bigintsss_{\partial D_{1,y_3}}{\frac{n^\perp(y_h)dy_hdy_3}{|\gamma_1(\phi,\theta)-y|}}-\frac{1}{4\pi}\bigintsss_{-d_2}^{d_2}\bigintsss_{\partial D_{2,y_3}}{\frac{n^\perp(y_h)dy_hdy_3}{|\gamma_1(\phi,\theta)-y|}}\\
=&\frac{d_1}{4\pi}\bigintsss_{0}^{\pi}\bigintsss_0^{2\pi}\frac{\sin(\varphi)\big(\partial_{\eta}r_1(\varphi,\eta)+ir_1(\varphi,\eta)\big)e^{i\eta}}{|\gamma_1(\phi,\theta)-\gamma_1(\varphi,\eta)|}d\eta d\varphi\nonumber\\
&-\frac{d_2}{4\pi}\bigintsss_{0}^{\pi}\bigintsss_0^{2\pi}\frac{\sin(\varphi)\big(\partial_{\eta}r_2(\varphi,\eta)+ir_2(\varphi,\eta)\big)e^{i\eta}}{|\gamma_1(\phi,\theta)-\gamma_2(\varphi,\eta)|}d\eta d\varphi,\nonumber
\end{align}
and
\begin{align*}
U\big(\gamma_2(\phi,\theta)\big)
=&\frac{d_1}{4\pi}\bigintsss_{0}^{\pi}\bigintsss_0^{2\pi}\frac{\sin(\varphi)\big(\partial_{\eta}r_1(\varphi,\eta)+ir_1(\varphi,\eta)\big)e^{i\eta}}{|\gamma_2(\phi,\theta)-\gamma_1(\varphi,\eta)|}d\eta d\varphi\nonumber\\
&-\frac{d_2}{4\pi}\bigintsss_{0}^{\pi}\bigintsss_0^{2\pi}\frac{\sin(\varphi)\big(\partial_{\eta}r_2(\varphi,\eta)+ir_2(\varphi,\eta)\big)e^{i\eta}}{|\gamma_2(\phi,\theta)-\gamma_2(\varphi,\eta)|}d\eta d\varphi.\nonumber
\end{align*}
Consider the functionals   $\mathcal{I}_1$ and $\mathcal{I}_2$  
\begin{align}
\mathcal{I}_1(R,z)=&\frac{d_1}{4\pi}\bigintsss_{0}^{\pi}\bigintsss_0^{2\pi}\frac{\sin(\varphi)\big(\partial_{\eta}r_1(\varphi,\eta)+ir_1(\varphi,\eta)\big)e^{i\eta}}{|(Re^{i\theta},z)-\gamma_1(\varphi,\eta)|}d\eta d\varphi,\label{I1}\\
\mathcal{I}_2(R,z)=&\frac{d_2}{4\pi}\bigintsss_{0}^{\pi}\bigintsss_0^{2\pi}\frac{\sin(\varphi)\big(\partial_{\eta}r_2(\varphi,\eta)e+ir_2(\varphi,\eta)\big)e^{i\eta}}{|(Re^{i\theta},z)-\gamma_2(\varphi,\eta)|}d\eta d\varphi.\label{I2}
\end{align}
Then 
$$
U\big(\gamma_j(\phi,\theta)\big)=\mathcal{I}_1\big(\gamma_j(\phi,\theta)\big)-\mathcal{I}_2\big(\gamma_j(\phi,\theta)\big).
$$
Consequently, the system \eqref{vel-form-ini} becomes
\begin{align}\label{vel-form}
\textnormal{Re}\left[\left(\mathcal{I}_1\big(\gamma_j(\phi,\theta)\big)-\mathcal{I}_2\big(\gamma_j(\phi,\theta)\big)-i\Omega r_j(\phi,\theta)e^{i\theta}\right)\big(i\partial_{\theta}r_j(\phi,\theta)+r_j(\phi,\theta)\big)e^{-i\theta}\right]=0,
\end{align}
for $j=1,2$ and $(\phi,\theta)\in[0,\pi]\times\T$. Next, we shall write the equations for   rotating solutions close to a doubly connected  stationary solution  described in Lemma \ref{Prop-stat}. More precisely, we are looking for a parametrization in the form
\begin{align}\label{f}
r_j(\phi,\theta)=&r_{0,j}(\phi)+f_j(\phi,\theta)\quad\hbox{with} \quad f_j(\phi,\theta)=\sum_{n\geqslant 1}f_{j,n}(\phi)\cos(n\theta),\,j=1,2.
\end{align}
Remark that from the structure of the perturbation, we are implicitly assuming  that the domain $D_j$ is at least one-fold in the horizontal variable, that is,  symmetric with respect to the vertical plane $x_2=0$.
In addition, we shall work with the Dirichlet  boundary conditions,
\begin{align*}
r_{0,j}(0)=r_{0,j}(\pi)=f_j(0,\theta)=f_j(\pi,\theta)=0,\, j=1,2.
\end{align*}
Define the functional $F^{\bf v}=(F_1^{\bf v}, F_2^{\bf v})$ with
\begin{align*}
F_i^{\bf v}(\Omega,f_1,f_2)(\phi,\theta)=&\textnormal{Re}\left[\left(e^{-i\theta}\big[(\mathcal{I}_1-\mathcal{I}_2)\big(\gamma_i(\phi,\theta)\big)\big]-i\Omega r_i(\phi,\theta)\right)\big(i\partial_{\theta}r_i(\phi,\theta)+r_i(\phi,\theta)\big)\right].
\end{align*}
The superscript $\bf{v}$ refers to the velocity formulation and we use it to compare it later to the stream function formulation. Hence, the rotating doubly connected patches are solutions to 
\begin{equation}\label{Formu-v1}
F^{\bf v}(\Omega,f_1,f_2)(\phi,\theta)=0, \quad (\phi,\theta)\in[0,\pi]\times\T.
\end{equation}
We note that  Lemma \ref{Prop-stat} implies that $F^{\bf v}(\Omega,0,0)= 0$, for any $\Omega\in\R$.

The next task is to formulate the equations of doubly connected rotating patches in terms of  the stream function defined by \eqref{Stream-F}. First, 
 we can check that \eqref{vel-form} agrees with
\begin{align*}
 \partial_\theta \left\{\psi\big(r_j(\phi,\theta)e^{i\theta},d_j\cos(\phi)\big)-\tfrac{\Omega}{2}|r_j(\phi,\theta)|^2\right\}=0,\, j=1,2.
\end{align*}
Integrating in $\theta$ yields to the equations
\begin{align*}
\psi\big(r_j(\phi,\theta)e^{i\theta},d_j\cos(\phi)\big)-\tfrac{\Omega}{2}|r_j(\phi,\theta)|^2=&m_j(\Omega,f_1,f_2)(\phi),
\end{align*}
where $m_j(\Omega,f_1,f_2)$ is the average in $\theta$ of the  functionals and  given by  
\begin{align}\label{meanT}
m_j(\Omega,f_1,f_2)(\phi)=\frac{1}{2\pi}\int_0^{2\pi}\left\{\psi\big(\gamma_j(\phi,\theta)\big)-\tfrac{\Omega}{2}r^2_j(\phi,\theta)\right\}d\theta,\, j=1,2.
\end{align}
 Let us define the functional ${F}^{\bf  s}=({F}_1^{\bf  s},{F}_2^{\bf  s})$ through
\begin{align}\label{nonlinearfunction2}
{F}^{\bf  s}_j(\Omega,f_1,f_2)(\phi,\theta)=&\psi\big(r_j(\phi,\theta)e^{i\theta},d_j\cos(\phi)\big)-\tfrac{\Omega}{2}r^2_j(\phi,\theta)-m_j(\Omega,f_1,f_2)(\phi),
\end{align}
where the stream function takes in view of \eqref{Stream-F} the form 
\begin{align*}
\psi(Re^{i\theta},z)=&-\frac{d_1}{4\pi}\bigintsss_{0}^{\pi}\bigintsss_0^{2\pi}\bigintsss_0^{r_1(\varphi,\eta)}\frac{\sin(\varphi)rdrd\eta d\varphi}{|(re^{i\eta},d_1\cos(\varphi))-(Re^{i\theta},z)|}\\
&+\frac{d_2}{4\pi}\bigintsss_{0}^{\pi}\bigintsss_0^{2\pi}\bigintsss_0^{r_2(\varphi,\eta)}\frac{\sin(\varphi)rdrd\eta d\varphi}{|(re^{i\eta},d_2\cos(\varphi))-(Re^{i\theta},z)|}\cdot
\end{align*}
Thus the equation \eqref{Formu-v1} takes the form
\begin{equation}\label{Formu-st1}
{F}^{\bf  s}(\Omega,f_1,f_2)(\phi,\theta)=0, \quad \forall (\phi,\theta)\in[0,\pi]\times\T.
\end{equation}
Motivated by the Section \ref{Linearized operator} on the structure of the linearized operator, we find it convenient  to filter the singularities of the poles at the nonlinear level by working with   the modified functional 
\begin{align}\label{Ftilde}
\tilde{F}(\Omega,f_1,f_2)(\phi,\theta):=&\big(\tilde{F}_1(\Omega,f_1,f_2)(\phi,\theta),\tilde{F}_2(\Omega,f_1,f_2)(\phi,\theta)\big)\\
=&\left(\frac{F^{\bf  s}_1(\Omega,f_1,f_2)(\phi,\theta)}{r_{0,1}(\phi)}, \frac{F^{\bf  s}_2(\Omega,f_1,f_2)(\phi,\theta)}{r_{0,2}(\phi)} \right).\nonumber
\end{align}

\subsection{Profiles assumptions}\label{Pro-assu}
We intend to list the required assumptions on the profiles { $r_{0,1},r_{0,2}$} of stationary solutions   used in the ansatz \eqref{f}. For the sake of simplicity, we shall   throughout this paper refer to them by  {\bf (H)} :
{\begin{itemize}

\item[{\bf(H1)}] {\it{(Regularity assumption)}} For $j=1,2$,  ${ r_{0,j}}\in \mathscr{C}^{2}([0,\pi])$.

\item[{\bf(H2)}] 

{\it{(Non-self intersection)}} There exists $C>0$ such that
$$
\forall\,\phi\in[0,\pi],\quad  C^{-1}\sin\phi\leq { r_{0,j}}(\phi)\leq C\sin(\phi).
$$
\item[{\bf(H3)}] {\it{(Equatorial symmetry)}} ${ r_{0,j}}$ is symmetric with respect to $\phi=\frac{\pi}{2}$, i.e., 
$${ r_{0,j}}\left({\pi}-\phi\right)={r_{0,j}}\left(\phi\right),\,\forall\, \phi\in[0,{\pi}].
$$
\item[{\bf(H4)}] {\it (Interfaces separation)} There exists $\delta>0$ such that  
\begin{align*}
 \forall \phi,\varphi\in [0,\pi],\quad { (r_{0,1}(\phi)-r_{0,2}(\varphi))^2}+(d_1\cos(\phi)-d_2\cos(\varphi))^2\geqslant \delta.
 \end{align*}
  
\end{itemize}}
The assumptions {\bf(H2)}-{\bf(H4)}  correspond to the non-degeneracy of the geometric structure of the domain $D$. Actually, {\bf(H2)} implies in part  that the profile does not vanish inside $(0,\pi)$ and this means that each surface of $\partial D$ does not self intersect. The simple zero at $\phi=0,\pi$   is needed  to circumvent  shape singularities  of $D$ at the poles.  However,  {\bf(H4)} is needed to avoid the intersection between the two surfaces which should be  well-separated. We point out that without these assumptions, lot of results are expected to fail and we do not know whether or not our main result on the existence of periodic solutions persists.  The assumption {\bf(H3)} turns out to be a commodity  to restrict the analysis to $\phi\in[0,\tfrac\pi2]$ and can be relaxed. \\
Next we shall discuss  some consequences of {\bf (H)}  needed for later purposes.\\

\ding{70}  From {\bf(H1)}-{\bf(H2)} we have that $r_{0,i}'(0)>0$ { and  $r_{0,i}'(\pi)<0.$} Then, by continuity of the derivative  there exists $\delta>0$ such that {$|r_{0,i}'(\phi)>0|$ for $\phi\in[0,\delta]\cup[\pi-\delta,\pi]$.} 
Combining this with the mean value theorem, we deduce the arc-chord estimate: there exists $C>0$ such that for any $ \phi,\varphi\in[0,\pi], j=1,2$
\begin{equation}\label{Chord}
 C^{-1}(\phi-\varphi)^2\leqslant ({ r_{0,j}(\varphi)-r_{0,j}(\phi))^2}+d_j^2(\cos(\phi)-\cos(\varphi))^2\leqslant C(\phi-\varphi)^2.
\end{equation}

\ding{70} By {\bf(H1)}-{\bf(H2)} we deduce  that $\frac{\sin(\cdot)}{{r_{0,j}(\cdot)}}\in \mathscr{C}^1([0,\pi])$, and then { $\phi\in[0,\frac{\pi}{2}]\mapsto\frac{\phi}{{r_{0,j}}(\phi)}$
$\in \mathscr{C}^1([0,\frac{\pi}{2}])$}.

\ding{70}  For  $i,j\in\{1,2\}$ we define  \begin{align}\label{Rij}
R_{i,j}(\phi,\varphi)=&({ r_{0,i}(\phi)+r_{0,j}(\varphi)})^2+(d_i\cos(\phi)-d_j\cos(\varphi))^2.
 \end{align}
Then, from {\bf(H2)} and {\bf(H4)} one may find $ \overline\delta_{i,j}\in(0,1]$ such that
\begin{equation}\label{Chord-1}
\forall \phi,\varphi\in[0,\pi],\quad 0\leqslant \frac{4 {r_{0,i}(\phi) r_{0,j}(\varphi)}}{R_{i,j}(\phi,\varphi)}\leqslant \overline\delta_{i,j}
\end{equation}
with $$
\overline\delta_{i,i}=1 \quad \hbox{and  for }i\neq j,\quad  \overline\delta_{i,j}<1.
$$ 
Indeed, this is clear for $i=j$ since
\begin{align*}
 \frac{4{ r_{0,i}(\phi) r_{0,i}(\varphi)}}{R_{i,i}(\phi,\varphi)}&\leqslant{ \frac{4 r_{0,i}(\phi) r_{0,i}(\varphi)}{(r_{0,i}(\phi)+ r_{0,i}(\varphi))^2}}\leqslant 1.
\end{align*}
However, for $i\neq j$ we use {\bf(H4)} giving
\begin{align*}
 {R_{i,j}(\phi,\varphi)}&={{4  r_{0,i}(\phi)} { r_{0,j}(\varphi)}}+{({r_{0,i}(\phi)}-  {r_{0,j}(\varphi)})^2}+(d_i\cos\phi-d_j\cos\varphi)^2\\
 &\geqslant {4{ r_{0,i}(\phi) r_{0,j}(\varphi)}}+\delta.
\end{align*}
Consequently, from {\bf(H2)} combined with the monotonicity of $x\mapsto \frac{x}{x+\delta}$ we get
\begin{align*}
 \frac{4 {r_{0,i}(\phi) r_{0,j}(\varphi)}}{R_{i,j}(\phi,\varphi)}&\leqslant{\frac{4 r_{0,i}(\phi) r_{0,j}(\varphi)}{{4 r_{0,i}(\phi) r_j(\varphi)}+\delta}
 \leqslant  \tfrac{C^2}{C^2+\delta}:= \overline\delta_{i,j} <1.}
  \end{align*}

\subsection{Function spaces} 
The main purpose of this section  is to introduce the function spaces that will be used along this paper together with some of their basic properties.  
First, we recall the classical  H\"{o}lder spaces defined on an open nonempty  set $\mathscr{O}\subset \R^d$ as follows. Given  $\alpha\in (0,1)$ then 
$$
\mathscr{C}^{1,\alpha}(\mathscr{O})=\Big\{ f:\mathscr{O}\mapsto\R, \|f\|_{\mathscr{C}^{1,\alpha}}<\infty\Big\},
$$ 
with
$$
 \|f\|_{\mathscr{C}^{1,\alpha}}=\|f\|_{L^\infty}+\|\nabla f\|_{L^\infty}+\sup_{x\neq y\in \mathscr{O}}\frac{|\nabla f(x)-\nabla f(y)|}{|x-y|^\alpha}\cdot
$$
The space  $\mathscr{C}^{1,\alpha}(\mathscr{O})$ is a Banach  algebra, with 
$$
\|fg\|_{\mathscr{C}^{1,\alpha}}\leq C \|f\|_{\mathscr{C}^{1,\alpha}}\|g\|_{\mathscr{C}^{1,\alpha}}.
$$ 
On the other hand,  we may   identify  in the periodic setting the space $\mathscr{C}^{1,\alpha}(\T)$ with the space $\mathscr{C}_{2\pi}^{1,\alpha}(\R)$ of  $2\pi$--periodic functions that belongs to $\mathscr{C}^{1,\alpha}(\R)$. \\
Now, we shall introduce  the function spaces  used  in a crucial way in the regularity  study of the functional $\tilde{F}$ defined in \eqref{Ftilde}. For $\alpha\in(0,1)$ and $m\in\N^\star,$ set
\begin{equation}\label{space}
X_m^\alpha=\Big\{f\in \mathscr{C}^{1,\alpha}\big((0,\pi)\times\T\big) 
 , \, f\left(\phi,\theta\right)=\sum_{n\geqslant 1}f_n(\phi)\cos(nm\theta)\Big\},
\end{equation}
supplemented with the boundary and symmetry  conditions
\begin{equation}\label{spaceXL1}
\forall (\phi,\theta)\in[0,\pi]\times\T,\quad  f(0,\theta)=f(\pi,\theta)= 0\quad\hbox{and} \quad f\left(\pi-\phi,\theta\right)=f\left(\phi,\theta\right).
\end{equation}
This space is equipped with the same norm as $\mathscr{C}^{1,\alpha}((0,\pi)\times{\T}).$ 
Remark  that any element  $f\in\mathscr{C}^{1,\alpha}\big((0,\pi)\times\T\big) $ admits a continuous  extension up to the boundary and therefore  the assumptions  \eqref{spaceXL1} are  meaningful.  On the other hand, the Dirichlet boundary conditions  together with  Taylor's formula allow  to get a constant $C>0$ such that for   any $f\in X_m^\alpha$
\begin{align}\label{platitL1}
|f(\phi,\theta)|\leqslant& C\|f\|_{\textnormal{Lip}}\sin \phi,\nonumber\\
\partial_\theta f(0,\theta)=\partial_\theta f(\pi,\theta)=0&\quad\hbox{and}\quad |\partial_\theta f(\phi,\theta)|\leqslant C\|f\|_{\mathscr{C}^{1,\alpha}}\sin^\alpha(\phi).
\end{align}
The ball  in $X_m^\alpha$ centered in $0$ with radius $\varepsilon$ is denoted by $B_{X_m^\alpha}(\varepsilon)$.

\section{Linearized operator}\label{Linearized operator}
In this section, we shall describe the structure of the linearized operator around the stationary solutions. We find it convenient to  use the stream function formulation \eqref{Formu-st1}-\eqref{Ftilde}, and to alleviate the notation  we will  omit the superscript $\bf s$ from $F^{\bf s}$.  Two different useful  representations  for the linearized operator will be now discussed.

\subsection{Linearization at the equilibrium state}
The main target in this section is to detail  the structure of the  linearized operator of $\tilde{F}$ introduced in \eqref{Ftilde}--\eqref{nonlinearfunction2}  around the trivial solution $(\Omega,0,0)$. The following quantities are useful below
\begin{align}\label{Hcal}
\mathcal{H}_{i,j}^n(\phi,\varphi,\eta)=\frac{d_j}{4\pi r_{0,i}(\phi)}\frac{\sin(\varphi)r_{0,j}(\varphi)\cos(n\eta)}{\sqrt{A_{i,j}(\phi,\varphi,\eta)}},
\end{align}
and
\begin{align}\label{nuOmega}
\nu_{i,\Omega}(\phi)
=&(-1)^{i-1}\left(\int_{0}^{\pi}\int_0^{2\pi} \mathcal{H}^1_{i,1}(\phi,\varphi,\eta)d\eta d\varphi-\int_{0}^{\pi}\int_0^{2\pi}\mathcal{H}^1_{i,2}(\phi,\varphi,\eta)d\eta d\varphi   -\Omega\right),
\end{align}
with
\begin{align}
A_{i,j}(\phi,\varphi,\eta)
=&(r_{0,i}(\phi)-r_{0,j}(\varphi))^2+(d_i\cos(\phi)-d_j\cos(\varphi))^2+2{r_{0,i}(\phi)r_{0,j}(\varphi)}(1-\cos(\eta)).\label{A11}
\end{align}
\begin{pro}\label{Prop-MMlin1}
Let $j=1,2$ and consider two smooth   functions 
$$(\phi,\theta)\in[0,\pi]\times\T\mapsto h_j(\phi,\theta)=\sum_{n\geqslant 1}h_{j,n}(\phi)\cos(n\theta).
$$
Then, 
the Gateaux  derivative  of the functional $\tilde{F}$   introduced  in  \eqref{Ftilde} 
takes the form
\begin{align*}
\partial_{f_1,f_2}\tilde{F}_i(\Omega,0,0)(h_1,h_2)(\phi,\theta)=&(-1)^{i-1}\nu_{i,\Omega}(\phi)\sum_{n\geqslant 1}\cos(n\theta)\Big(h_{i,n}(\phi)-\mathcal{T}^n_{i,\Omega}(h_{1,n},h_{2,n})(\phi)\Big),
\end{align*}
with
\begin{align*}
\mathcal{T}^n_{i,\Omega}(h_{1,n},h_{2,n})(\phi)=&(-1)^{i-1}\int_0^\pi\int_0^{2\pi}\tfrac{\mathcal{H}^n_{i,1}(\phi,\varphi,\eta)}{\nu_{i,\Omega}(\phi)}h_{1,n}(\varphi)d\eta d\varphi\\
&+(-1)^{i}\int_0^\pi\int_0^{2\pi}\tfrac{\mathcal{H}^n_{i,2}(\phi,\varphi,\eta)}{\nu_{i,\Omega}(\phi)}h_{2,n}(\varphi)d\eta d\varphi,\nonumber
\end{align*}
and $\mathcal{H}^n_{i,j}$ is defined in \eqref{Hcal}.
\end{pro}
\begin{proof}
It suffices to check the following identities
\begin{align*}
\partial_{f_i} \tilde{F}_i(\Omega,0,0)h_i(\phi,\theta)=&(-1)^{i-1}\sum_{n\geqslant 1}\cos(n\theta)\left({\nu_{i,\Omega}(\phi)h_{i,n}(\phi)}-\int_{0}^{\pi}\int_0^{2\pi}\mathcal{H}_{i,i}^n(\phi,\varphi,\eta)h_{i,n}(\varphi)d\eta d\varphi\right),
\end{align*}
and for $i\neq j$ 
\begin{align*}
\partial_{f_j} \tilde{F}_i(\Omega,0,0)h_j(\phi,\theta)=&(-1)^j\sum_{n\geqslant 1}\cos(n\theta)\int_{0}^{\pi}\int_0^{2\pi}\mathcal{H}_{i,j}^n(\phi,\varphi,\eta)h_{j,n}(\varphi) d\eta d\varphi.
\end{align*}
First, we write in view of \eqref{Ftilde} 
\begin{align*}
\partial_{f_i} \tilde{F}_i(\Omega,0,0)h_i(\phi,\theta)=&\frac{1}{r_{0,i}(\phi)}\Big( h_i(\phi,\theta)(\nabla_h \psi)(r_{0,i}(\phi)e^{i\theta},d_i\cos(\phi))\cdot e^{i\theta}\\
&+(\partial_{f_i} \psi)h_i(r_{0,i}(\phi)e^{i\theta},d_i\cos(\phi))-\Omega r_{0,i}(\phi)h_i(\phi,\theta)-\partial_{f_i} m_i(\Omega,0,0)h_i(\phi)\Big),\end{align*}
and for $i\neq j$
\begin{align*}
\partial_{f_j} \tilde{F}_i(\Omega,0,0)h_j(\phi,\theta)=&\frac{1}{r_{0,i}(\phi)}\Big((\partial_{f_j} \psi)h_j(r_{0,i}(\phi)e^{i\theta},d_i\cos(\phi))-\partial_{f_j} m_i(\Omega,0,0)h_j(\phi)\Big).
\end{align*}
Remind that  the quantities { $\partial_{f_j} m_i$ } are  just the $\theta-$ average of the first part of the right-hand side  and therefore they  will not contribute on the total operator. 
We need to  compute the variations of $ \psi$ with respect to the shape of the patch,
\begin{align*}
(\partial_{f_i} \psi)h_i(r_{0,i}(\phi)e^{i\theta},d_i\cos(\phi))=&\tfrac{(-1)^id_i}{4\pi}\bigintsss_{0}^{\pi}\bigintsss_0^{2\pi}\frac{\sin(\varphi)r_{0,i}(\varphi)h_i(\varphi,\eta) d\eta d\varphi}{|(r_{0,i}(\varphi)e^{i\eta},d_i\cos(\varphi))-(r_{0,i}(\phi)e^{i\theta},d_i\cos(\phi))|}
\end{align*}
and 
\begin{align*}
(\partial_{f_j} \psi)h_j(r_{0,i}(\phi)e^{i\theta},d_i\cos(\phi))=&\tfrac{(-1)^jd_j}{4\pi}\bigintsss_{0}^{\pi}\bigintsss_0^{2\pi}\frac{\sin(\varphi)r_{0,j}(\varphi)h_j(\varphi,\eta) d\eta d\varphi}{|(r_{0,j}(\varphi)e^{i\eta},d_j\cos(\varphi))-(r_{0,i}(\phi)e^{i\theta},d_i\cos(\phi))|}\cdot\end{align*}
On the other hand,  one finds that $(\nabla_h \psi)(Re^{i\theta},z)\cdot e^{i\theta}=U(Re^{i\theta},z)\cdot (ie^{i\theta})$ (Here $\cdot $ denotes the standard inner product in $\R^2$ and $\nabla_h$ the horizontal gradient) and therefore
\begin{align*}
(\nabla_h \psi)(r_{0,i}(\phi)e^{i\theta},d_i\cos(\phi))\cdot e^{i\theta}=&-\frac{d_1}{4\pi}\bigintsss_{0}^{\pi}\bigintsss_0^{2\pi}\frac{\sin(\varphi)(\partial_{\eta}r_{0,1}(\varphi)e^{i\eta}+ir_{0,1}(\varphi)e^{i\eta})\cdot (-ie^{i\theta})}{|(r_{0,i}(\phi)e^{i\theta},d_i\cos(\phi))-(r_{0,1}(\varphi)e^{i\eta},d_1\cos(\varphi))|}d\eta d\varphi\\
&+\frac{d_2}{4\pi}\bigintsss_{0}^{\pi}\bigintsss_0^{2\pi}\frac{\sin(\varphi)(\partial_{\eta}r_{0,2}(\varphi)e^{i\eta}+ir_{0,2}(\varphi)e^{i\eta})\cdot (-ie^{i\theta})}{|(r_{0,i}(\phi)e^{i\theta},d_i\cos(\phi))-(r_{0,2}(\varphi)e^{i\eta},d_2\cos(\varphi))|}d\eta d\varphi\cdot
\end{align*}
Integration by parts combined with suitable identities lead to
\begin{align*}
(\nabla_h \psi)(r_{0,i}(\phi)e^{i\theta},d_i\cos(\phi))\cdot e^{i\theta}=&\frac{d_1}{4\pi}\bigintsss_{0}^{\pi}\bigintsss_0^{2\pi}\frac{\sin(\varphi)r_{0,1}(\varphi)\cos(\theta-\eta)}{|(r_{0,2}(\phi)e^{i\theta},d_2\cos(\phi))-(r_{0,1}(\varphi)e^{i\eta},d_1\cos(\varphi))|}d\eta d\varphi\\
&-\frac{d_2}{4\pi}\bigintsss_{0}^{\pi}\bigintsss_0^{2\pi}\frac{\sin(\varphi)r_{0,2}(\varphi)\cos(\theta-\eta)}{|(r_{0,2}(\phi)e^{i\theta},d_2\cos(\phi))-(r_{0,2}(\varphi)e^{i\eta},d_2\cos(\varphi))|}d\eta d\varphi.
\end{align*}
For more details about this formula, we refer to the proof of Proposition $3.1$ in \cite{GHM}.
Consequently, 
\begin{align*}
\partial_{f_i} \tilde{F}_i(\Omega,0,0)h_i(\phi,\theta)=&\tfrac{h_i(\phi,\theta)}{4\pi r_{0,i}(\phi)}\left({d_1}\bigintsss_{0}^{\pi}\bigintsss_0^{2\pi}\frac{\sin(\varphi)r_{0,1}(\varphi)\cos(\theta-\eta)}{|(r_{0,i}(\phi)e^{i\theta},d_i\cos(\phi))-(r_{0,1}(\varphi)e^{i\eta},d_1\cos(\varphi))|}d\eta d\varphi\right.\\
&\left.-{d_2}\bigintsss_{0}^{\pi}\bigintsss_0^{2\pi}\frac{\sin(\varphi)r_{0,2}(\varphi)\cos(\theta-\eta)}{|(r_{0,i}(\phi)e^{i\theta},d_i\cos(\phi))-(r_{0,2}(\varphi)e^{i\eta},d_2\cos(\varphi))|}d\eta d\varphi-\Omega\right)\\
&-\frac{d_i}{4\pi r_{0,i}(\phi)}\bigintsss_{0}^{\pi}\bigintsss_0^{2\pi}\frac{\sin(\varphi)r_{0,1}(\varphi)h_1(\varphi,\eta) d\eta d\varphi}{|(r_{0,1}(\varphi)e^{i\eta},d_1\cos(\varphi))-(r_{0,1}(\phi)e^{i\theta},d_1\cos(\phi))|},
\end{align*}
and for $i\neq j$
\begin{align*}
\partial_{f_j} \tilde{F}_i(\Omega,0,0)h_j(\phi,\theta)=&\frac{(-1)^jd_j}{4\pi r_{0,i}(\phi)}\bigintsss_{0}^{\pi}\bigintsss_0^{2\pi}\frac{\sin(\varphi)r_{0,j}(\varphi)h_j(\varphi,\eta) d\eta d\varphi}{|(r_{0,j}(\varphi)e^{i\eta},d_j\cos(\varphi))-(r_{0,i}(\phi)e^{i\theta},d_i\cos(\phi))|}\cdot
\end{align*}
Therefore,  to get the desired expressions for the partial derivatives  it suffices to  expand into Fourier series  the functions  $h_1$ and $h_2$ and use  the functions introduced in \eqref{Hcal} and \eqref{nuOmega}.
\end{proof}

\begin{rem}\label{rema-1}
From the expression of $\nu_{\Omega,i}$ detailed in \eqref{nuOmega} and the previous proof, we can relate this function to the  the stream function as follows
\begin{align}\label{nu-streamfunction}
\nu_{\Omega,i}(\phi)=(-1)^{i-1}\left(\frac{1}{r_{0,i}(\phi)} (\nabla_h \psi)(r_{0,i}(\phi)e^{i\theta},d_i\cos(\phi))\cdot e^{i\theta}-\Omega\right).
\end{align}
{One may check by the the rotation invariance that the right hand side is in fact independent on the angle $\theta.$ Therefore, with  $\theta=0$ we infer
\begin{align*}
\nu_{\Omega,i}(\phi)=(-1)^{i-1}\left(\frac{1}{r_{0,i}(\phi)} (\partial_1 \psi)\big(r_{0,i}(\phi),0,d_i\cos(\phi)\big)-\Omega\right).
\end{align*}
}
\end{rem}

\subsection{Representation with hypergeometric functions}
For $n\geqslant1$,  set 
$$
 F_n(x)=F\left(n+\tfrac{1}{2},n+\tfrac{1}{2};2n+1;x\right), \quad  x\in[0,1), 
$$
where the hypergeometric functions are described in the  Appendix \ref{Ap-spfunctions}. 
For $i,j=1,2$, we consider the functions
\begin{align}\label{HMM11}
H_{i,j}^n(\phi,\varphi)=&\tfrac{2^{2n-1}\left(\frac12\right)^2_{n}}{(2n)!}\frac{\sin(\varphi){r_{0,i}^{n-1}(\phi)r_{0,j}^{n+1}(\varphi)}}{\left[R_{i,j}(\phi,\varphi)\right]^{n+\frac12}} F_n\left(\tfrac{4{r_{0,i}(\phi)r_{0,j}(\varphi)}}{R_{i,j}(\phi,\varphi)}\right),
\end{align}
where $R_{i,j}$ is defined in \eqref{Rij}. The next goal is 
to  prove the following result.
\begin{lem}\label{nuOmega12}
The functions $\nu_{\Omega,i}, i=1,2$,  defined in \eqref{nuOmega} admit the following representation
\begin{align*}
\nu_{i,\Omega}(\phi)=(-1)^{i-1}\left(d_1\int_0^\pi H_{i,1}^1(\phi,\varphi)d\varphi-d_2\int_0^{\pi} H_{i,2}^1(\phi,\varphi)d\varphi-\Omega\right).
\end{align*}
In addition, the operator $\mathcal{T}^{n}_{i,\Omega}$ defined in Proposition  \ref{Prop-MMlin1} takes the form 
\begin{align}\label{Tnomega2}
\mathcal{T}^{n}_{i,\Omega} (h_1,h_2)(\phi)=&d_1(-1)^{i-1}\int_0^{\pi} \tfrac{H_{i,1}^n(\phi,\varphi)}{\nu_{i,\Omega}(\phi)}h_1(\varphi)d\varphi+d_2(-1)^{i}\int_0^\pi \tfrac{H_{i,2}^n(\phi,\varphi)}{\nu_{i,\Omega}(\phi)}h_2(\varphi)d\varphi.
\end{align}
\end{lem}
 \begin{proof}
  It suffices to use Lemma \ref{Lem-integral}, see \cite{GHM}. For the sake of clarity, let us compute one term, for instance the first term in $\nu_{1,\Omega}$:
  \begin{align*}
  \tilde{\nu}_{1,\Omega}(\phi):=&\int_0^\pi\int_0^{2\pi}\mathcal{H}_{1,1}^1(\phi,\varphi,\eta)d\eta d\varphi\\
  =&\frac{d_1}{4\pi r_{0,1}(\phi)}\int_0^\pi\int_0^{2\pi}\tfrac{\sin(\varphi)r_{0,1}(\varphi)\cos(n\eta)}{\sqrt{A_{1,1}(\phi,\varphi,\eta)}}d\eta d\varphi\\
   =&\frac{d_1}{4\pi r_{0,1}(\phi)}\int_0^\pi\int_0^{2\pi}\tfrac{\sin(\varphi)r_{0,1}(\varphi)\cos(n\eta)}{\sqrt{ r_{0,1}^2(\phi)+r_{0,1}^2(\varphi)+d_1^2(\cos\phi-\cos\varphi)^2-2r_{0,1}(\phi)r_{0,1}(\varphi)\cos(\eta) }}d\eta d\varphi.
  \end{align*}
 Set
  $$
  A:=\frac{r_{0,1}^2(\phi)+r_{0,1}^2(\varphi)+d_1^2(\cos\phi-\cos\varphi)^2}{2r_{0,1}(\phi)r_{0,1}(\varphi)},
  $$
we can write $\tilde{\nu}_{1,\Omega}$ as  
    \begin{align*}
  \tilde{\nu}_{1,\Omega}(\phi)
   =&\frac{d_1}{4\pi r_{0,1}(\phi)}\int_0^\pi\int_0^{2\pi}\tfrac{\sin(\varphi)r_{0,1}(\varphi)\cos(n\eta)}{\sqrt{2r_{0,1}(\phi)r_{0,\varphi)}}\sqrt{A-\cos(\eta) }}d\eta d\varphi.
  \end{align*}
  We observe that $A>1$ and hence we can apply Lemma \ref{Lem-integral} obtaining
    \begin{align*}
  \tilde{\nu}_{1,\Omega}(\phi)
   =&  \frac{d_1}{8 r_{0,1}(\phi)}\bigintsss_0^\pi \frac{\sin(\varphi)r_{0,1}(\varphi)}{\sqrt{2r_{0,1}(\phi)r_{0,\varphi)}} (1+A)^{\frac{\beta}{2}+n}}F\left(\frac{3}{2}, \frac32; 2n+1; \frac{2}{1+A}\right ) d\varphi.
  \end{align*}
  Now, using the expression of $A$ we find
  $$
  1+A=\frac{(r_{0,1}(\phi)+r_{0,1}(\varphi))^2+d_1^2(\cos\phi-\cos\varphi)^2}{2r_{0,1}(\phi)r_{0,1}(\varphi)},
  $$
  achieving
   \begin{align*}
  \tilde{\nu}_{1,\Omega}(\phi)=d_1\int_0^\pi H_{1,1}^1(\phi,\varphi)d\varphi,
  \end{align*}
  where the kernel is defined in \eqref{HMM11}.
 \end{proof}

Combining Lemma \ref{nuOmega12} with Proposition \ref{Prop-MMlin1} we obtain the following representation of the linearized operator.
\begin{cor}\label{cor-lin2}
Let $i=1,2$ and consider two smooth   functions 
$$(\phi,\theta)\in[0,\pi]\times\T\mapsto h_i(\phi,\theta)=\sum_{n\geqslant 1}h_{i,n}(\phi)\cos(n\theta).
$$
Then,
\begin{align*}
\partial_{f_1,f_2} \tilde{F}_i(\Omega,0,0)(h_1,h_2)(\phi,\theta)=&(-1)^{i-1}\nu_{i,\Omega}(\phi)\sum_{n\geqslant 1}\cos(n\theta)\Big(h_{i,n}(\phi)-\mathcal{T}^{n}_{i,\Omega} (h_{1,n},h_{2,n})(\phi)\Big),
\end{align*}
where $\nu_{i,\Omega}$ and $\mathcal{T}^{n}_{i,\Omega}$ are detailed in Lemma $\ref{nuOmega12}.$
\end{cor}

In what follows we shall discuss some properties of the functions $H_{i,j}^n$. Their proofs are a slight adaptation of  \cite[Lemma 3.1]{GHM} and will be briefly sketched. 
\begin{lem}\label{Lem-Hndecreasing}
For any $\varphi\neq\phi\in(0,\pi)$, the sequence
$
n\in \N^\star\mapsto H_{i,j}^n(\phi,\varphi)
$
is strictly decreasing.
Moreover, if we assume that {$r_{0,1},r_{0,2}$} satisfy {\bf{(H2)}}, then for any $0\leqslant\gamma<\beta\leqslant1 $
there exists  a  {constant $C>0$} such that  for any $ n\geqslant1$ and  $ \phi\neq \varphi\in(0,\pi)$
\begin{equation}\label{estim-asym}
|H_{i,j}^n(\phi,\varphi)|\leqslant C n^{-\gamma}\sin(\varphi) \tfrac{{ |r_{0,j}}(\varphi)|^{\frac12}}{{ |r_{0,i}(\phi)|^{\frac32}}}|\phi-\varphi|^{-{2\beta}}.
\end{equation}
\end{lem}
\begin{proof}
By virtue  of \eqref{HMM11} we may write
\begin{align*}
H_{i,j}^n(\phi,\varphi)=\tfrac{2^{2n-1}\Gamma^2\left(n+\frac12\right)}{(2n)!\pi}\frac{\sin(\varphi){| r_{0,j}(\varphi)|^{\frac12}}}{4^{n+\frac12}{ |r_{0,i}(\phi)|^{\frac32}}}x^{n+\frac12}F_n(x),
\end{align*}
with $x:=\frac{4{r_{0,i}(\phi)r_{0,j}(\varphi)}}{R_{i,j}(\phi,\varphi)}$ belongs to
$[0,1)$ provided that $\varphi\neq\phi$. Using the integral representation of hypergeometric functions \eqref{Ap-spfunctions-integ} and similarly to \cite[Lemma 3.1]{GHM} we obtain
\begin{align}\label{trax1}
\nonumber { H^n_{i,j}(\phi,\varphi)}=&\tfrac{2^{2n-1}\Gamma^2\left(n+\frac12\right)}{(2n)!\pi}\frac{\sin(\varphi){ |r_{0,j}(\varphi)|^{\frac12}}}{4^{n+\frac12}{|r_{0,i}(\phi)|^{\frac32}}}\tfrac{(2n)!}{\Gamma^2\left(n+\frac12\right)}x^{n+\frac12}\bigintsss_0^1 t^{n-\frac12}(1-t)^{n-\frac12}(1-tx)^{-n-\frac12}dt\\
=&\frac{1}{4\pi}
\frac{\sin(\varphi){|r_{0,j}(\varphi)|^{\frac12}}}
{{|r_{0,i}(\phi)|^{\frac32}}}
\mathcal{H}^n(x),
\end{align}
with the notation
$$
\mathcal{H}^n(x):= x^{n+\frac12}\int_0^1 t^{n-\frac12}(1-t)^{n-\frac12}(1-tx)^{-n-\frac12}dt.
$$
The monotonicity of    $n\mapsto \mathcal{H}^n(x)$ is established in \cite[Lemma 3.1]{GHM} implying that for any $\phi\neq\varphi$,  the sequence  $n\mapsto H_{i,j}^n(\phi,\varphi)$ is  strictly decreasing. It remains to prove the decay estimate of $H_{i,j}^n$ for large $n$.  This is also done in the proof of \cite[Lemma 3.1]{GHM} {except that  a slight typo without any consequence is discovered} there related to the term $|\phi-\varphi|^{-2\beta}$ that we shall correct here.  

Remark that according  \eqref{estim-asym}  and  \eqref{trax1} one has
$$
|H_{i,j}^n(\phi,\varphi)|\lesssim  \frac{\sin(\varphi)r_j(\varphi)^\frac12}{r_i(\phi)^\frac32} |\mathcal{H}^n(x)|.
$$
On the other hand, it is shown in the proof of \cite[Lemma 3.1]{GHM}  that for any $1\geqslant\beta>\gamma\geqslant0$
\begin{align*}
|\mathcal{H}^n(x)|\lesssim{n^{-\gamma} (1-x)^{-\beta}}\cdot
\end{align*}
According to \eqref{Rij} we get
\begin{align*}
1-x=&\frac{({ r_{0,i}(\phi)-r_{0,j}(\varphi)})^2+(d_i\cos(\phi)-d_j\cos(\varphi))^2}{({ r_{0,i}(\phi)+r_{0,j}(\varphi)})^2+(d_i\cos(\phi)-d_j\cos(\varphi))^2}\cdot
 \end{align*}
Thus, using {(\bf{H2})} we deduce that
$$
{
1-x\geqslant C \frac{(\phi-\varphi)^2}{(\phi+\varphi)^2}}\cdot
$$
Putting together the two last estimates yields
\begin{align*}
|\mathcal{H}^n(x)|
\lesssim&\frac{1}{n^\gamma |\phi-\varphi|^{2\beta}},
\end{align*}
which gives   the desired inequality and concludes the proof of Lemma \ref{Lem-Hndecreasing}.
\end{proof}

\section{Spectral study}\label{sec-spectral}
This section is devoted to the analysis of some  fundamental  spectral properties of the linearized operator seen before in Corollary \ref{cor-lin2} and   required  by  the Crandall--Rabinowitz theorem.  We reduce the spectral study to a countable family of one-dimensional operators $\hbox{Id}-\mathcal{T}^{n}_{i,\Omega}$ acting on the variable $\phi$. In the same spirit of  \cite{GHM}, we first succeed to symmetrize $\mathcal{T}^{n}_{i,\Omega}$  via   a suitable  use of a weighted Hilbert space and show  that it   is a self-adjoint Hilbert-Schmidt operator. Since the kernel study reduces to analyze the eigenvalues of $\mathcal{T}^{n}_{i,\Omega}$, we provide here some properties about the eigenvalues and the eigenfunctions. In particular, we show that  the largest eigenvalue denoted by $\lambda_n(\Omega)$ is simple  under suitable assumptions on the stationary profiles, however  the problem remains open for the lower  eigenvalues. A refined study on the distribution  of the sequence $\big\{\lambda_n(\Omega), n\geqslant1\big\}$ where $\Omega$ ranges in a bounded set  $(\overline\Omega_2,\overline{\Omega}_1)$ is needed to check the one-dimensionality of the kernel of $\hbox{Id}-\mathcal{T}^n_{i,\Omega}$. Moreover, we will also provide the Fredholm structure of the linearized concluding the one codimensionality of the range of the linear operator, needed to apply the Crandall-Rabinowiz theorem. Finally, we shall check the transversal condition at the end of this section which is more tricky and subtle compared to \cite{GHM}.

\subsection{Symmetrization and compactness}
We shall explore  and take advantage of  a hidden structure allowing to symmetrize the one-dimensional operator  $\hbox{Id}-\mathcal{T}^{n}_{i,\Omega}$ introduced in \mbox{Corollary \ref{cor-lin2}} through   a suitable Hilbert space, basically an  $L^2\times L^2$ space with respect to a special  Borel measure.  We show  in particular that it  a self--adjoint compact operator and analyze in the next sections some qualitative properties on the distribution of the largest eigenvalues.  
Let us start with  the  symmetrization  of  the operator $\hbox{Id}-\mathcal{T}^{n}_{i,\Omega}$ which requires to introduce suitable Hilbert spaces. First, recall the expression of $\nu_{i,\Omega}$  given in Lemma  \ref{nuOmega12},
\begin{align*}
\nu_{i,\Omega}(\phi)=(-1)^{i-1}\left(d_1\int_0^\pi H_{i,1}^1(\phi,\varphi)d\varphi-d_2\int_0^{\pi} H_{i,2}^1(\phi,\varphi)d\varphi-\Omega\right).
\end{align*}
These functions are continuous on the compact set $[0,\pi]$ and therefore they reach their extreme values. This allows to define  the following numbers playing a crucial role on the spectral study, 
\begin{equation}\label{Omega1}
 \overline\Omega_1:=\inf_{\phi\in[0,\pi]}d_1\int_0^\pi H_{1,1}^1(\phi,\varphi)d\varphi-d_2\int_0^{\pi} H_{1,2}^1(\phi,\varphi)d\varphi
\end{equation}
and
\begin{equation}\label{Omega2}
\overline\Omega_2:=\sup_{\phi\in[0,\pi]}d_1\int_0^\pi H_{2,1}^1(\phi,\varphi)d\varphi-d_2\int_0^\pi H_{2,2}^1(\phi,\varphi)d\varphi.
\end{equation}
We shall make the following important assumption 
\begin{align}\label{assump1} 
\overline\Omega_2< \overline\Omega_1.
\end{align}
\begin{rem}
The assumption \eqref{assump1} will be a technical hypothesis in our main theorem. Indeed, by virtue of  \eqref{nu-streamfunction}, it can be expressed in terms of the stream function or the velocity field as follows
\begin{align}
\overline{\Omega}_1=&\inf_{\phi\in[0,{ \pi}]} \frac{1}{r_{0,1}(\phi)}(\nabla_h \psi)(r_{0,1}(\phi)e^{i\theta},d_1\cos(\phi))\cdot e^{i\theta}\nonumber\\
=&\inf_{\phi\in[0,{\pi}]} \frac{1}{r_{0,1}(\phi)}(U(r_{0,1}(\phi)e^{i\theta},d_1\cos(\phi))\cdot (ie^{i\theta})\label{Omega1-2}
\end{align}
and
\begin{align}
\overline{\Omega}_2=&\sup_{\phi\in[0,{ \pi}]} \frac{1}{r_{0,2}(\phi)}(\nabla_h \psi)(r_{0,2}(\phi)e^{i\theta},d_2\cos(\phi))\cdot e^{i\theta}\nonumber\\
=&\sup_{\phi\in[0,{\pi}]} \frac{1}{r_{0,2}(\phi)}(U(r_{0,2}(\phi)e^{i\theta},d_2\cos(\phi))\cdot (ie^{i\theta}).\label{Omega2-2}
\end{align}
Note that $\overline{\Omega}_1$ and $\overline{\Omega}_2$ do not depend on $\theta$ due to their symmetry about the vertical axis, that can be directly checked from the expressions \eqref{Omega1}-\eqref{Omega2}.

At the  first sight, it is not clear whether this condition is generic and satisfied with  the stationary profiles  subject to  the assumptions $({\bf{H}})$ of Section $\ref{Pro-assu}.$ However, we shall show in \mbox{Proposition \ref{prop-ellipsoid}} the validity of this assumption with specific elementary shapes combining sphere  and ellipsoid and by default any of their small perturbation, since \eqref{assump1}  is stable by perturbation. Moreover, we can also check \eqref{assump1} when the shapes are very well-separated, see Proposition $\ref{prop-d1}.$
\end{rem}

For $j=1,2$ and $\Omega\in[\overline\Omega_2,\overline\Omega_1],$ we consider  the Borel  measures
\begin{align}\label{MeasureB-1}
d\mu_j(\varphi)=\sin(\varphi){ r_{0,j}^2}(\varphi)\nu_{j,\Omega}(\varphi) d\varphi.
\end{align}
Define the Hilbert space $L^2_{\mu_j}$  as the set of measurable function $f:(0,\pi)\to\R$ such that
$$
\|f\|_{\mu_j}^2=\int_0^\pi|f(\varphi)|^2d\mu_j(\varphi)<\infty,
$$
which is equipped with the standard inner product,
\begin{equation}\label{scalar-prod1}
\langle f,g\rangle_{\mu_j}=\int_0^\pi f(\varphi) g(\varphi) d\mu_j(\varphi)
\end{equation}
that induces the norm
$$
\|f\|^2_{\mu_j}=\int_0^\pi |f(\varphi)|^2d\mu_j(\varphi).
$$
Now, consider the vectorial Hilbert space $\mathbb{H}_{\Omega}:=L^2_{\mu_1}\times L^2_{\mu_2}$ endowed  with the inner product
\begin{align}\label{inner-P}
\nonumber H=(h_1,h_2), \widehat{H}=(\widehat{h}_1,\widehat{h}_2),\quad \langle H,\widehat{H}\rangle_{\Omega}&:=\langle h_1,\widehat{h}_1\rangle_{\mu_1}+{\tfrac{d_2}{d_1}}\langle h_2,\widehat{h}_2\rangle_{\mu_2}\\
&=\int_0^\pi h_1(\varphi)\widehat{h}_1(\varphi)d\mu_1(\varphi)+{\tfrac{d_2}{d_1}}\int_0^\pi h_2(\varphi)\widehat{h}_2(\varphi)d\mu_2(\varphi)
\end{align}
and  the induced norm
\begin{equation}\label{Norm-Hil1}
\|H\|^2:=\|h_1\|^2_{\mu_1}+\tfrac{d_2}{d_1}\|h_2\|^2_{\mu_2}.
\end{equation}
Let us recall the vectorial linear operator $\mathcal{T}^n_{\Omega}=(\mathcal{T}^{n}_{1,\Omega},\mathcal{T}^{n}_{2,\Omega})$ whose components are described  \mbox{by \eqref{Tnomega2}.}
Next, denote by $K_{i,j}^n$    the symmetric kernel
\begin{align}\label{K-kernel}
K_{i,j}^n(\phi,\varphi)=\frac{H_{i,j}^n(\phi,\varphi)}{\nu_{i,\Omega}(\phi)\nu_{j,\Omega}(\varphi)\sin(\varphi){r_{0,j}^2}(\varphi)},
\end{align}
then we may write
\begin{align}\label{TnomegaPP2}
\mathcal{T}_{1,\Omega}^n (h_1,h_2)(\phi)=&d_1\int_0^{\pi} K_{1,1}^n(\phi,\varphi)h_1(\varphi)d\mu_1(\varphi)-d_2\int_0^\pi {K_{1,2}^n(\phi,\varphi)}h_2(\varphi)d\mu_2(\varphi),\\
\nonumber\mathcal{T}_{2,\Omega}^n (h_1,h_2)(\phi)=& d_2\int_0^{\pi}K_{2,2}^n(\phi,\varphi) h_2(\varphi)d\mu_2(\varphi)-d_1\int_0^\pi K_{2,1}^n(\phi,\varphi)h_1(\varphi)d\mu_1(\varphi).
\end{align}
The main result of this section reads as follows.
\begin{lem}\label{lem-ssym}
Let $r_{0,1}, r_{0,2}$ satisfy {{(\bf{H})}} and $ \Omega\in(\overline\Omega_2,\overline\Omega_1)$, then for any $n\geqslant 1$ the operator $\mathcal{T}^n_{\Omega}:\mathbb{H}_\Omega\to \mathbb{H}_\Omega$ is  a self-adjoint Hilbert-Schmidt  operator.
\end{lem}
\begin{proof}
In order to check that $\mathcal{T}^n_{{\Omega}}$ is a Hilbert--Schmidt operator, we need to verify 
that the kernel $K_{i,j}^n$  satisfies the integrability condition
\begin{align}\label{HS-Noo}
\interleave\mathcal{T}^n_\Omega\interleave_{\Omega}:=\sup_{i,j}\left(\int_0^\pi\int_0^\pi {|K_{i,j}^n}(\phi,\varphi)|^2d\mu_{i}(\phi)d\mu_{j}(\varphi)\right)^{\frac12}<\infty.
\end{align}
Indeed, by \eqref{K-kernel}. \eqref{MeasureB-1} and \eqref{HMM11}, one deduces that  
\begin{align*}
\interleave\mathcal{T}^n_\Omega\interleave_{\Omega}^2 =\sup_{i,j}\,& C_n\bigintsss_0^\pi\bigintsss_0^\pi\frac{\sin(\varphi)\sin(\phi)\,{r_{0,i}^{2n}(\phi)r_{0,j}^{2n}(\varphi)}}{R_{i,j}^{2n+1}(\phi,\varphi)\,\nu^i_{\Omega}(\varphi)\nu^j_{\Omega}(\phi)}\,F_n^2\left(\frac{{4r_{i,0}(\phi)r_{0,j}(\varphi)}}{R_{i,j}(\phi,\varphi)}\right)d\varphi d\phi,
\end{align*}
for some constant $C_n$ and $R_{i,j}$ was defined in \eqref{Rij}. Moreover, according to Proposition \ref{Lem-meas} the function  $\nu_{i,\Omega}(\varphi)$ is not vanishing in the interval $[0,\pi]$ provided that $\Omega\in(\overline{\Omega}_{2},\overline{\Omega}_1)$. Therefore we get 
\begin{align*}
{\interleave\mathcal{T}^n_\Omega\interleave_{\Omega}^2}\lesssim&\bigintsss_0^\pi\bigintsss_0^\pi\frac{\sin(\varphi)\sin(\phi)}
{{ R_{i,j}(\phi, \varphi)}}
F_n^2\left(\frac{{ 4r_{0,i}(\phi)r_{0,j}}(\varphi)}{R_{i,j}(\phi,\varphi)}\right)d\varphi d\phi.
\end{align*}
Hence, applying \eqref{Chord-1} and {(\bf{H2})} we find
\begin{align*}
\interleave\mathcal{T}^n_\Omega\interleave_{\Omega}^2\lesssim&\sup_{i,j}\,\bigintsss_0^\pi\bigintsss_0^\pi F_n^2\left(\frac{{ 4r_{0,i}(\phi)r_{0,j}(\varphi)}}{R_{i,j}(\phi,\varphi)}\right)d\varphi d\phi.
\end{align*}
By \eqref{estimat-1}, \eqref{Chord-1} and the assumption {\bf{(H2)}} we deduce that
\begin{align*}
\interleave\mathcal{T}^n_\Omega\interleave_{\Omega}^2\leq&C+C\sup_{i}\bigintsss_0^\pi\bigintsss_0^\pi
\ln^2\left(1-\frac{{ 4r_{0,i}(\phi)r_{0,j}(\varphi)}}{{R_{i,j}}(\phi,\varphi)}\right)d\varphi d\phi\\
\leq&C+C\sup_i\bigintsss_0^\pi\bigintsss_0^\pi
\ln^2\left(\frac{{(r_{0,i}(\phi)-r_{0,i}}(\varphi))^2+d_i^2(\cos\phi-\cos\varphi)^2}{R_{i,i}(\phi,\varphi)}\right)d\varphi d\phi.
\end{align*}
According to   \eqref{Chord} and the mean value theorem,  we get for some constant $C>0$  
\begin{equation}\label{Est-LogM}\forall\, \phi\neq \varphi\in[0,\pi],\quad 1\leqslant \frac{{(r_{0,i}(\phi)+r_{0,i}}(\varphi))^2+d_i^2(\cos\phi-\cos\varphi)^2}{(r_i(\phi)-r_i(\varphi))^2+d_i^2(\cos\phi-\cos\varphi)^2}\leqslant C\frac{(\sin\phi+\sin\varphi)^2}{(\phi-\varphi)^2}+C.
\end{equation}
It follows that 
\begin{align*}
\interleave\mathcal{T}^n_\Omega\interleave_{\Omega}^2\leqslant&C+C\bigintsss_0^\pi\bigintsss_0^\pi
\ln^2\left(\frac{\sin\phi+\sin\varphi}{|\phi-\varphi|}\right)d\varphi d\phi<\infty.
\end{align*}
This concludes that the operator $\mathcal{T}^n_{{\Omega}}$ is bounded and is of  Hilbert--Schmidt type. 
As a consequence of spectral  theory this operator is necessarily compact.\\
It remains to check that the operator $\mathcal{T}^n_\Omega$ is symmetric. Indeed,
\begin{align*}
 \langle \mathcal{T}^n_\Omega H,\widehat{H}\rangle_{\Omega} =&\langle \mathcal{T}_{1,\Omega}^n H,\widehat{h}_1\rangle_{\mu_1}+{\frac{d_2}{d_1}}\langle \mathcal{T}_{2,\Omega}^n H,\widehat{h}_2\rangle_{\mu_2}\\
 &=d_1\int_0^{\pi}\int_0^{\pi} K_{1,1}^n(\phi,\varphi)h_1(\varphi)\widehat{h}_1(\phi)d\mu_1(\varphi)d\mu_1(\phi)\\
 &\quad-d_2\int_0^{\pi}\int_0^{\pi}K_{1,2}^n(\phi,\varphi)h_2(\varphi)\widehat{h}_1(\phi)d\mu_1(\phi)d\mu_2(\varphi)\\
 &\qquad +{\frac{{d_2^2}}{d_1}}\int_0^{\pi}\int_0^{\pi}K_{2,2}^n(\phi,\varphi) h_2(\varphi)\widehat{h}_2(\phi)d\mu_2(\varphi)d\mu_2(\phi)\\
 &\qquad\quad-{d_2}\int_0^{\pi}\int_0^{\pi}K_{2,1}^n(\phi,\varphi)h_1(\varphi)\widehat{h}_2(\phi)d\mu_2(\phi)d\mu_1(\varphi).
\end{align*}
Similarly
\begin{align*}
 \langle \mathcal{T}^n_\Omega\widehat{H},H\rangle_{\Omega} &=d_1\int_0^{\pi}\int_0^{\pi} K_{1,1}^n(\phi,\varphi)\widehat{h}_1(\varphi)h_1(\phi)d\mu_1(\varphi)d\mu_1(\phi)\\
 &\quad -d_2\int_0^{\pi}\int_0^{\pi}K_{1,2}^n(\phi,\varphi)\widehat{h}_2(\varphi)h_1(\phi)d\mu_1(\phi)d\mu_2(\varphi)\\
 &\qquad+{\frac{d_2^2}{d_1}}\int_0^{\pi}\int_0^{\pi}K_{2,2}^n(\phi,\varphi) \widehat{h}_2(\varphi){h}_2(\phi)d\mu_2(\varphi)d\mu_2(\phi)\\
 &\quad\qquad -d_2\int_0^{\pi}\int_0^{\pi}K_{2,1}^n(\phi,\varphi)\widehat{h}_1(\varphi)h_2(\phi)d\mu_2(\phi)d\mu_1(\varphi).
\end{align*}
By exchanging $(\varphi,\phi)$  we infer
$$
\int_0^{\pi}\int_0^{\pi}K_{1,2}^n(\phi,\varphi)h_2(\varphi)\widehat{h}_1(\phi)d\mu_1(\phi)d\mu_2(\varphi)=\int_0^{\pi}\int_0^{\pi}K_{1,2}^n(\varphi,\phi) h_2(\phi)\widehat{h}_1(\varphi)d\mu_1(\varphi)d\mu_2(\phi).
$$
From the structure of $K_{i,j}^n$ we get
$$
K_{i,j}^n(\phi,\varphi)=K_{j,i}^n(\varphi,\phi)
$$
leading to 
$$
\int_0^{\pi}\int_0^{\pi}K_{1,2}^n(\phi,\varphi)h_2(\varphi)\widehat{h}_1(\phi)d\mu_1(\phi)d\mu_2(\varphi)=\int_0^{\pi}\int_0^{\pi}K_{2,1}^n(\phi,\varphi) h_2(\phi)\widehat{h}_1(\varphi)d\mu_1(\varphi)d\mu_2(\phi).
$$
One may also obtain using similar arguments
$$
\int_0^{\pi}\int_0^{\pi}K_{2,1}^n(\phi,\varphi)h_1(\varphi)\widehat{h}_2(\phi)d\mu_2(\phi)d\mu_1(\varphi)=\int_0^{\pi}\int_0^{\pi}K_{1,2}^n(\phi,\varphi) h_1(\phi)\widehat{h}_2(\varphi)d\mu_1(\phi)d\mu_2(\varphi).
$$
Therefore we find that $\mathcal{T}^n_\Omega$ is symmetric, that is, 
$$
\langle \mathcal{T}^n_\Omega\widehat{H},H\rangle_{\Omega}  =\langle \mathcal{T}^n_\Omega H,\widehat{H}\rangle_{\Omega}.
$$
This ends the proof of Lemma \ref{lem-ssym}
\end{proof}
\subsection{Structures of the functions $\nu_{i,\Omega}$.}
The main task is to investigate some analytical properties of the functions $\nu_{i,\Omega}$ given in Lemma \ref{nuOmega12} that will be used later on the distribution of the largest eigenvalues of the operators seen in the preceding section. We point out that the two functions do not play symmetric roles in the spectral study, only $\nu_{1,\Omega}$ will be with a crucial contribution.
\begin{pro} \label{Lem-meas}
Let  $r_{0,1},r_{0,2}$ be two profiles  satisfying the assumptions ${\bf{(H)}}$ and  $\alpha\in(0,1)$.  Then  the following properties hold true.
\begin{enumerate}
\item  
There exists $\mathtt{M}>0$ such that for any $ \Omega\in(\overline\Omega_2, \overline\Omega_1)$ we get
$$
\forall  \phi\in [0,\pi],\quad \mathtt{M}\geqslant \nu_{1,\Omega}(\phi)\geqslant\overline\Omega_1-\Omega>0,\quad  \mathtt{M}\geqslant \nu_{2,\Omega}(\phi)\geqslant\Omega-\overline\Omega_2>0.
$$
\item  The function $\phi\in[0,\pi]\mapsto \nu_{1,\Omega}(\phi)$ belongs to $\mathscr{C}^{1,\alpha}([0,\pi]).$ In addition,
$$
\nu_{1,\Omega}^\prime(0)=\nu_{1,\Omega}^\prime(\pi)=0.
$$
\item Let  $\Omega\in[\overline\Omega_2, \overline\Omega_1]$ and assume that  $\nu_{1,\Omega}$ reaches its  minimum at a point $\phi_0\in[0,\pi]$ then there exists  $C>0$ independent of $\Omega$  such that, $$
\forall \phi\in[0,\pi],\quad 0\leqslant \nu_{1,\Omega}\big(\phi\big)-\nu_{1,\Omega}\big(\phi_0\big)\leqslant C|\phi-\phi_0|^{1+\alpha}.
$$
Moreover, for  $\Omega=\overline\Omega_1$ this result becomes 
$$
\forall \phi\in[0,\pi],\quad 0\leqslant \nu_{1,\overline\Omega_1}\big(\phi\big)\leqslant C|\phi-\phi_0|^{1+\alpha}.
$$

\end{enumerate}
\end{pro}
\begin{proof}
$(1)$ It is immediate and follows by definition from \eqref{Omega1}-\eqref{Omega2}.\\
$(2)$ We first make the decompose
\begin{align*}
\nu_{1,\Omega}(\phi)&=d_1\int_0^\pi H_{1,1}^1(\phi,\varphi)d\varphi-d_2\int_0^{\pi} H_{1,2}^1(\phi,\varphi)d\varphi-\Omega\\
&=:\nu^1_{1,\Omega}(\phi)-\nu^2_{1,\Omega}(\phi)-\Omega.
\end{align*}
 According to {\cite[Proposition 4.1-(3)]{GHM},} the function $\nu^1_{1,\Omega}$ is $\mathscr{C}^{1,\alpha}([0,\pi])$ and   
 $$
(\nu^1_{1,\Omega})^\prime(0)=0=(\nu^1_{1,\Omega})^\prime(\pi).
$$
{As to $\nu^2_{1,\Omega}$, we write from \eqref{HMM11}
$$
H_{1,2}^1(\phi,\varphi)=\frac{1}{2}\frac{\sin(\varphi)r_{0,2}^{2}(\varphi)}{\left[R_{1,2}(\phi,\varphi)\right]^{\frac32}} F_1\left(\frac{4r_{0,1}(\phi)r_{0,2}(\varphi)}{R_{1,2}(\phi,\varphi)}\right).
$$
This case will be easier because the function $H_{1,2}^1(\phi,\varphi)$ does not have singularities. More precisely,  the functions $\phi\in[0,\pi]\mapsto \frac{1}{
\left[R_{1,2}(\phi,\varphi)\right]^{\frac32}}, \frac{4r_{0,1}(\phi)r_{0,2}(\varphi)}{R_{1,2}(\phi,\varphi)}$ are $\mathscr{C}^{1,\alpha}([0,\pi])$ uniformly in the variable $\varphi.$
On the other hand, by \eqref{Chord-1}  
$\frac{4r_{0,1}(\phi)r_{0,2}(\varphi)}{R_{1,2}(\phi,\varphi)}\leqslant  \overline\delta_{1,2}<1$ and the hypergeometric function $F_1$ is analytic in the open unit disc. Then by the chain rule $\phi\mapsto F_1\left(\frac{4r_{0,1}(\phi)r_{0,2}(\varphi)}{R_{1,2}(\phi,\varphi)}\right)$ is $\mathscr{C}^{1,\alpha}([0,\pi])$. Since $\mathscr{C}^{1,\alpha}([0,\pi])$  is an algebra  we get that the function $\phi\mapsto H_{1,2}^1(\phi,\varphi)\in \mathscr{C}^{1,\alpha}([0,\pi])$ uniformly in $\varphi\in[0,\pi].$ Therefore  $\phi\mapsto \nu^2_{1,\Omega}(\phi))$ is $\mathscr{C}^{1,\alpha}([0,\pi])$.\\
To compute the derivatives of the function $\nu^1_{1,\Omega}(\phi)$ at the points $0$ and $\pi$ we use the definition of the hypergeometric function as a power series (see \eqref{GaussF}), 
\begin{align*}
H_{1,2}^1(\phi,\varphi)&=\frac{1}{2}\frac{\sin(\varphi)r_{0,2}^{2}(\varphi)}{\left[R_{1,2}(\phi,\varphi)\right]^{\frac32}} F_1\left(\frac{4r_{0,1}(\phi)r_{0,2}(\varphi)}{R_{1,2}(\phi,\varphi)}\right)\\
&=\frac{1}{2}\frac{\sin(\varphi)r_{0,2}^{2}(\varphi)}{\left[R_{1,2}(\phi,\varphi)\right]^{\frac32}} \left(1+3\frac{r_{0,1}(\phi)r_{0,2}(\varphi)}{R_{1,2}(\phi,\varphi)}+\sum_{n\ge 2}\frac{(\frac{3}{2})_n(\frac{3}{2})_n}{(3)_n}\frac{4^n r_{0,1}(\phi)^n r_{0,2}(\varphi)^n}{n! R_{1,2}(\phi,\varphi)^n}\right).
\end{align*}
Since $r_{0,1}(0)=0$ and 
$\partial_\phi R_{1,2}(0,\varphi)=2r_{0,1}^\prime(0)r_{0,2}(\varphi),$
 we obtain that 
$$
\partial_\phi H_{1,2}^1(0,\varphi)
= \frac{1}{2}\sin(\varphi)r_{0,2}^2(\varphi)\Big(-\frac{3}{2} \frac{2r^\prime_{0.1}(0)r_{0,2}(\varphi)}{(R_{1,2}(0,\varphi))^{5/2}}+3 \frac{r^\prime_{0,1}(0)r_{0,2}(\varphi)}{(R_{1,2}(0,\varphi))^{5/2}}  \Big)=0.
$$
Repeating the same arguments one can prove that 
\begin{align*}
\partial_\phi H_{1,2}^1(\pi,\varphi)=&0.
\end{align*}
Hence
$$
(\nu^2_{1,\Omega})^\prime(0)=0=(\nu^2_{1,\Omega})^\prime(\pi).
$$
}
$(3)$
By the extreme value theorem,  the function  $\nu_{1,\Omega}$ reaches its absolute  minimum at some \mbox{  point $\phi_0\in[0,\pi]$.}  If this point belongs to the open set $(0,\pi)$ then necessary $\nu_{1,\Omega}^\prime(\phi_0)=0$.  However when $\phi_0\in\{0,\pi\}$, using
 point {\bf{(2),}} we deduce  that the derivative is also vanishing at $\phi_0.$ Hence, using the mean value theorem, we obtain for any $\phi\in[0,\pi]$
\begin{align*}
\nu_{1,\Omega}(\phi)=&\nu_{1,\Omega}(\phi_0)+\nu_{1,\Omega}^\prime\big(\overline\phi\big)(\phi-\phi_0)=\nu_{1,\Omega}(\phi_0)+\big(\nu_{1,\Omega}^\prime\big(\overline\phi\big)-\nu_{1,\Omega}^\prime\big(\phi_0\big)\big)(\phi-\phi_0),
\end{align*}
for some $\overline\phi\in(\phi_0,\phi)$. Since $\nu_{1,\Omega}^\prime\in \mathscr{C}^\alpha,$ then 
$$
\Big|\nu_{1,\Omega}^\prime\big(\overline\phi\big){ -}\nu_{1,\Omega}^\prime\big(\phi_0\big)\Big|\leqslant \|\nu_{1,\Omega}^\prime\|_{\mathscr{C}^\alpha}|\phi-\phi_0|^\alpha.
$$
Notice that $\|\nu_{1,\Omega}^\prime\|_{\mathscr{C}^\alpha}$ is independent of $\Omega.$ Consequently
$$
\forall \phi\in[0,\pi],\quad 0\leqslant \nu_{1,\Omega}\big(\phi\big)-\nu_{1,\Omega}\big(\phi_0\big)\leqslant  C|\phi-\phi_0|^{1+\alpha},
$$
for some absolute constant $C$. In the particular case $\Omega=\overline\Omega_1$ we get from the definition \eqref{Omega1} that  $\nu_{1,\overline\Omega_1}(\phi_0)=0$ and therefore the preceding result becomes
$$
\forall \phi\in[0,\pi],\quad 0\leq \nu_{1,\overline\Omega_1}\big(\phi\big)\leqslant C|\phi-\phi_0|^{1+\alpha}.
$$
The proof is now complete. 
\end{proof}
\subsection{Largest eigenvalues distribution}
In what follows we shall perform refined estimates on the distribution of the largest eigenvalues  with respect to the modes $n$. This is the cornerstone part in checking the assumptions of Crandall-Rabinowitz theorem. The results are summarized in the following proposition. For the different expressions of  the operator $\mathcal{T}^n_{{\Omega}},$ we refer to \eqref{Tnomega2}  and \eqref{TnomegaPP2}.

\begin{pro}\label{prop-operatorV2}
Let $\Omega\in(\overline\Omega_2, \overline\Omega_1)$ and { $r_{0,1},r_{0,2}$} satisfy the assumptions ${\bf{(H)}}$. Then, the following assertions hold true.
\begin{enumerate}
\item For any $n\geqslant1$, the eigenvalues of $\mathcal{T}^n_{{\Omega}}$ form a countable family of real numbers. Let  $\lambda_n(\Omega)$ be  the largest eigenvalue, then it is strictly positive   and satisfies
\begin{align*}
&{d_1}\tiny{\iint_{[0,\pi]^2}\frac{H_{1,1}^n(\phi,\varphi)}{{\nu_{1,\Omega}}(\varphi)^{\frac12}{\nu_{1,\Omega}}(\phi)^{\frac12}}\frac{\sin^\frac12(\phi){ r_{0,1}}(\phi)}{\sin^\frac12(\varphi){r_{0,1}}(\varphi)} \varrho(\varphi)\varrho(\phi)d\varphi d\phi}\leqslant \lambda_n(\Omega)\\
& \qquad \qquad \leqslant2{\big(d_1+d_2\big)}\max_{i,j}{\left(\iint_{[0,\pi]^2}|K_{i,j}^n(\phi,\varphi)|^2d\mu_{i}(\phi)d\mu_{j}(\varphi)\right)^\frac12,}
\end{align*}
for any function $\varrho$ such that $\displaystyle{\int_0^\pi \varrho^2(\varphi)d\varphi=1.}$
\item We have the following decay: for any $\gamma\in[0,{\frac14})$ there exists $C>0$ such that 
$$
\forall\, \Omega\in(\overline\Omega_2,\overline\Omega_1),\,\forall\, n\geqslant1,\quad \ \lambda_n(\Omega)\leqslant {Cn^{-\gamma}\max_{i,j}\left|(\Omega-\overline\Omega_i)(\Omega-\overline\Omega_j)\right|^{-\frac12} }.
$$
\item The eigenvalue $\lambda_n(\Omega)$ is simple {{and the components of any  associated nonzero 
 eigenfunctions $(h_1,h_2)$  are with  constant opposite signs.}} 
\item For any $\Omega\in(\overline\Omega_2,\overline\Omega_1)$, the sequence $n\geqslant1\mapsto \lambda_n(\Omega)$ is strictly decreasing.
\item For any $n\geqslant1$ the map $\Omega\in(\overline\Omega_2,\overline\Omega_1)\mapsto \lambda_n(\Omega)$ is differentiable and strictly increasing.
\end{enumerate}
\end{pro}

\begin{proof}

\medskip
\noindent
{\bf (1)}  \medskip
\noindent
Using  the spectral theorem on self-adjoint compact operators, we get  that the eigenvalues of $\mathcal{T}^n_{{\Omega}}$ form a countable family of real numbers. Set
$$
m=\inf_{\|H\|=1} \langle\mathcal{T}^n_{{\Omega}} H, H\rangle_{{\Omega}} \quad \textnormal{and} \quad M=\sup_{\|H\|=1} \langle\mathcal{T}^n_{{\Omega}} H, H\rangle_{{\Omega}}.
$$
Since $\mathcal{T}^n_\Omega$ is self-adjoint, we get that the spectrum of $\mathcal{T}^n_{{\Omega}}$ denoted by $\sigma(\mathcal{T}^n_{{\Omega}})$ satisfies   $\sigma (\mathcal{T}^n_{{\Omega}})\subset[m,M]$, with $m\in\sigma (\mathcal{T}^n_{{\Omega}})$ and $M\in\sigma(\mathcal{T}^n_{{\Omega}})$. By definition of   the largest eigenvalue $\lambda_n(\Omega)$ we infer
\begin{equation}\label{eigenvPP}
\lambda_n(\Omega)=M=\sup_{\|H\|=1} \langle\mathcal{T}^n_{{\Omega}} H, H\rangle_{{\Omega}}.
\end{equation}
We shall prove that $M>0$ and $|m|\leq M$. Indeed, for any $H=\begin{pmatrix}
h_1\\
h_2
\end{pmatrix}\in \mathbb{H}_\Omega$, 

\begin{align*}
 \langle\mathcal{T}^n_{{\Omega}} H, H\rangle_{{\Omega}}
 =&d_1{\int_0^{\pi}}\int_0^{\pi} K_{1,1}^n(\phi,\varphi)h_1(\varphi)h_1(\phi)d\mu_1(\varphi)d\mu_1(\phi)\\
 &\quad-d_2{\int_0^{\pi}}\int_0^\pi K_{1,2}^n(\phi,\varphi)h_2(\varphi)h_1(\phi)d\mu_1(\phi)d\mu_2(\varphi)\\
 &\qquad +\frac{d_2^2}{d_1}{\int_0^{\pi}}\int_0^{\pi}K_{2,2}^n(\phi,\varphi) h_2(\varphi)h_2(\phi)d\mu_2(\varphi)d\mu_2(\phi)\\
 &\quad\qquad -d_2{\int_0^{\pi}}\int_0^\pi K_{2,1}^n(\phi,\varphi)h_1(\varphi)h_2(\phi)d\mu_2(\phi)d\mu_1(\varphi).
\end{align*}
By symmetry of $K_{i,j}^n$ we infer
\begin{align}\label{Tit1}
 \nonumber\langle\mathcal{T}^n_{{\Omega}} H, H\rangle_{{\Omega}}
 =&d_1{\int_0^{\pi}}\int_0^{\pi} K_{1,1}^n(\phi,\varphi)h_1(\varphi)h_1(\phi)d\mu_1(\varphi)d\mu_1(\phi)\\
  \nonumber&\quad+\frac{d_2^2}{d_1}{\int_0^{\pi}}\int_0^{\pi}K_{2,2}^n(\phi,\varphi) h_2(\varphi)h_2(\phi)d\mu_2(\varphi)d\mu_2(\phi)\\
 &\qquad -2d_2{\int_0^{\pi}}\int_0^\pi K_{1,2}^n(\phi,\varphi)h_2(\varphi)h_1(\phi)d\mu_1(\phi)d\mu_2(\varphi).
 \end{align}
Denote $H^+:=\begin{pmatrix}
|h_1|\\
-|h_2|
\end{pmatrix}$, then it belongs to  $ \mathbb{H}_\Omega$ with  the same norm as $H$. Notice that
\begin{align}\label{Tit2}
  \nonumber\langle\mathcal{T}^n_{{\Omega}} H_+, H_+\rangle_{{\Omega}}
 =&d_1{\int_0^{\pi}}\int_0^{\pi} K_{1,1}^n(\phi,\varphi)|h_1(\varphi)h_1(\phi)|d\mu_1(\varphi)d\mu_1(\phi)\\
  \nonumber&\quad+\frac{d_2^2}{d_1}{\int_0^{\pi}}\int_0^{\pi}K_{2,2}^n(\phi,\varphi) |h_2(\varphi)h_2(\phi)|d\mu_2(\varphi)d\mu_2(\phi)\\
 &\qquad+2d_2{\int_0^{\pi}}\int_0^\pi K_{1,2}^n(\phi,\varphi)|h_2(\varphi)h_1(\phi)|d\mu_1(\phi)d\mu_2(\varphi).
 \end{align}
  Combined with the positivity of the kernels $K_{i,j}^n$ it implies
$$
|\langle\mathcal{T}^n_{{\Omega}} H, H\rangle_{{\Omega}}|\leqslant \langle\mathcal{T}^n_{{\Omega}} H_+, H_+\rangle_{{\Omega}}.
$$
Therefore 
$$
\sup_{\|H\|=1} \langle\mathcal{T}^n_{{\Omega}} H, H\rangle_{{\Omega}}=\sup_{ \|H_+\|=1} \langle\mathcal{T}^n_{{\Omega}} H_+, H_+\rangle_{{\Omega}}.
$$
Using once again  the positivity of the kernels one deduces that 
$$
 \|H_+\|=1\Longrightarrow \langle\mathcal{T}^n_{{\Omega}} H_+, H_+\rangle_{{\Omega}}>0.
$$
Consequently, we obtain that $M>0$ and  $|m|\le M$. By Cauchy-Schwarz inequality we get  for
$$\|H_+\|^2=\|h_1\|_{\mu_1}^2+\tfrac{d_2}{d_1}\|h_2\|_{\mu_2}^2=1
$$
that
\begin{align*}
 \langle\mathcal{T}^n_{{\Omega}} H_+, H_+\rangle_{{\Omega}}
 \leqslant& d_1\left(\int_0^{\pi}\int_0^{\pi} |K_{1,1}^n(\phi,\varphi)|^2d\mu_1(\varphi)d\mu_1(\phi)\right)^{\frac12}\|h_1\|_{\mu_1}^2\\
 &+\tfrac{d_2^2}{d_1}\left(\int_0^{\pi}\int_0^{\pi}|K_{2,2}^n(\phi,\varphi)|^2d\mu_2(\varphi)d\mu_2(\phi)\right)^{\frac12}\|h_2\|_{\mu_2}^2\\
 &+2d_2\left(\int_0^{\pi}\int_0^{\pi}|K_{1,2}^n(\phi,\varphi)|^2d\mu_1(\phi)d\mu_2(\varphi)\right)^{\frac12}\|h_1\|_{\mu_1}\|h_2\|_{\mu_2}.
 \end{align*}
It follows that 
$$
\lambda_n({{\Omega}})
\leqslant 2{\big(d_1+d_2\big)}\max_{i,j}\left(\int_0^\pi\int_0^\pi|K_{i,j}^n(\phi,\varphi)|^2d\mu_{i}(\phi)d\mu_{j}(\varphi)\right)^{\frac12}.
$$
To get  the lower bound,  it suffices to  work with the particular  function
{$$
H=\begin{pmatrix}
h\\
0
\end{pmatrix},\quad h(\varphi)=\frac{\varrho(\varphi)}{\sin(\varphi)^\frac12 {r_{0,1}}(\varphi)\nu_{1,\Omega}(\varphi)^{\frac12}},\quad \varphi\in(0,\pi),
$$}
with the  normalized condition $\|H\|=1$ which is equivalent to 
$$
\int_0^\pi\varrho^2(\varphi)d\varphi=1,
$$
and by definition we infer
$$
\lambda_n(\Omega)\geqslant\langle\mathcal{T}^n_{{\Omega}}H,H\rangle={d_1}\bigintsss_0^\pi\bigintsss_0^\pi\frac{H_n(\phi,\varphi)}{\nu_{1,\Omega}^{\frac12}(\varphi)\nu_{1,\Omega}^{\frac12}(\phi)}\frac{\sin^\frac12(\phi) { r_{0,1}}(\phi)}{\sin^\frac12(\varphi){ r_{0,1}}(\varphi)} \varrho(\varphi)\varrho(\phi)d\varphi d\phi.
$$
This gives the result on lower bound for the largest eigenvalue.

\medskip
\noindent
{\bf (2)}
Applying  \eqref{K-kernel} and \eqref{HS-Noo} we easily get
$$
 \interleave\mathcal{T}^n_\Omega\interleave_{\Omega}^2=\sup_{i,j}\bigintsss_0^\pi\bigintsss_0^\pi 
 \frac{|H_{i,j}^n(\phi,\varphi)|^2}{\nu_{i,\Omega}(\phi)\nu_{j,\Omega}(\varphi)}\frac{\sin(\phi){r_{0,i}^2}(\phi)}{\sin(\varphi){ r_{0,j}^2}(\varphi)}d\phi d\varphi.
$$
Using Proposition \ref{Lem-meas}-(1) we infer 
$$
\forall \Omega\in (\overline\Omega_2, \overline\Omega_1),\quad \forall \, \phi\in[0,\pi],\quad \nu_{i,\Omega}(\phi)\geq |\overline\Omega_i-\Omega
|$$
and therefore we obtain
$$
 \interleave\mathcal{T}^n_\Omega\interleave_{\Omega}^2\lesssim \sup_{i,j}|\overline\Omega_i-\Omega
|^{-1}|\overline\Omega_j-\Omega
|^{-1}
\bigintsss_0^\pi\bigintsss_0^\pi {|H_{i,j}^n(\phi,\varphi)|^2}\frac{\sin(\phi){ r_{0,i}^2(\phi)}}{\sin(\varphi){ r_{0,j}^2}(\varphi)}d\phi  d\varphi.
$$ 
Applying Lemma \ref{Lem-Hndecreasing} combined with the assumption {\bf{(H2)}} yields for any $0\leq \gamma<\beta\leq1$
\begin{align*}
\interleave\mathcal{T}^n_\Omega\interleave_{\Omega}^2\lesssim& \sup_{i,j} |\overline\Omega_i-\Omega
|^{-1}|\overline\Omega_j-\Omega
|^{-1}
 n^{-2\gamma}\bigintsss_0^\pi\bigintsss_0^\pi|\phi-\varphi|^{-4\beta}d\phi d\varphi.
\end{align*}
By taking $\beta<\frac14$ we get the convergence of the integral and consequently we obtain the desired result,
\begin{align*}
{\interleave\mathcal{T}^n_\Omega\interleave_{\Omega}^2\lesssim \sup_{i,j}|\overline\Omega_i-\Omega
|^{-1}|\overline\Omega_j-\Omega
|^{-1} n^{-2\gamma}.}
\end{align*}
It follows that
\begin{align}\label{Est-kern-PM}
\lambda_n(\Omega)\lesssim \sup_{i,j}|\overline\Omega_i-\Omega
|^{-\frac12}|\overline\Omega_j-\Omega
|^{-\frac12} n^{-\gamma}.
\end{align}
Since $\beta<\frac14$ and $\gamma\leq \beta$, we get $\gamma<\frac14$.

\medskip
\noindent
{\bf (3)} 
Consider a  normalized  eigenfunction $H=(h_1,h_2)\in\mathbb{H}_\Omega$ associated to the largest eigenvalue $\lambda_n(\Omega)$, then
$$
\mathcal{T}^n_\Omega H=\lambda_n(\Omega) H\quad\hbox{with}\quad \|H\|=1.
$$
Moreover, for $H_+=(|h|,-|k|)$ 
$$
 \lambda_n(\Omega)=\langle\mathcal{T}^n_{{\Omega}} H, H\rangle_{{\Omega}}= \langle\mathcal{T}^n_{{\Omega}} H_+, H_+\rangle_{{\Omega}}.
$$
Therefore, combining this with \eqref{Tit1} and \eqref{Tit2} we infer from the positivity of the kernels
\begin{align*}
 \nonumber\int_0^{\pi}\int_0^{\pi} K_{1,1}^n(\phi,\varphi)\Big(|h_1(\varphi)h_1(\phi)|-h_1(\varphi)h_1(\phi)\Big)d\mu_1(\varphi)d\mu_1(\phi)&=0
 \\
\int_0^{\pi}\int_0^{\pi}K_{2,2}^n(\phi,\varphi) \Big(|h_2(\varphi)h_2(\phi)|-h_2(\varphi)h_2(\phi)\Big)d\mu_2(\varphi)d\mu_2(\phi)&=0\\
\int_0^{\pi}\int_0^{\pi} K_{1,2}^n(\phi,\varphi)\Big(|h_2(\varphi)h_1(\phi)|+h_2(\varphi)h_1(\phi)\Big)d\mu_1(\phi)d\mu_2(\varphi)&=0.
 \end{align*}
This implies that almost everywhere in $\varphi,\phi\in(0,\pi)$
$$
h_1(\varphi)h_1(\phi)\geqslant 0, h_2(\varphi)h_2(\phi)\geqslant 0\quad\hbox{and}\quad h_2(\varphi)h_1(\phi)\leqslant 0.
$$
{Finally, we shall  check  that the subspace generated by the eigenfunctions associated to $\lambda_n(\Omega)$ is one--dimensional. Assume that we have two independent normalized eigenfunctions $H_0=\begin{pmatrix}
h_{0}\\
k_0
\end{pmatrix}$ and $H_1=\begin{pmatrix}
h_{1}\\
k_1
\end{pmatrix}$. Then $h_i,k_i$ are with constant opposite signs (we may assume  that $h_i$ positive and $k_i$ negative) and we may find from the independence that $\phi_0\in(0,\pi)$ such that the functions are non vanishing at this point and 
$$
\frac{h_1(\phi_0)}{h_0(\phi_0)}\neq \frac{k_1(\phi_0)}{k_0(\phi_0)}\cdot
$$ 
From this latter property one can find two numbers $\alpha, \beta$ such tghat
$$
-\tfrac{h_1(\phi_0)}{h_0(\phi_0)}\beta<\alpha<-\tfrac{k_1(\phi_0)}{k_0(\phi_0)}\beta.
$$
Since the linear combination $\alpha H_0+\beta H_1=$ is an eigenfunction then its components $\alpha h_0+\beta h_1$ $\alpha k_0+\beta k_1$ should be  with constant and opposite signs that can be checked at $\phi_0$. However, at $\phi_0$ we have 
$$
\alpha h_0(\phi_0)+\beta h_1(\phi_0)>0\quad \hbox{and}\quad  \alpha k_0(\phi_0)+\beta k_1(\phi_0)>0
$$
which is a contradiction. }Therefore, the eigenspace associated to the largest eigenvalue is of dimension one.\\
\medskip
\noindent
{\bf (4)} Let $H_+=\begin{pmatrix}
h_{1}\\
{-}h_{2}
\end{pmatrix}$ be a unit eigenfunction associated to $\lambda_{n+1}(\Omega)$ with $h_1,h_2\geqslant0$.  Then
$$
\lambda_{n+1}(\Omega)=
\langle\mathcal{T}^{n+1}_{{\Omega}} H_+, H_+\rangle_{{\Omega}}.
$$
From \eqref{Tit2} we infer
\begin{align*}
  \nonumber\lambda_{n+1}(\Omega)
 =&d_1\int_0^\pi\int_0^{\pi} K_{1,1}^{n+1}(\phi,\varphi)h_1(\varphi)h_1(\phi)d\mu_1(\varphi)d\mu_1(\phi) \\
 &+\tfrac{d_2^2}{d_1}\int_0^{\pi}\int_0^{\pi}K_{2,2}^{n+1}(\phi,\varphi) h_2(\varphi)h_2(\phi)d\mu_2(\varphi)d\mu_2(\phi)\\
 &\qquad+{2d_2}\int_0^\pi \int_0^\pi K_{1,2}^{n+1}(\phi,\varphi)h_2(\varphi)h_1(\phi)d\mu_1(\phi)d\mu_2(\varphi).
 \end{align*}
  {Applying Lemma \ref{Lem-Hndecreasing} with \eqref{K-kernel} implies that  $n\in\N^{\star}\mapsto K_{i,j}^n(\phi,\varphi)$ is strictly decreasing for any $\varphi\neq \phi\in(0,\pi)$. Then, from the positivity of $h_i$ we infer for any $\Omega\in(\overline\Omega_2,\overline\Omega_1)$ 
  \begin{align*}
  \nonumber\lambda_{n+1}(\Omega)
 <&d_1\int_0^{\pi}\int_0^{\pi} K_{1,1}^{n}(\phi,\varphi)h_1(\varphi)h_1(\phi)d\mu_1(\varphi)d\mu_1(\phi)\\
 &\quad+\frac{d_2^2}{d_1}\int_0^{\pi}\int_0^{\pi}K_{2,2}^{n}(\phi,\varphi) h_2(\varphi)h_2(\phi)d\mu_2(\varphi)d\mu_2(\phi)\\
 &\qquad+{2d_2}\int_0^{\pi}\int_0^{\pi}K_{1,2}^{n}(\phi,\varphi)h_2(\varphi)h_1(\phi)d\mu_1(\phi)d\mu_2(\varphi)=\langle\mathcal{T}^{n}_{{\Omega}} H_+, H_+\rangle_{{\Omega}}\\
 &\quad\qquad <\lambda_{n}(\Omega).\end{align*}}
This achieves the desired result.

\medskip
\noindent
{\bf (5)} {  Fix $\Omega_0\in(\overline\Omega_2,\overline\Omega_1)$ and denote by $H_\Omega=\begin{pmatrix}
h_{1,\Omega}\\
{-}h_{2,\Omega}
\end{pmatrix}$ a normalized  eigenfunction of $\mathcal{T}^n_\Omega$  associated to the eigenvalue $\lambda_n(\Omega)$ with $h_{1,\Omega},h_{2,\Omega}\geqslant0$. Using the definition of the eigenfunction yields
\begin{align}\label{normaliz-1}
\lambda_n(\Omega)=\frac{\langle \mathcal{T}^n_{{\Omega}} H_\Omega, H_{\Omega_0}\rangle_{{\Omega_0}}}{\langle H_\Omega, H_{\Omega_0}\rangle_{{\Omega_0}}}, \quad \|H_\Omega\|_{{\Omega_0}}=1.
\end{align}
Remark that in this part  we use the notation  $\|\cdot\|_{{\Omega_0}} $ to denote the norm defined in \eqref{Norm-Hil1}. The regularity follows from the general theory using the fact that this eigenvalue is simple. However we can in our  case give a direct proof for its differentiability in a similar way to \cite[Proposition 4.2]{GHM}. 
From the decomposition 
\begin{align*}
\frac{1}{\nu_{1,\Omega}(\phi)}&=\frac{1}{\nu_{1,\Omega_0}(\phi)}+\frac{\Omega-\Omega_0}{\nu_{1,\Omega}(\phi)\nu_{1,\Omega_0}(\phi)},\\
\frac{1}{\nu_{2,\Omega}(\phi)}&=\frac{1}{\nu_{2,\Omega_0}(\phi)}-\frac{\Omega-\Omega_0}{\nu_{2,\Omega}(\phi)\nu_{2,\Omega_0}(\phi)},
\end{align*} 
we get according to the expression of $\mathcal{T}^n_{{1,\Omega}}$ that for any $H\in\mathbb{H}_\Omega$
\begin{align*}
 \mathcal{T}^n_{{1,\Omega}} H
=&\mathcal{T}^n_{{1,\Omega_0}} H+{\frac{(\Omega-\Omega_0)}{\nu_{1,\Omega_0}(\phi)}\mathcal{T}^n_{1,{\Omega}} H},\\
\mathcal{T}^n_{{2,\Omega}} H
=&\mathcal{T}^n_{{2,\Omega_0}} H-{\frac{(\Omega-\Omega_0)}{\nu_{2,\Omega_0}(\phi)}\mathcal{T}^n_{2,{\Omega}} H.}
\end{align*}
Consequently
\begin{align}\label{Iden-t1}
\mathcal{T}^n_{\Omega}=\mathcal{T}^n_{{\Omega_0}}+(\Omega-\Omega_0)\mathscr{R}_n^{\Omega_0,\Omega}, \quad \mathscr{R}_n^{\Omega_0,\Omega}:=\begin{pmatrix}
\nu_{1,\Omega_0}^{-1}\mathcal{T}^n_{1,\Omega}\\
-\nu_{2,\Omega_0}^{-1}\mathcal{T}^n_{2,\Omega}
\end{pmatrix}.
\end{align}
Therefore we obtain
\begin{align*}
\lambda_n(\Omega)=\frac{\langle \mathcal{T}^n_{{\Omega_0}} H_\Omega, H_{\Omega_0}\rangle_{{\Omega_0}}}{\langle H_\Omega, H_{\Omega_0}\rangle_{{\Omega_0}}}+(\Omega-\Omega_0)\frac{\langle\mathscr{R}_n^{\Omega_0,\Omega}H_\Omega, H_{\Omega_0}\rangle_{{\Omega_0}}}{\langle H_\Omega, H_{\Omega_0}\rangle_{{\Omega_0}}}.
\end{align*}
{As $\mathcal{T}^n_{{\Omega_0}}$ is self-adjoint on the Hilbert space $\mathbb{H}_{{\Omega_0}}$  then
$$
\frac{\langle \mathcal{T}^n_{{\Omega_0}} H_\Omega, H_{\Omega_0}\rangle_{{\Omega_0}}}{\langle H_\Omega, H_{\Omega_0}\rangle_{{\Omega_0}}}=\frac{\langle H_\Omega, \mathcal{T}^n_{{\Omega_0}}H_{\Omega_0}\rangle_{{\Omega_0}}}{\langle H_\Omega, H_{\Omega_0}\rangle_{{\Omega_0}}}=\lambda_n(\Omega_0).
$$}
Arguing as in \cite{GHM}, we can prove that
\begin{align}\label{con-R}
{\lim_{\Omega\to \Omega_0}\frac{\langle \mathscr{R}_n^{\Omega_0,\Omega}H_\Omega, H_{\Omega_0}\rangle_{{\Omega_0}}}{\langle H_\Omega, H_{\Omega_0}\rangle_{{\Omega_0}}}={\langle \mathscr{R}_n^{\Omega_0,\Omega_0}H_{\Omega_0}, H_{\Omega_0}}\rangle_{{\Omega_0}}}.
\end{align}
Then we deduce that  $\Omega\mapsto \lambda_n(\Omega)$ is differentiable at $\Omega_0$ with 
\begin{align*}
\lambda_n^\prime(\Omega_0)=&{\langle \mathscr{R}_n^{\Omega_0,\Omega_0}H_{\Omega_0}, H_{\Omega_0}\rangle_{{\Omega_0}}}.
\end{align*}
From \eqref{Iden-t1} we infer
\begin{align*}
 \mathscr{R}_n^{\Omega_0,\Omega_0}H_{\Omega_0}&=\begin{pmatrix}
\nu_{1,\Omega_0}^{-1}\mathcal{T}^n_{1,\Omega_0}H_{\Omega_0}\\
-\nu_{2,\Omega_0}^{-1}\mathcal{T}^n_{2,\Omega_0}H_{\Omega_0}
\end{pmatrix}\\
 &=\lambda_n(\Omega_0)
 \begin{pmatrix}
\nu_{1,\Omega_0}^{-1}h_{1,\Omega_0}\\
{-}\nu_{2,\Omega_0}^{-1}h_{2,\Omega_0}
\end{pmatrix}.
\end{align*}
Combining the preceding identities with  \eqref{inner-P} and \eqref{MeasureB-1}
\begin{align*}
\lambda_n^\prime(\Omega_0)=&\lambda_n(\Omega_0)\left(\int_0^\pi \nu_{1,\Omega_0}^{-1}(\phi)h_{1,\Omega_0}^2(\phi)d\mu_1(\phi)+{\tfrac{d_2}{d_1}}\int_0^\pi \nu_{2,\Omega_0}^{-1} h_{2,\Omega_0}^2(\phi)d\mu_2(\phi) \right).
\end{align*}
Then using Proposition \ref{Lem-meas}-(i) since $H_{\Omega_0}$ is normalized and $\lambda_n(\Omega_0)>0$, we find
\begin{align*}
\lambda_n^\prime(\Omega_0)\geqslant &\mathtt{M}^{-1}\lambda_n(\Omega_0)>0.
\end{align*}
This implies  that 
$ \lambda_n^\prime(\Omega_0)>0$ for any $\Omega_0\in(\overline\Omega_2,\overline\Omega_2)$, and thus $\Omega\in (\overline\Omega_2,\overline\Omega_1)\mapsto \lambda_n(\Omega)$ is strictly increasing. Therefore the proof of Proposition \ref{prop-operatorV2} is now achieved.}
 \end{proof}
Next we shall establish  the following result.
\begin{pro}\label{prop-kernel-onedim}
Let $n\geqslant 1$  and $r_{0,1},r_{0,2}$ satisfy the assumptions ${\bf{(H)}}$. Set
\begin{equation}\label{set-eigenvalues}
\mathscr{S}_n:=\Big\{\Omega\in[\overline\Omega_2,\overline\Omega_1] \quad\textnormal{s.t.}\quad \lambda_n(\Omega)=1\Big\}.
\end{equation} Then the following holds true
\begin{enumerate}
\item There exists $n_0\in\N$ such that for any $n\geqslant n_0$, the set $\mathscr{S}_n$ is nonempty and contains a single point denoted by   $\Omega_n$.
\item The  sequence $(\Omega_n)_{n\geqslant n_0}$ is strictly increasing and satisfies 
$$
\lim_{n\to\infty}\Omega_n=\overline\Omega_1.
$$
\end{enumerate}

\end{pro}
\begin{proof}
\medskip
\noindent
{\bf (1)}
To verify that the set $\mathscr{S}_n$ is non empty for large $n$ we shall use the mean value theorem.
 From \eqref{Est-kern-PM} we find that 
 $$
 0<\lambda_n(\Omega)\leqslant\sup_{i,j}|\overline\Omega_i-\Omega
|^{-\frac12}|\overline\Omega_j-\Omega
|^{-\frac12} n^{-\gamma},
 $$
for any $\gamma\in(0,\frac14)$. Thus by  taking $\Omega=\Omega_{\textnormal{av}}=\frac{\overline\Omega_1+\overline\Omega_2}{2}$ we deduce that 
$$
 0<\lambda_n(\Omega_{\textnormal{av}})\leqslant\tfrac{{2}}{\overline\Omega_1-\overline\Omega
_2} n^{-\gamma}.
 $$
 Consequently, there exists $n_0$ such that for any $n\geqslant n_0$
\begin{equation}\label{liminfTT0}
0\leq\lambda_n(\Omega_{\textnormal{av}})\leqslant\tfrac12\cdot
\end{equation}
Next, we intend to show that 
\begin{equation}\label{liminfPP}
\lim_{\Omega\to \overline\Omega_1}\lambda_n(\Omega)=\infty.
\end{equation}
Using the lower bound of $\lambda_n(\Omega)$ in Proposition \ref{prop-operatorV2}--(2),  we find by virtue of Fatou Lemma
\begin{equation*}
{d_1}\bigintsss_0^\pi\bigintsss_0^\pi\frac{H_{1,1}^n(\phi,\varphi)}{\nu_{1,\overline\Omega_1}^{\frac12}(\varphi)\nu_{1,\overline\Omega_1}^{\frac12}(\phi)}\frac{\sin^\frac12(\phi){ r_{0,1}(\phi)}}{\sin^\frac12(\varphi){ r_{0,1}(\varphi)}} \varrho(\phi)\varrho(\varphi)d\varphi d\phi\leqslant \liminf_{\Omega\rightarrow \overline\Omega_1} \lambda_n(\Omega),
\end{equation*}
for any nonnegative $\varrho$ satisfying $\displaystyle{\int_0^\pi\varrho^2(\phi)d\phi=1}$.  According to Proposition \ref{Lem-meas}--$(3)$, the function $\nu_{1,\overline\Omega_1}$ reaches its minimum at a point $\phi_0\in[0,\pi]$ and
$$
\forall\, \phi\in[0,\pi],\quad 0\leqslant \nu_{1,\overline\Omega_1}(\phi)\leq C|\phi-\phi_0|^{1+\alpha}.
$$
There are two possibilities: $\phi_0\in(0,\pi)$ or $\phi_0\in\{0,\pi\}$. Let us start with the first case and we shall  take $\varrho$ as follows
$$
\varrho(\phi)=\frac{c_\beta}{|\phi-\phi_0|^\beta},
$$
with $\beta<\frac12$ and the constant $c_\beta$ is chosen such that $\varrho$ is normalized. Hence using the preceding estimates  we get
\begin{equation}\label{Tramb1}
C\bigintsss_0^\pi\bigintsss_0^\pi\frac{H_{1,1}^n(\phi,\varphi)}{|\phi-\phi_0|^{\frac{1+\alpha}{2}+\beta}|\varphi-\phi_0|^{\frac{1+\alpha}{2}+\beta}}\frac{\sin^\frac12(\phi) r_1(\phi)}{\sin^\frac12(\varphi) r_1(\varphi)} d\varphi d\phi\leqslant\liminf_{\Omega\rightarrow\overline\Omega_1} \lambda_n(\Omega).
\end{equation}
Let $\varepsilon>0$ such that $[\phi_0-\varepsilon,\phi_0+\varepsilon]\subset (0,\pi)$. According to \eqref{HMM11} the function    $H_{1,1}^n$ is strictly positive in the domain $(0,\pi)^2$, hence  there exists $\delta>0$ such 
$$
\forall\, (\phi,\varphi)\in[\phi_0-\varepsilon,\phi_0+\varepsilon]^2, \quad \frac{{ H^n_{1,1}(\phi,\varphi)}\sin^\frac12(\phi){r_{0,1}}(\phi)}{\sin^\frac12(\varphi){ r_{0,1}}(\varphi)}\geqslant \delta.
$$
Thus we obtain
\begin{equation*}
C\bigintsss_{\phi_0-\varepsilon}^{\phi_0+\varepsilon}\bigintsss_{\phi_0-\varepsilon}^{\phi_0+\varepsilon}\frac{d\phi\,d\varphi}{|\phi-\phi_0|^{\frac{1+\alpha}{2}+\beta}|\varphi-\phi_0|^{\frac{1+\alpha}{2}+\beta}}\leqslant \liminf_{\Omega\rightarrow \overline\Omega_1} \lambda_n(\Omega).
\end{equation*}
By taking $\frac{1+\alpha}{2}+\beta>1$, which is an admissible configuration, we
find
$$
\lim_{\Omega\rightarrow \overline\Omega_1} \lambda_n(\Omega)=\infty.
$$
Next, let us move to the second alternative   where $\phi_0\in\{0,\pi\}$ and without 
any loss of generality we can only deal with the case $\phi_0=0.$ From \eqref{HMM11} and using the inequality
$$
\forall \, x\in[0,1),\quad F_n(x)\geqslant 1,
$$
we obtain
\begin{align*}
\forall \phi,\varphi\in(0,\pi),\quad H_{1,1}^n(\phi,\varphi)\geqslant&c_n\frac{\sin(\varphi){r_{0,1}^{n-1}(\phi)}{r_{0,1}^{n+1}}(\varphi)}{\left[R_{1,1}(\phi,\varphi)\right]^{n+\frac12}}.
\end{align*}
Combined with the assumption ${\bf{(H2)}}$, it implies
\begin{align*}
\forall \phi,\varphi\in(0,\pi),\quad H_{1,1}^n(\phi,\varphi)\geqslant&c_n\frac{\sin^{n+2}(\varphi)\sin^{n-1}(\phi)}{\left[R_{1,1}(\phi,\varphi)\right]^{n+\frac12}}.
\end{align*}
 Plugging this into \eqref{Tramb1} we find
\begin{equation*}
C_n\bigintsss_0^\pi\bigintsss_0^\pi\frac{1}{\phi^{\frac{1+\alpha}{2}+\beta}\varphi^{\frac{1+\alpha}{2}+\beta}} \frac{\sin^{n+\frac12}(\varphi)\sin^{n+\frac12}(\phi)}{\left[R_{1,1}(\phi,\varphi)\right]^{n+\frac12}}d\varphi d\phi\leqslant \liminf_{\Omega\rightarrow \overline\Omega_1} \lambda_n(\Omega).
\end{equation*}
By virtue of ({\bf{H2}}) and the mean value theorem we find according to \eqref{Rij},
$$
\forall \phi,\varphi\in[0,\tfrac\pi2],\quad  R_{1,1}(\phi,\varphi)\leqslant C(\phi+\varphi)^2.
$$
Thus
 \begin{equation*}
C_n\bigintsss_0^{\frac\pi2}\bigintsss_0^{\frac\pi2}\frac{1}{\phi^{\frac{1+\alpha}{2}+\beta}\varphi^{\frac{1+\alpha}{2}+\beta}} \frac{\varphi^{n+\frac12}\phi^{n+\frac12}}{(\phi+\varphi)^{2n+1}}d\varphi d\phi\leqslant \liminf_{\Omega\rightarrow \overline\Omega_1} \lambda_n(\Omega)
\end{equation*}
which yields  after simplification
\begin{equation*}
C_n\bigintsss_0^{\frac\pi2}\bigintsss_0^{\frac\pi2}\frac{\varphi^{n-\frac\alpha2-\beta}\phi^{n-\frac\alpha2-\beta}}{(\phi+\varphi)^{2n+1}}d\varphi d\phi\leqslant \liminf_{\Omega\rightarrow \overline\Omega_1} \lambda_n(\Omega).
\end{equation*}
Making the change of variables $\varphi=\phi \theta$ we obtain
\begin{align*}
\bigintsss_0^{\frac\pi2}\bigintsss_0^{\frac\pi2}\frac{\varphi^{n-\frac\alpha2-\beta}\phi^{n-\frac\alpha2-\beta}}{(\phi+\varphi)^{2n+1}}d\varphi d\phi=&\bigintsss_0^{\frac\pi2}\phi^{-\alpha-2\beta}\bigintsss_0^{\frac{\pi}{2\phi}}\frac{\theta^{n-\frac\alpha2-\beta}}{(1+\theta)^{2n+1}}d\theta d\phi.
\end{align*}
This integral diverges provided that $\alpha+2\beta>1$ and thus  under this admissible assumption we obtain \eqref{liminfPP}.
By the { intermediate value theorem} , we achieve in view of \eqref{liminfTT0} the existence of at least one solution in the range $(\overline\Omega_2,\overline\Omega_1)$  for the equation  
$$
\lambda_n(\Omega)=1,
$$
provided that $n\geqslant n_0$.
Consequently, using Proposition \ref{prop-operatorV2}-$(5)$ we deduce  that  the set $\mathscr{S}_n$ contains only one single element for any  $n\geqslant n_0$.

\medskip
\noindent
{\bf (2)} By the preceding construction  $\Omega_n$ satisfies  for  $n\geqslant n_0$ the equation
$$
\lambda_n(\Omega_n)=1.
$$
According to Proposition \ref{prop-operatorV2}--$(4)$ the sequence $k\mapsto \lambda_k(\Omega_n)$ is strictly decreasing. It implies in particular that
$$
\lambda_{n+1}(\Omega_n)<\lambda_n(\Omega_n)=1.
$$
Hence by \eqref{liminfPP} one may apply { the intermediate value theorem} and find  an element of the set $\mathscr{S}_{n+1}$
in the interval $(\Omega_n,\overline\Omega_1)$. This means that $\overline\Omega_1>\Omega_{n+1}>\Omega_n$ and thus this sequence is strictly increasing and therefore it converges to  some element $\overline{\Omega}_2<\overline{\Omega}\leqslant \overline\Omega_1.$
It remains to prove that $\overline{\Omega}=\overline\Omega_1.$  For this purpose, we shall argue by contradiction by assuming that $\overline{\Omega}_2<\overline{\Omega}<\overline\Omega_1.$
By Proposition \ref{prop-operatorV2}--$(5),$ we know that  $\Omega\mapsto \lambda_n({\Omega})$ is strictly increasing and thus  
\begin{align}\label{low-contra1} 
\forall n\geq n_0,\quad \lambda_n(\overline{\Omega})> \lambda_n({\Omega}_n)=1.
\end{align}
Using the upper-bound estimate stated in Proposition \ref{prop-operatorV2}--$(2)$ 
we obtain for any $\gamma\in(0,\tfrac14)$
\begin{align*}
\forall\, n\geq1,\quad 0< \lambda_n(\overline{\Omega})&\leqslant C  n^{-\gamma}\max_{i,j}\left|(\overline\Omega-\overline\Omega_i)(\overline\Omega-\overline\Omega_j)\right|^{-\frac12}\\
&\leqslant C  n^{-\gamma}.
\end{align*}
By taking the limit as $n\to\infty$ we find
$$
\lim_{n\to\infty}\lambda_n(\overline{\Omega})=0.
$$
This contradicts \eqref{low-contra1} and the proof is complete.
\end{proof}
\subsection{Regularity of the eigenfunctions}\label{Reg-EigenF}
This main concern of this section is to perform  strong regularity  of the eigenfunctions associated to the operator $\mathcal{T}_\Omega^n$ and  
constructed in Lemma \ref{lem-ssym} and  \mbox{Proposition \ref{prop-operatorV2}.} Notice that we have already  seen that these eigenfunctions belongs to a weak function space $\mathbb{H}_{\Omega}$. 
In what follows,  we shall prove  first their continuity and later their H\"{o}lder regularity.

\subsubsection{Continuity}
The main goal is to establish the following result. 
\begin{pro}\label{prop-higher-regM}
Let $\Omega\in(\overline\Omega_2,\overline\Omega_1),$  $r_{0,1},r_{0,2}$ satisfy the assumptions ${\bf{(H)}}$. Let  $H=(h_1,h_2)$ be
an eigenfunction for $\mathcal{T}^n_{{\Omega}}$ associated to a non-vanishing  eigenvalue, for $n\geqslant2$. Then $H$ is continuous
\mbox{on $[0,\pi]$}  and  satisfies the boundary condition $H(0)=H(\pi)=0$.  
\end{pro}
 
\begin{proof}
Let $H=(h_1,h_2)\in \mathbb{H}_{{\Omega}}$ be any non trivial eigenfunction of the operator  $\mathcal{T}^n_{{\Omega}}$ defined through the coupled integral equations  \eqref{Tnomega2} and associated to an eigenvalue $\lambda\neq0$, then for almost $\varphi\in[0,\pi],$

\begin{align}\label{Tnomega1}
d_1\int_0^{\pi} \tfrac{H_{1,1}^n(\phi,\varphi)}{\nu_{1,\Omega}(\phi)}h_1(\varphi)d\varphi-d_2\int_0^\pi \tfrac{H_{1,2}^n(\phi,\varphi)}{\nu_{1,\Omega}(\phi)}h_2(\varphi)d\varphi&=\lambda h_1(\phi)\nonumber \\
-d_1\int_0^{\pi} \tfrac{H_{2,1}^n(\phi,\varphi)}{\nu_{2,\Omega}(\phi)}h_1(\varphi)d\varphi+d_2\int_0^\pi \tfrac{H_{2,2}^n(\phi,\varphi)}{\nu_{2,\Omega}(\phi)}h_2(\varphi)d\varphi&=\lambda h_2(\phi).
\end{align}
From \eqref{MeasureB-1} combined with the fact 
that  $r_{0,i}\in L^2_{\mu_i}$ and the assumptions $({\bf{H}})$ we deduce that  the functions $\mathtt{h}_1: \varphi\in [0,\pi]\mapsto \sin^{\frac32}(\varphi)h_1(\varphi)$ and $\mathtt{h}_2: \varphi\in [0,\pi]\mapsto \sin^{\frac32}(\varphi)h_1(\varphi) $ belong to  $ L^2((0,\pi);d\varphi)$. Therefore the equations \eqref{Tnomega1} can be written in terms of $\mathtt{h}_1$ and $\mathtt{h}_2$ as follows
\begin{align}\label{Tnomega2-0}
\tfrac{d_1}{\nu_{1,\Omega}(\phi)}\int_0^{\pi} \tfrac{H_{1,1}^n(\phi,\varphi)\sin^{\frac32}(\phi)}{\sin^{\frac32}(\varphi)}\mathtt{h}_1(\varphi)d\varphi-\tfrac{d_2}{\nu_{1,\Omega}(\phi)}\int_0^\pi \tfrac{H_{1,2}^n(\phi,\varphi)\sin^{\frac32}(\phi)}{\sin^{\frac32}(\varphi)}\mathtt{h}_2(\varphi)d\varphi&=\lambda \mathtt{h}_1(\phi),\nonumber \\
-\tfrac{d_1}{\nu_{2,\Omega}(\phi)}\int_0^{\pi} \tfrac{H_{2,1}^n(\phi,\varphi)\sin^{\frac32}(\phi)}{\sin^{\frac32}(\varphi)}\mathtt{h}_1(\varphi)d\varphi+\tfrac{d_2}{\nu_{2,\Omega}(\phi)}\int_0^\pi \tfrac{H_{2,2}^n(\phi,\varphi)\sin^{\frac32}(\phi)}{\sin^{\frac32}(\varphi)}\mathtt{h}_2(\varphi)d\varphi&=\lambda \mathtt{h}_2(\phi).
\end{align}
According  to the definition of $H_{i,j}^n$ in \eqref{HMM11} combined with the  assumption {\bf{(H2)}} we obtain for some constant $c_n$ the formula
 \begin{align*}
\sin^{-\frac32}(\varphi){\sin^{\frac32}(\phi)}H^n_{i,j}(\phi,\varphi)\leqslant &c_n\frac{r_{0,i}^{n+\frac12}(\phi)r_{0,j}^{n+\frac12}(\varphi)}{\left[{R_{i,j}}(\phi,\varphi)\right]^{n+\frac12}} F_n\left(\frac{4{r_{0,i}}(\phi){r_{0,j}}(\varphi)}{R_{i,j}(\phi,\varphi)}\right).
\end{align*}
Using \eqref{estimat-1} and the assumptions {\bf{(H2)}}- {\bf{(H4)}}  yields
\begin{align}\label{H-XL1}
\sin^{-\frac32}(\varphi){\sin^{\frac32}(\phi)}H_{i,j}^n(\phi,\varphi)\lesssim&\frac{{r_{0,i}^{n+\frac12}}(\phi){r_{0,j}^{n+\frac12}}(\varphi)}{{R_{i,j}}^{n+\frac12}(\phi,\varphi)}\left(1+\ln\left(\frac{\sin(\phi)+\sin(\varphi)}{|\phi-\varphi|}\right) \right)\\
\nonumber\lesssim&1+\ln\left(\frac{\sin(\phi)+\sin(\varphi)}{|\phi-\varphi|}\right). 
\end{align}
Then inserting this estimate into \eqref{Tnomega2-0} and  using Cauchy-Schwarz inequality  combined with  the fact that $\nu_{i,\Omega}$ is  bounded  bounded away from zero according to Proposition \ref{Lem-meas}-(1)
\begin{align*}
 (|\mathtt{h}_1|+|\mathtt{h}_2|)(\phi)\lesssim& \bigintsss_0^\pi\left(1+\ln\left(\frac{\sin(\phi)+\sin(\varphi)}{|\phi-\varphi|}\right)\right)(|\mathtt{h}_1|+|\mathtt{h}_2|)(\varphi)d\varphi\\
 \lesssim& \||\mathtt{h}_1|+|\mathtt{h}_2|\|_{L^2(d\varphi)}\left(1+\bigintsss_0^\pi\ln^2\left(\frac{\sin(\phi)+\sin(\varphi)}{|\phi-\varphi|}\right)d\varphi\right)^{\frac12}\\
  \lesssim& \|H\|\quad \forall \phi\in (0,\pi)\,\textnormal{a.e.}
\end{align*}
It follows that $(\mathtt{h}_1,\mathtt{h}_2)$ is bounded. Therefore, we infer from  \eqref{Tnomega1}, \eqref{HMM11} and  {\bf{(H2)}} 
\begin{align}\label{split}
 |H(\phi)|\lesssim&\|(\mathtt{h}_1,\mathtt{h}_2)\|_{L^\infty}\sup_{i,j}\int_0^\pi  \sin^{-\frac32}(\varphi)H_{i,j}^n(\phi,\varphi) d\varphi,\nonumber\\
 \lesssim&  \|H\|\sup_{i,j}\bigintss_0^\pi\frac{{\sin^{n+\frac12}}(\varphi)\sin^{n-1}(\phi)}{\left[R_{i,j}(\phi,\varphi)\right]^{n+\frac12}} F_n\left(\frac{4{r_{0,i}}(\phi){r_{0,j}}(\varphi)}{R_{i,j}(\phi,\varphi)}\right)d\varphi.
\end{align}
Using once again \eqref{estimat-1} and the assumptions {\bf{(H2)}}- {\bf{(H4)}}  we deduce that
 \begin{align*}
 |H(\phi)|\lesssim& \|H\|\bigintss_0^\pi\frac{\sin^{n-1}(\phi)\sin^{n+\frac12}(\varphi)}{\big((\sin(\phi)+\sin(\varphi))^2+(\cos\phi-\cos\varphi)^2\big)^{n+\frac12}}\Big(1+\ln\Big(\frac{\sin(\phi)+\sin(\varphi)}{|\phi-\varphi|}\Big)\Big)d\varphi.
\end{align*}
By symmetry we may  restrict the analysis to $\phi\in[0,\frac\pi2]$. Thus, from elementary trigonometric identities  we get

\begin{eqnarray*}
\inf_{\varphi\in[\pi/2,\pi]
\atop\phi\in[0,\pi/2]}\big[(\sin(\phi)+\sin(\varphi))^2+(\cos\phi-\cos\varphi)^2&=&{\inf_{\varphi\in[\pi/2,\pi]
\atop\phi\in[0,\pi/2]}}2(1-\cos(\phi+\varphi))\big]\\
&\geqslant& 2.
\end{eqnarray*}
It follows that
\begin{align*}
 |H(\phi)|\lesssim& \|H\|\bigintss_0^{\frac\pi2}\frac{\sin^{n-1}(\phi)\sin^{n+\frac12}(\varphi)}{\big(\sin(\phi)+\sin(\varphi)
 \big)^{2n+1}}\Big(1+\ln\Big(\frac{\sin(\phi)+\sin(\varphi)}{|\phi-\varphi|}\Big)\Big)d\varphi\\
 +&\|H\|\bigintss_{\frac\pi2}^\pi\Big(1+\ln\Big(\frac{\sin(\phi)+\sin(\varphi)}{|\phi-\varphi|}\Big)\Big)d\varphi.
\end{align*}
Thus
\begin{align*}
 |H(\phi)|\lesssim& \|H\|\bigintss_0^{\frac\pi2}\frac{\phi^{n-1}\varphi^{n+\frac12}}{\big(\phi+\varphi\big)^{2n+1}}\Big(1+\ln\Big(\frac{\phi+\varphi}{|\phi-\varphi|}\Big)\Big)d\varphi
 +\|H\|.
\end{align*}
Applying  the change of variables $\varphi=\phi \theta$ we get
\begin{align*}
 |H(\phi)|\lesssim& \|H\|_{\Omega}\phi^{-\frac12}\bigintss_0^{\frac{\pi}{2\phi}}\frac{\theta^{n+\frac12}}{\big(1+\theta\big)^{2n+1}}\Big(1+\ln\Big(\frac{1+\theta}{|1-\theta|}\Big)\Big)d\theta
 +\|H\|\\
 \lesssim& \|H\|\phi^{-\frac12}.
\end{align*}
Consequently we find 
$$
\sup_{\phi\in(0,\pi)}\sin^{\frac12}(\phi)| H(\phi)|\lesssim  \|H\|.
$$
Inserting  this estimate into \eqref{Tnomega1} and using the first line in \eqref{H-XL1} yields
\begin{align*}
|H(\phi)|\lesssim& \|H\| \sup_{i,j}\bigintsss_0^\pi \sin^{-\frac12}(\varphi) H_{i,j}^n(\phi,\varphi) d\varphi\\
\lesssim& \|H\|  \sup_{i,j}\bigintss_0^\pi\frac{\sin^{n+\frac32}(\varphi)\sin^{n-1}(\phi)}{R_{i,j}^{n+\frac12}(\phi,\varphi)}\left(1+\ln\left(\frac{\sin(\phi)+\sin(\varphi)}{|\phi-\varphi|}\right) \right) d\varphi.
\end{align*}
As before we can restrict $\phi\in [0,\frac\pi2]$ and by using the fact
$$
\inf_{\varphi\in[\pi/2,\pi]\\
\atop\phi\in[0,\pi/2]}R_{i,j}(\phi,\varphi)>0,
$$
we deduce by  splitting the integral
\begin{align*}
|H(\phi)|\lesssim&  \|H\|\sin^{n-1}(\phi)+\|H\| \bigintss_0^{\frac\pi2}\frac{\varphi^{n+\frac32}\phi^{n-1}}{(\phi+\varphi)^{2n+1}}\left(1+\ln\left(\frac{\phi+\varphi}{|\phi-\varphi|}\right) \right) d\varphi.
\end{align*}
Applying  the change of variables $\varphi=\phi \theta$ leads for $n\geq2$ to 
\begin{align*}
\forall \phi\in[0,\pi/2],\quad |H(\phi)|\lesssim& \|H\| \sin^{n-1}(\phi)+\|H\|\phi^{\frac12} \bigintss_0^{\frac{\pi}{2\phi}}
\frac{\theta^{n+\frac32}}{(1+\theta)^{2n+1}}\left(1+\ln\left(\frac{\theta+1}{|\theta-1|}\right) \right) d\theta\\
\lesssim& \|H\|\big(\phi^{n-1}+{ \phi^{\frac{1}{2}}}\big).
\end{align*}
Consequently we get
\begin{align*}
\forall \phi\in(0,\pi),\quad |H(\phi)|\lesssim& \|H\|\big( \sin^{n-1}(\phi)+ \sin^{\frac12}(\phi)\big).
\end{align*}
Therefore, we find that $H$ is bounded on  $(0,\pi)$ and using  the dominated convergence theorem one may  show that $H=(h_1,h_2)$ is in fact continuous on $[0,\pi]$ and satisfies for $n\geqslant2$ 
the boundary condition
$$
H(0)=H(\pi)=0.
$$
This ends the proof of Proposition \ref{prop-higher-regM}.
\end{proof}

\subsubsection{H\"older continuity of the eigenvalues}
We shall focus in   this section on  the H\"{o}lder regularity of the eigenfunctions that will be proved using classical arguments from potential theory as in \cite{GHM}.
\begin{pro}\label{prop-holder}
Let $n\geqslant 2$, $\Omega\in(\overline\Omega_2,\overline\Omega_1),$  $r_{0,1},r_{0,2}$ satisfy the assumptions {\bf{(H)}}. Then any solution  $H=(h_1,h_2)$ to
\begin{equation}\label{eigenf-eq}
\mathcal{T}^n_{\Omega} H(\phi)=\lambda H(\phi), \quad\phi\in(0,\pi),
\end{equation}
 with $\lambda\neq 0$, belongs to $\mathscr{C}^{1,\alpha}(0,\pi)$, where $\mathcal{T}^n_{\Omega}=(\mathcal{T}^n_{1,\Omega},\mathcal{T}^n_{2,\Omega})$ is defined in \eqref{Tnomega2}.
\end{pro}
\begin{proof}
In this proof, we shall privilege  the expression of $\mathcal{T}^n_\Omega$ detailed in Proposition \ref{Prop-MMlin1} instead of \eqref{Tnomega2}, and then  perform  some classical potential theory estimates. We first write
\begin{align*}
\mathcal{T}_{i,\Omega}^n(\phi)=&\frac{(-1)^{i-1}}{\nu_{i,\Omega}(\phi)}\left\{\int_0^\pi\int_0^{2\pi}\mathcal{H}^n_{i,1}(\phi,\varphi,\eta)h_1(\varphi)d\eta d\varphi-\int_0^\pi\int_0^{2\pi}\mathcal{H}^n_{i,2}(\phi,\varphi,\eta)h_2(\varphi)d\eta d\varphi\right\}\\
:=&\frac{1}{\nu_{i,\Omega}(\phi)}\mathcal{F}^n_{i}(h_1,h_2)(\phi),
\end{align*}
and $\mathcal{H}^n_{i,j}$ is defined in \eqref{Hcal}.
Let us denote $\mathcal{F}^n=\begin{pmatrix}
\mathcal{F}^n_1\\
\mathcal{F}^n_2
\end{pmatrix}$, then it is obvious that   \eqref{eigenf-eq} is equivalent to
$$
\begin{pmatrix}
h_1\\
h_2
\end{pmatrix}=\frac{1}{\lambda} \begin{pmatrix}
\frac{1}{\nu_{1,\Omega}(\phi)}\mathcal{F}^n_{1}(h_1,h_2)(\phi)\\
\frac{1}{\nu_{2,\Omega}(\phi)}\mathcal{F}^n_{2}(h_1,h_2)(\phi)
\end{pmatrix}.
$$
Furthermore, we have seen in  Proposition \ref{Lem-meas}  that $\nu_{i,\Omega}\in\mathscr{C}^{1,\alpha}(0,\pi)$ and it does not vanish if $\Omega\in(\overline{\Omega}_2,\overline{\Omega}_1)$. Hence, we should check that $\mathcal{F}^n_i(h_1,h_2)\in\mathscr{C}^{1,\alpha}(0,\pi)$ in order to achieve the regularity result.
Applying \eqref{A11} and the  assumption {\bf (H4)} on $r_{0,i}$ we find for  $i\neq j,$
\begin{align}\label{reg-eigen-1}
A_{i,j}(\phi,\varphi,\eta)
\geqslant &(r_{0,i}(\phi)-r_{0,j}(\varphi))^2+(d_i\cos(\phi)-d_j\cos(\varphi))^2\geqslant  \delta>0.
\end{align}
On the other hand, from \cite[Proposition 4.5, (4.108)]{GHM} we get 
\begin{equation}\label{reg-eigen-2}
A_{i,i}(\phi,\varphi,\eta)\geqslant C\left\{(\phi-\varphi)^2+(\sin^2(\phi)+\sin^2(\varphi))\sin^2(\eta/2)\right\}, 
\end{equation}
which is a consequence of {\bf (H2)}. Next, let us decompose $\mathcal{F}^n_i$ as follows
\begin{align*}
\mathcal{F}^n_{i}(h_1,h_2)=&{(-1)^{i-1}}\mathcal{F}^n_{i,1}(h_1)+{(-1)^i}\mathcal{F}^n_{i,2}(h_2),\\
\mathcal{F}^n_{i,j}(h_j):=&\int_0^\pi\int_0^{2\pi}\mathcal{H}_{i,j}^n(\phi,\varphi,\eta)h_j(\varphi)d\eta d\varphi.
\end{align*}
Notice that  the terms  generated from the induced effects of each boundary, corresponding to $\mathcal{F}^n_{i,i}$,  are similar to the paper   \cite{GHM}  and hence we can borrow and perform the same analysis  with slight adaptation, however  the interaction terms are smooth  enough since the kernel does not admit a diagonal singularity according to \eqref{reg-eigen-1}. Below we give some details where the regularity is shown by taking advantage of an elliptic smoothing effect from  the integral representation. We point out that we shall start with the result of Proposition \ref{prop-higher-regM} stating that $H=(h_1,h_2)\in L^\infty(0,\pi)$. More precisely,  $H$ is continuous on $[0,\pi]$  and satisfies the Dirichlet boundary \mbox{condition $H(0)=H(\pi)=0.$}
\vspace{0.3cm}
\\
{
\noindent\medskip \ding{228} {\bf Self-induced terms.} They correspond to the terms $\mathcal{F}^n_{i,i}(h_i)$. Their treatment is detailed in the proof of \cite[Proposition 4.5]{GHM} where we first show that   the smoothing effect  $h_i\in L^\infty$ implies  $\mathcal{F}^n_{i,i}(h_i)\in\mathscr{C}^\alpha(0,\pi)$ which in turns implies after the  differentiation that $\mathcal{F}^n_{i,i}(h_i)\in\mathscr{C}^{1,\alpha}(0,\pi)$.
\\
\noindent\medskip \ding{228} {\bf Interaction terms.} They correspond to the terms $\mathcal{F}^n_{i,j}(h_j)$ for $i\neq j.$
By \eqref{reg-eigen-1} and  \eqref{Hcal}, the denominator of the integral is never vanishing inside $(0,\pi)$ and the singularities are located  at the boundary $\phi=0, \pi$ due to the term $r_{0,i}$. To remove these singularities, we integrate by parts in the variable $\eta$ leading to,
\begin{align*}
\mathcal{F}^n_{i,j}(h_j)(\phi)=-\frac{1}{n}\bigintsss_0^\pi\bigintsss_0^{2\pi}\frac{\sin(\varphi) r_{0,j}(\varphi)^2\sin(n\eta)\sin(\eta)h_j(\varphi)}{|A_{i,j}(\phi,\varphi,\eta)|^\frac32}d\eta d\varphi.
\end{align*}
Therefore  using \eqref{reg-eigen-1} we can easily check t that $\mathcal{F}^n_{i,j}(h_j)\in\mathscr{C}^{\alpha}$ for $i\neq j$. 
Now, coming back to \eqref{eigenf-eq} we infer  that $H\in\mathscr{C}^\alpha(0,\pi)$.\\
Next, we shall check that $h_j\in\mathscr{C}^\alpha(0,\pi)$ ( $L^\infty$ is enough) implies under the assumption {({\bf{H1}})}
$$\big(\mathcal{F}^n_{i,j}(h_j)\big)^\prime, \big(\mathcal{F}^n_{i,j}(h_j)\big)^{\prime\prime}\in L^\infty(0,\pi),\quad i\neq j.
$$
which is of course more stronger that $\mathcal{F}^n_{i,j}(h_j)\in\mathscr{C}^{1,\alpha}(0,\pi)$ according to Sobolev embeddings.
By differentiation inside the integral, we deduce that
\begin{align*}
\mathcal{F}^n_{ij}(h_j)'(\phi)=\frac{3}{2n}\bigintsss_0^\pi \bigintsss_0^{2\pi}\frac{\sin(\varphi)r_{0,j}(\varphi)^2\sin(n\eta)\sin(\eta)h_j(\varphi)\partial_\phi A_{ij}(\phi,\varphi,\eta)}{|A_{ij}(\phi,\varphi,\eta)|^\frac52}\cdot
\end{align*}
Then by virtue of   \eqref{reg-eigen-1} we easily get $\mathcal{F}^n_{i,j}(h_j)'\in L^\infty(0,\pi)$. A second differentiation combined with  \eqref{reg-eigen-1} and the assumption {({\bf{H1}})}  yield  $\mathcal{F}^n_{i,j}(h_j)^{\prime\prime}\in L^\infty(0,\pi)$. Finally, we deduce from the preceding steps  that $\mathcal{F}^n H \in\mathscr{C}^{1,\alpha}(0,\pi),$ that we combine with the equation \eqref{eigenf-eq} in order to get  $H \in\mathscr{C}^{1+\alpha}(0,\pi).$ This ends the proof of the desired results.
}
\end{proof}

\subsection{Fredholm structure}
The main purpose of this section is to establish  the Fredholm structure of the linearized operator defined through Proposition \ref{Prop-MMlin1} and required by Theorem \ref{CR}. Our   result reads as follows.

\begin{pro}\label{prop-fredholmTT} 
Let  $m\geqslant2,\Omega\in(\overline\Omega_2,\overline\Omega_1),$  $r_{0,1},r_{0,2}$ satisfy the assumptions {\bf{(H)}}. 
Then the  operator
$$\partial_{f_1,f_2}\tilde{F}(\Omega,0,0):X_m^\alpha\times X_m^\alpha\rightarrow X_m^\alpha\times X_m^\alpha
$$ is a Fredholm operator with zero index. {Moreover, there exists $m_0\in\N$ such that for  $m\geqslant m_0$ and  for $\Omega=\Omega_m$, the dimension of its  kernel and the codimension of its range  are equal to one.}
\end{pro}
\begin{proof}
{Let us recall that $\Omega_m$ is defined in \eqref{prop-kernel-onedim}. The idea is to write the linear operator $\partial_{f_1,f_2}\tilde{F}(\Omega,0,0)$ as a compact perturbation of an isomorphism. Indeed, by virtue of Proposition \ref{Prop-MMlin1} it can be decomposed as
\begin{align*}
\partial_{f_1,f_2}\tilde{F}(\Omega,0,0)(h_1,h_2)(\phi,\theta)=&\sum_{n\geqslant 1}\cos(n\theta)\Big(\mathcal{L}_{1,\Omega}(h_{1,n},h_{2,n})(\phi)-\mathcal{L}^n_2(h_{1,n},h_{2,n})(\phi)\Big)\\
:=&\sum_{n\geqslant 1}\cos(n\theta)\mathcal{L}^n_\Omega (h_{1,n},h_{2,n})(\phi),
\end{align*}
where 
\begin{align*}
\mathcal{L}_{1,\Omega}(h_{1,n},h_{2,n}):=&\begin{pmatrix}
\nu_{1,\Omega}(\phi)h_{1,n}(\phi)\\
- \nu_{2,\Omega}(\phi)h_{2,n}(\phi)
\end{pmatrix},\\
\mathcal{L}_2^n(h_{1,n},h_{2,n}):=
&\begin{pmatrix}
\nu_{1,\Omega}\mathcal{T}^n_{1,\Omega}(h_{1,n},h_{2,n})\\
-\nu_{2,\Omega}\mathcal{T}^n_{2,\Omega}(h_{1,n},h_{2,n})
\end{pmatrix}.
\end{align*}
By Proposition \ref{Lem-meas} we have that $\nu_{i,\Omega}$ is not vanishing if $\Omega\in(\overline{\Omega}_2,\overline{\Omega}_1)$ and belongs to the class $\mathscr{C}^{1,\alpha}(0,\pi)$. Then we can check that $\mathcal{L}_{1,\Omega}$ is an isomorphism as in the proof of \cite[Proposition 4.6]{GHM}. On the other hand, we infer from the proofs of Proposition \ref{prop-holder} and \cite[Proposition 4.6]{GHM}  that $\mathcal{F}^n(h_{1,n},h_{2,n})\in\mathscr{C}^{1,\beta}$ if ${h_{i,n}\in\mathscr{C}^{1,\alpha}}$, for any $\beta\in(0,1)$. In particular, taking $\beta>\alpha$ we have that the embedding 
$$
\mathcal{L}_2^n:X_m^\alpha\times X_m^\alpha\rightarrow \big(X_m^\alpha\cap \mathscr{C}^{1,\beta}\big)\times \big(X_m^\alpha\cap \mathscr{C}^{1,\beta}\big)\hookrightarrow X_m^\alpha\times X_m^\alpha,
$$
is compact since $X_m^\alpha\cap \mathscr{C}^{1,\beta}\hookrightarrow X_m^\alpha$ is  compact. Since compact perturbations of Fredholm operator remains Fredholm with the same index (in particular, an isomorphism is Fredholm of zero index), then we deduce  the first statement of the proposition.
\\
Fix $m\geqslant m_0:=n_0$, where $n_0$ was introduced in  Proposition \ref{prop-kernel-onedim}.  Then the kernel of $\partial_{f_1,f_2}\tilde{F}(\Omega_m,0,0):X_m^\alpha\times X_m^\alpha\rightarrow X_m^\alpha\times X_m^\alpha$ is generated by 
$$
(\theta,\phi)\mapsto H(\phi)\cos(n m\theta),\quad \hbox{with}\quad \mathcal{T}^{nm}_{\Omega_m}(H)=H
$$ 
for some $n\geqslant 1.$ 
Then, Proposition \ref{prop-kernel-onedim} gives us that the dimension of the kernel of $\mathcal{L}^{nm}_\Omega$ is one and is generated by the eigenfunction associated to the eigenvalue $\lambda_{nm}(\Omega_m)=1$. Notice that such eigenvalue belongs to $X_m^\alpha\times X_m^\alpha$ due to Proposition \ref{prop-holder}. Now, we have by construction $\lambda_{m}(\Omega_m)=1$ and from   {Proposition \ref{prop-operatorV2}-(4) } we get  for $n>1$,  $\lambda_{nm}(\Omega_m)<\lambda_m(\Omega_m)=1$. Therefore,  $1$ can not be an eigenvalue of $\mathcal{T}^{nm}_{\Omega_m}$, for $n>1.$ 
 Hence, the kernel of $\partial_{f_1,f_2}\tilde{F}(\Omega_m,0,0):X_m^\alpha\times X_m^\alpha\rightarrow X_m^\alpha\times X_m^\alpha$ is one-dimensional and it is generated by the eigenfunction
\begin{equation}\label{fmstar}
(\phi,\theta)\mapsto { f_m^\star(\phi,\theta):=H_m^\star(\phi)\cos(m\theta)},
\end{equation}
with $H_m^\star$ a nonzero vector satisfying
$$ \mathcal{T}^{nm}_{\Omega_m}H_m^\star=H_m^\star.$$
Finally, since $\partial_{f_1,f_2}\tilde{F}(\Omega_m,0,0)$ is a Fredholm operator of zero index, we get that the codimension of its range is one.
}
\end{proof}

\subsection{Transversality}
The transversal assumption which amounts to checking 
$$
\partial_\Omega \partial_{f_1,f_2} \tilde{F}(\Omega_m,0,0)f_m^\star\notin \textnormal{Im}(\partial_{f_1,f_2} \tilde{F}(\Omega_m,0,0)),
$$
where $f_m^\star$ is a generator  of the kernel of $\partial_{{f_1,f_2}} \tilde{F}(\Omega_m,0,0)$ given in \eqref{fmstar}, and $\Omega_m$ is given in Proposition \ref{prop-kernel-onedim}. By Proposition \ref{Prop-MMlin1} we get
$$
\partial_\Omega \partial_{f_1,f_2} \tilde{F}(\Omega,0,0)(h_1,h_2)=-(h_1,h_2).
$$
Our main result of this section reads as follows.
\begin{pro}\label{prop-transversal}
There exists $m_0\in\N$ such that for any  $m\geqslant m_0,$  the transversality condition holds true, that is, 
$$
\partial_\Omega \partial_{f_1,f_2} \tilde{F}(\Omega_m,0,0)f_m^\star\notin \textnormal{Im}(\partial_{f_1,f_2} \tilde{F}(\Omega_m,0,0)),
$$
where $f_m^\star$ is a generator  of the kernel of $\partial_{f_1,f_2} \tilde{F}(\Omega_m,0,0)$ given in \eqref{fmstar}.
\end{pro}
\begin{proof}
Recall from the proof of Proposition \ref{prop-fredholmTT}  that the function $f_m^\star$ has the form
$$
f_m^\star(\phi,\theta)=H_m^\star(\phi)\cos(m\theta),
$$
where $H_m^\star=(h_{1,m}^\star, h_{2,m}^\star)$ is a unit vector in $\mathbb{H}_{\Omega_m}$ satisfying  the equation
$$ 
\mathcal{T}^m_{\Omega_m} H_m^\star(\phi)=H_m^\star(\phi).
$$
It follows that 
$$
\partial_\Omega\partial_{f_1,f_2} \tilde{F}(\Omega_m,0,0)f_m^\star(\phi,\theta)=-H_m^\star(\phi)\cos(m\theta).
$$
Assume that this element belongs to the range of $\partial_{f_1,f_2} \tilde{F}(\Omega_m,0,0)$. Then we can find $H_m=(h_{1,m}, h_{2,m})$ such that
$$
\left(\tfrac{h_{1,m}^\star}{\nu_{1,\Omega_m}},-\tfrac{h_{2,m}^\star}{\nu_{2,\Omega_m}}\right)= H_{m}(\phi)-\mathcal{T}^m_{\Omega_m} (H_m)(\phi).
$$
 Taking the inner product of this vector with $H_{m}^{\star}$, with respect to $\langle\cdot ,\cdot \rangle_{\Omega_m}$ defined in \eqref{scalar-prod1}    yields by the symmetry of $ \mathcal{T}^m_{\Omega_m}$
\begin{align*}
\Big\langle \left(\tfrac{h_{1,m}^\star}{\nu_{1,\Omega_m}},-\tfrac{h_{2,m}^\star}{\nu_{2,\Omega_m}}\right),H_{m}^{\star}\Big\rangle_{\Omega_m}=&\Big\langle H_{m},H_{m}^{\star}\Big\rangle_{\Omega_m}-\Big\langle \mathcal{T}^m_{\Omega_m} (H_m),H_{m}^{\star}\Big\rangle_{\Omega_m}\\
=&\Big\langle H_{m},H_{m}^{\star}\Big\rangle_{\Omega_m}-\Big\langle  H_{m}, \mathcal{T}^m_{\Omega_m}(H_m^{\star})\Big\rangle_{\Omega_m}\\
=&\Big\langle H_{m},H_{m}^{\star}-\mathcal{T}^m_{\Omega_m}(H_m^{\star})\Big\rangle_{\Omega_m}\\
=&0.
\end{align*}
Coming back to the definition of the inner product \eqref{scalar-prod1} and \eqref{MeasureB-1}, we find
{\begin{align*}
\Big\langle \left(\tfrac{h_{1,m}^\star}{\nu_{1,\Omega_m}},-\tfrac{h_{2,m}^\star}{\nu_{2,\Omega_m}}\right),(h_{1,m}^\star,h_{2,m}^\star\Big\rangle_{\Omega_m}=&\Big\langle \tfrac{h_{1,m}^\star}{\nu_{1,\Omega_m}},h_{1,m}^\star\Big\rangle_{\mu_1}\\
&-\tfrac{d_2}{d_1}\Big\langle \tfrac{h_{2,m}^\star}{\nu_{2,\Omega_m}},h_{2,m}^\star\Big\rangle_{\mu_2}\\
=&\int_0^\pi |h_{1,m}^\star(\phi)|^2\sin(\phi)r_{0,1}^2(\phi)d\phi\\
&-\tfrac{d_2}{d_1}\int_0^\pi |h_{2,m}^\star(\phi)|^2\sin(\phi)r_{0,2}^2(\phi)d\phi.
\end{align*}}
{
To check the transversality we will need 
\begin{equation}\label{transversality-cond}
\int_0^\pi |h_{1,m}^\star(\phi)|^2\sin(\phi)r_{0,1}^2(\phi)d\phi-\tfrac{d_2}{d_1}\int_0^\pi |h_{2,m}^\star(\phi)|^2\sin(\phi)r_{0,2}^2(\phi)d\phi\neq 0.\end{equation}
}
{
We will argue by contradiction and assume that 
\begin{equation}\label{Trans-1}
\int_0^\pi |h_{1,m}^\star(\phi)|^2\sin(\phi)r_{0,1}^2(\phi)d\phi=\tfrac{d_2}{d_1}\int_0^\pi |h_{2,m}^\star(\phi)|^2\sin(\phi)r_{0,2}^2(\phi)d\phi.
\end{equation}
On the other hand, the vector $h_m^\star=(h_{1,m}^\star, h_{2,m}^\star)$ can be chosen to be  unit, that is, 
$$
\|h_{1,m}^\star\|_{\mu_1}^2+\tfrac{d_2}{d_1}\|h_{2,m}^\star\|_{\mu_2}^2=1
$$
and therefore we deduce from \eqref{MeasureB-1}
$$
\int_0^\pi |{h^\star_{1,m}}(\phi)|^2\sin(\phi) r_{0,1}^2(\phi)\nu_{1,\Omega_m}(\phi)d\phi+\tfrac{d_2}{d_1}\int_0^\pi |{h_{2,m}^\star}(\phi)|^2\sin(\phi) r_{0,2}^2(\phi)\nu_{2,\Omega_m}(\phi)d\phi=1.
$$
According to Proposition \ref{Lem-meas}-(1)  we infer
$$
\mathtt{M}^{-1}\leqslant \int_0^\pi |h^\star_{1,m}(\phi)|^2\sin(\phi)r_{0,1}^2(\phi)\ d\phi+\tfrac{d_2}{d_1}\int_0^\pi |{h_{2,m}^\star}(\phi)|^2\sin(\phi)r_{0,2}^2(\phi)\ d\phi.
$$
Combining this estimate with the identity \eqref{Trans-1} and ({\bf{H2}})we find
\begin{align}\label{Est-cle}
\tfrac{d_1C^{-2}}{2d_2\mathtt{M}}\leqslant \int_0^\pi |h_{2,m}^\star(\phi)|^2\sin^3(\phi)d\phi.
\end{align}
Recall that
$$
h_{m}^{\star}=\mathcal{T}^m_{\Omega_m}(h_m^{\star})
$$
and in particular we infer from the definition \eqref{Tnomega2}
\begin{align*}
{\nu_{2,\Omega_m}(\phi)}h_{2,m}^{\star}(\phi)
=&d_2\int_0^\pi {H_{2,2}^m(\phi,\varphi)}h_{2,m}^{\star}(\varphi)d\varphi-d_1\int_0^{\pi} {H_{2,1}^m(\phi,\varphi)}h_{1,m}^{\star}(\varphi)d\varphi.
\end{align*} 
Using once again Proposition \ref{Lem-meas}-(1) together with \eqref{estim-asym} we obtain
\begin{align*}
(\Omega_m-\overline{\Omega}_2)|h_{2,m}^{\star}(\phi)|\sin^{\frac32}(\phi)
\leqslant &Cm^{-\alpha}\int_0^\pi |\phi-\varphi|^{-2\beta}|h_{2,m}^{\star}(\varphi)|\sin^{\frac32}(\varphi)d\varphi\\
+&Cm^{-\alpha}\int_0^{\pi} |\phi-\varphi|^{-2\beta}|h_{1,m}^{\star}(\varphi)|\sin^{\frac32}(\varphi)d\varphi.
\end{align*}
Applying Cauchy-Schwarz inequality on the right-hand side yields for $\beta<\frac14$
\begin{align*}
(\Omega_m-\overline{\Omega}_2)^2|h_{2,m}^{\star}(\phi)|^2\sin^{3}(\phi)
\leqslant &Cm^{-2\alpha}\int_0^\pi |h_{2,m}^{\star}(\varphi)|^2\sin^{3}(\varphi)d\varphi\\
+&Cm^{-2\alpha}\int_0^{\pi} |h_{1,m}^{\star}(\varphi)|^2\sin^{3}(\varphi)d\varphi.
\end{align*}
Integrating in $\phi$ implies
\begin{align*}
(\Omega_m-\overline{\Omega}_2)^2\int_0^\pi|h_{2,m}^{\star}(\phi)|^2\sin^{3}(\phi)d\phi
\leqslant &Cm^{-2\alpha}\int_0^\pi |h_{2,m}^{\star}(\varphi)|^2\sin^{3}(\varphi)d\varphi\\
+&Cm^{-2\alpha}\int_0^{\pi} |h_{1,m}^{\star}(\varphi)|^2\sin^{3}(\varphi)d\varphi.
\end{align*}
Coming back to \eqref{Trans-1} and using ({\bf{H2}})  we deduce that
 \begin{equation}\label{Tita--1}
\int_0^\pi |h_{1,m}^\star(\phi)|^2\sin^3(\phi)d\phi\leqslant C\int_0^\pi |h_{2,m}^\star(\phi)|^2\sin^3(\phi)d\phi.
\end{equation}
Putting together the previous two estimates allows to get
\begin{align*}
(\Omega_m-\overline{\Omega}_2)^2\int_0^\pi|h_{2,m}^{\star}(\phi)|^2\sin^{3}(\phi)d\phi
\leqslant &Cm^{-2\alpha}\int_0^\pi |h_{2,m}^{\star}(\varphi)|^2\sin^{3}(\varphi)d\varphi.
\end{align*}
Applying Proposition \ref{prop-kernel-onedim}-(2), together with $\overline{\Omega}_2<\overline{\Omega}_1$ we may find $m_0$ large enough such that
$$
\forall\, m\geqslant m_0,\quad (\Omega_m-\overline{\Omega}_2)^2-Cm^{-2\alpha}\geqslant \tfrac12 (\overline{\Omega}_1-\overline{\Omega}_2)^2>0.
$$
Consequently,
\begin{align*}
\int_0^\pi|h_{2,m}^{\star}(\phi)|^2\sin^{3}(\phi)d\phi
\leqslant 0.
\end{align*}
Since $h_{2,m}^{\star}$ is continuous on $[0,\pi]$ then it should vanish everywhere. Similarly we infer from \eqref{Tita--1} that $h_{1,m}^{\star}$ is identically zero. This contradicts the fact that $h^\star_m$ is a unit vector. Therefore the transversality assumption \ref{transversality-cond} is satisfied for any $m\geqslant m_0.$
Finally,  we deduce that $f_m^\star$  does not belong to the range of $\partial_{f_1,f_2} \tilde{F}(\Omega_m,0,0)$ and then the transversality condition is satisfied.
}
\end{proof}

\section{Nonlinear action}\label{sec-regularity}

The aim of this section is to check that the nonlinear functional $\tilde{F}:\R\times B_{X_m^\alpha}(\varepsilon)\times B_{X_m^\alpha}(\varepsilon)\rightarrow X_m^\alpha\times X_m^\alpha$ is well-defined  and smooth. Recall that this functional is defined in $\eqref{Ftilde}$ as

\begin{align*}
\tilde{F}(\Omega,f_1,f_2)(\phi,\theta)&=(\tilde{F}_1(\Omega,f_1,f_2)(\phi,\theta),\tilde{F}_2(\Omega,f_1,f_2)(\phi,\theta))\\
&=\left(\frac{F^{\bf  s}_1(\Omega,f_1,f_2)(\phi,\theta)}{r_{0,1}(\phi)}, \frac{F^{\bf  s}_2(\Omega,f_1,f_2)(\phi,\theta)}{r_{0,2}(\phi)} \right),
\end{align*}
where { $F_j^s$ is defined in \eqref{nonlinearfunction2} as}
$$
{F}^{\bf  s}_j(\Omega,f_1,f_2)(\phi,\theta)=\psi(r_j(\phi,\theta)e^{i\theta},d_j\cos(\phi))-\tfrac{\Omega}{2}r^2_j(\phi,\theta)-m_j(\Omega,f_1,f_2)(\phi),  
$$
{ and $r_j$ is given in \eqref{f}}. 
In order to check the regularity of $\tilde{F}$, we can restrict it only for $\tilde{F}_1$ since $\tilde{F}_2$ follows similarly. Note that in this case the stream function $\psi$ can be written as follows
\begin{align}
\psi(r_1(\phi,\theta)e^{i\theta},d_1\cos(\phi))=&-\frac{d_1}{4\pi}\bigintsss_{0}^{\pi}\bigintsss_0^{2\pi}\bigintsss_0^{r_1(\varphi,\eta)}\frac{\sin(\varphi)rdrd\eta d\varphi}{|(re^{i\eta},d_1\cos(\varphi))-(r_1(\phi,\theta)e^{i\theta},d_1\cos(\phi))|}\nonumber\\
&+\frac{d_2}{4\pi}\bigintsss_{0}^{\pi}\bigintsss_0^{2\pi}\bigintsss_0^{r_2(\varphi,\eta)}\frac{\sin(\varphi)rdrd\eta d\varphi}{|(re^{i\eta},d_2\cos(\varphi))-(r_1(\phi,\theta)e^{i\theta},d_1\cos(\phi))|}\nonumber\\
=:&\mathscr{T}_1(f_1,f_2)(\phi,\theta)-\mathscr{T}_2(f_1,f_2)(\phi,\theta),\label{streamfunction-decomp}
\end{align}
where
\begin{equation}\label{mathscrIi}
\mathscr{T}_i(f_1,f_2)(\phi,\theta)=-\frac{d_i}{4\pi}\bigintsss_{0}^{\pi}\bigintsss_0^{2\pi}\bigintsss_0^{r_i(\varphi,\eta)}\frac{\sin(\varphi)rdrd\eta d\varphi}{|(re^{i\eta},d_i\cos(\varphi))-(r_1(\phi,\theta)e^{i\theta},d_1\cos(\phi))|}\cdot
\end{equation}
Note that we have decomposed the stream function $\psi$ into two functions: $\mathscr{T}_1$ and $\mathscr{T}_2$. In this decomposition, $\mathscr{T}_1$ represents the interaction of the outer patch with itself and then { it turns out to be an integral with a singular kernel}. This term is the stream function that appeared when studying rotating patches with only one surface boundary in \cite{GHM}, and then we can use the previous study here. On the other hand, $\mathscr{T}_2$ describes the action of the outer patch on the inner one. We will observe in Lemma \ref{lemma-estimdenominator} that since the two patches are well-separated by the assumptions {\bf (H)}, we have that the denominator in $\mathscr{T}_2$ is never vanishing:
\begin{equation}\label{den-2}
|(re^{i\eta},d_2\cos(\varphi))-(r_1(\phi,\theta)e^{i\theta},d_1\cos(\phi))|\geq \delta>0,
\end{equation}
for $r\in(0,r_2(\varphi,\eta))$, $\varphi,\phi\in[0,\pi]$ and $\theta,\eta\in[0,2\pi]$. Here, we need to assume that $r_i$ are small perturbations of $r_{i,0}$. Hence the kernel in { $\mathscr{T}_2$ } is not singular and we can check that it is a nonlinear smooth operator as we will see in Lemma \ref{lemma-estimdenominator}-(2)-(3). 

First, we shall  study the deformed Euclidean norm via the spherical coordinates type  and check \eqref{den-2}. After that, we will be able to prove that $\tilde{F}$ is well-defined and $\mathscr{C}^1$.

\subsection{Deformation of the Euclidean norm}
{The Green function that defines the stream function is deformed through the spherical coordinates, and then also the associated velocity field. 
In the new coordinates system  the associated kernel is anisotropic and we find that the singularities are located at the boundary (the north and south poles) and at the diagonal line. This fact is not completely new and was revealed in the recent study  in \cite{GHM} where a suitable treatment was devised. Notice that, we should here deal with the new terms related to  the interaction between the  two surfaces and then the associated stream function contributes with an  extra term compared to \cite{GHM}: $\mathscr{T}_2$ in }\eqref{streamfunction-decomp}. In order to study the denominator in the kernel, define
$$
J_{ij}(s,\phi,\theta,\varphi,\eta):=(r_j(\varphi,\eta)-sr_i(\phi,\theta))^2+2sr_i(\phi,\theta)r_j(\varphi,\eta)(1-\cos(\theta-\eta))+(d_i\cos(\phi)-d_j\cos(\varphi))^2,
$$
for any $\phi,\varphi\in[0,\pi]$, $\theta,\eta\in[0,2\pi]$ and $s\in[0,1]$. Note that if $r_i\in X_m^\alpha$, then from the symmetry assumption we have that 
$$
J_{ij}(s,\pi-\phi,\theta,\pi-\varphi,\eta)=J_{ij}(s,\phi,\theta,\varphi,\eta),
$$
and hence we can work under the restriction $\phi\in[0,\pi/2]$ instead of $\phi\in[0,\pi]$. Such function will be studied in the following lemma.

\begin{lem}\label{lemma-estimdenominator}
Let $m\geq1, \alpha\in(0,1)$ and $\varepsilon$ small enough. Consider $r_i=r_{0,i}+f_i$, where $r_{0,i}$ satisfies ${\bf (H)}$, with  $f_i \in X_m^\alpha$ and $\|f_i\|\leq \varepsilon$, for $i=1,2$.  There exists two constants  $C,\delta>0$ such that
\begin{enumerate}
\item For { $i\in\{1,2\}$} we have
 \begin{equation*}
 |J_{ii}(0,\phi,\theta,\varphi,\eta)|\geqslant C\sin^2(\varphi)
 \end{equation*}
and 
\begin{equation*}
|J_{ii}(s,\phi,\theta,\varphi,\eta)|\geqslant C\Big(\big(\sin^2(\varphi)+s^2\phi^2\big){\sin^2((\theta-\eta)/2)}+(\varphi+\phi)^2(\phi-\varphi)^2\Big). 
\end{equation*}
\item For $i\neq j$ we have
\begin{equation*}
|J_{ij}(1,\phi,\theta,\varphi,\eta)|\geqslant \delta.
\end{equation*}
\item For $i\neq j$, there exists $\phi_0>0$ small enough such that
\begin{equation*}
|J_{ij}(s,\phi,\theta,\varphi,\eta)|\geqslant\delta,\quad {\phi\in [0,\phi_0]}. 
\end{equation*}
\end{enumerate}
\end{lem}

\begin{proof}

\noindent\medskip
{\bf(1)}
We refer to the proof of  \cite[Lemma 5.2]{GHM}.

\noindent\medskip 
{\bf(2)}
Take $i\neq j$, then
$$
|r_i(\phi,\theta)-r_j(\varphi,\eta)|\geq |r_{0,i}(\phi)-r_{0,j}(\varphi)|-|f_i(\phi,\theta)-f_j(\varphi,\eta)|.
$$
By using {\bf (H4)} and  $(A-B)^2\geq \frac{A^2}{2}-B^2$, that follows from to Young inequality, one has
\begin{align*}
(r_i(\phi,\theta)-r_j(\varphi,\eta))^2+&(d_i\cos(\phi)-d_j\cos(\varphi))^2\\
\geqslant& {\frac{1}{2}}(r_{0,i}(\phi)-r_{0,j}(\varphi))^2+(d_i\cos(\phi)-d_j\cos(\varphi))^2-|f_i(\phi)-f_j(\varphi)|^2\\
\geqslant & {\frac{\delta}{2}}-4\|f_i\|_{L^\infty}^2.
\end{align*}
Since $\|f_i\|_{L^\infty}<\varepsilon$ we can choose $\varepsilon$ small { enough such that}
\begin{align*}
(r_i(\phi,\theta)-r_j(\varphi,\eta))^2+(d_i\cos(\phi)-d_j\cos(\varphi))^2\geqslant  \tfrac{\delta}{4}.
\end{align*}
Then, for any $i\neq j$ we have
$$
J_{ij}(1,\phi,\theta,\varphi,\eta)\geqslant \tfrac{\delta}{4}.
$$

\noindent\medskip 
{\bf(3)}
According to the assumption {\bf (H2)} one has for any  $\phi_0\in(0,\pi)$ 
$$
r_{0,i}(\phi)\leqslant C\sin(\phi_0), \quad \forall \phi\in[0,\phi_0].
$$
 Note also that
\begin{align*}
|sr_i(\phi,\theta)-r_j(\varphi,\eta)|\geq& |r_{0,i}(\phi)-r_{0,j}(\varphi)|-|f_i(\phi,\theta)-f_j(\varphi,\eta)|-(1-s)r_{0,i}(\phi)-(1-s)|f_i(\phi,\theta)|\\
\geq& |r_{0,i}(\phi)-r_{0,j}(\varphi)|-|f_i(\phi,\theta)-f_j(\varphi,\eta)|-r_{0,i}(\phi)-|f_i(\phi,\theta)|\\
\geq& |r_{0,i}(\phi)-r_{0,j}(\varphi)|-2\varepsilon-C\sin(\phi_0)-\varepsilon,
\end{align*}
where we have used that $\|f_i\|_{L^\infty}\leq \varepsilon$. Moreover, by using {\bf (H4)} one has for $i\neq j$
\begin{align*}
&\left({(sr_i(\phi,\theta)-r_j(\varphi,\eta))^2+(d_i\cos(\phi)-d_j\cos(\varphi))^2}\right)^{\frac12}\\
&\qquad \geqslant {\tfrac{1}{2}}|sr_i(\phi,\theta)-r_j(\varphi,\eta)|+ {\frac{1}{2}}|d_i\cos(\phi)-d_j\cos(\varphi)|\\
&\qquad\qquad \geqslant {\tfrac{1}{2}}|r_{0,i}(\phi)-r_{0,j}(\varphi)|+ {\tfrac{1}{2}}|d_i\cos(\phi)-d_j\cos(\varphi)|-\tfrac{3}{2}\varepsilon-\tfrac{C}{2}\sin(\phi_0)\\
&\qquad\qquad\qquad\qquad \geqslant {\frac{1}{2}}\delta-\tfrac{3}{2}\varepsilon-\frac{1}{2}C\sin(\phi_0).
\end{align*}
Choosing $\varepsilon$ and $\phi_0$ small enough we have that
$$
 {\tfrac{\delta}{2}}-\tfrac{3}{2}\varepsilon-\tfrac{C}{2}\sin(\phi_0)>\tfrac{\delta}{4},
$$
concluding the proof.
\end{proof}

\subsection{Regularity persistence}
{In this section we shall investigate the regularity of the function $\tilde{F}$ introduced in \eqref{Ftilde}. Due to their similar structure we shall restrict the discussion  to  $\tilde{F}_1$, and the same arguments can be implemented to $\tilde{F}_2$. 
In the proof of the forthcoming Proposition \ref{prop-wellpos} we shall study the regularity of the velocity field. For this task, we need   first to analyze some algebraic properties  related to  the velocity field at the vertical axis. We shall see that  the revolution shape invariance of the the surface of the vortices  induces a vanishing  velocity at the symmetry axis.} From \eqref{U} we find that
\begin{align}\label{U-2}
U(Re^{i\theta},z)=&\frac{d_1}{4\pi}\bigintsss_{0}^{\pi}\bigintsss_0^{2\pi}\frac{\sin(\varphi)(\partial_{\eta}r_1(\varphi,\eta)e^{i\eta}+ir_1(\varphi,\eta)e^{i\eta})}{|(Re^{i\theta},z)-(r_1(\varphi,\eta)e^{i\eta},d_1\cos(\varphi))|}d\eta d\varphi\\
&-\frac{d_2}{4\pi}\bigintsss_{0}^{\pi}\bigintsss_0^{2\pi}\frac{\sin(\varphi)(\partial_{\eta}r_2(\varphi,\eta)e^{i\eta}+ir_2(\varphi,\eta)e^{i\eta})}{|(Re^{i\theta},z)-(r_2(\varphi,\eta)e^{i\eta},d_2\cos(\varphi))|}d\eta d\varphi.\nonumber
\end{align}
Recall that by using \eqref{I1} and \eqref{I2} we get 
$$
U(Re^{i\theta},z)=\mathcal{I}_1(R,z)-\mathcal{I}_2(R,z).
$$

In the next lemma, we state that the velocity field is vanishing at the vertical axis. We omit the proof due to its similarity with \cite[Lemma 5.1]{GHM}.
{\begin{lem}\label{lemma-velocidad0}
If { the initial profile  $r_{0,i}, \, i=1,2$ satisfy} ${\bf{(H)}}$ and $f_1,f_2\in B_{X_m^\alpha}(\varepsilon)$, with $m\geqslant 2$ and $\varepsilon$ small enough, then 
\begin{align*}
\forall\, z\in\R,\quad \bigintsss_0^\pi\bigintsss_0^{2\pi}\tfrac{\sin(\varphi)\partial_\eta(r_i(\varphi,\eta)e^{i\eta})d\eta d\varphi}{\big(r^2_i(\varphi,\eta)+\big(z-d_i\cos\varphi\big)^2\big)^\frac12}=&0.
\end{align*}
As a consequence, the velocity field $(U,0)=(u,v,0)$ defined  in \eqref{U-2} is vanishing at the vertical axis, that is, 
{
$$
\forall\, z\in\R,\quad U(0,0,z)=(0,0).
$$
} 
\end{lem}}
We are ready to prove the  $\tilde{F}$ is well-defined  and acts smoothly on a small ball of the function space $X_m^\alpha$ introduced in \eqref{space} and \eqref{spaceXL1}.
\begin{pro}\label{prop-wellpos}
Let  $m\geqslant 2, \alpha\in(0,1)$ and $r_{0,i}$ satisfy ${\bf{(H)}}$ for $i=1,2.$ There exists { $\varepsilon>0$ small enough} such that the functional 
$$\tilde{F}:\R\times B_{X_m^\alpha}(\varepsilon)\times B_{X_m^\alpha}(\varepsilon)\rightarrow X_m^\alpha\times X_m^\alpha,
$$ is well-defined and of class $\mathscr{C}^1$. 
\end{pro}
\begin{proof}
We shall establish the desired  results  for  $\tilde{F}_1$ and  the same kind of  arguments can be applied to  $\tilde{F}_2$. Let us start with checking the regularity properties and we will postpone  the symmetry to the end of the proof.
We decompose $\tilde{F}_1$ as in \eqref{streamfunction-decomp},
$$
\tilde{F}_1(\Omega,f_1,f_2)(\phi,\theta)=\frac{1}{r_{0,1}(\phi)}\left\{\mathscr{T}_1(f_1,f_2)(\phi,\theta)-\mathscr{T}_2(f_1,f_2)(\phi,\theta)-\frac{\Omega}{2}r_1^2(\phi,\theta)-m_1(\Omega,f_1,f_2)(\phi)\right\},
$$
where $\mathscr{T}_i$ is defined in \eqref{mathscrIi}. Moreover, let us recall the identity
$$
\mathscr{T}_1-\mathscr{T}_2=\psi((r_{0,1}(\phi)+f_1(\phi,\theta))e^{i\theta},d_1\cos(\phi)).
$$
Thus  we may  split the functional  $\tilde{F}_1$ into two parts
$$
\tilde{F}_1(\Omega,f_1,f_2)(\phi,\theta)=G_1(f_1,f_2)(\phi,\theta)-\frac{\Omega}{2} G_2(f_1)(\phi,\theta)-\frac{1}{2\pi}\bigintsss_0^{2\pi}\left[G_1(f_1,f_2)(\phi,\theta)-\frac{\Omega}{2} G_2(f_1)(\phi,\theta)\right]d\theta,
$$
with
\begin{align*}
G_1(f_1,f_2)(\phi,\theta)=&\frac{\mathscr{T}_1(f_1,f_2)(\phi,\theta)-\mathscr{T}_2(f_1,f_2)(\phi,\theta)}{r_{0,1}(\phi)},\\
G_2(f_1)(\phi,\theta)=&2f_1(\phi,\theta)+\frac{f^2_1(\phi,\theta)}{r_{0,1}(\phi)}\cdot
\end{align*}
Define also 
\begin{align*}
\mathscr{G}_1(f_1,f_2)(\phi,\theta):=&G_1(f_1,f_2)(\phi,\theta)-\frac{1}{2\pi}\bigintsss_0^{2\pi}G_1(f_1,f_2)(\phi,\theta)d\theta,\\
\mathscr{G}_2(f_1)(\phi,\theta):=&G_2(f_1)(\phi,\theta)-\frac{1}{2\pi}\bigintsss_0^{2\pi}G_2(f_1)(\phi,\theta)d\theta.
\end{align*}
By \eqref{streamfunction-decomp} one has
\begin{equation}\label{F1-X1}
G_1(f_1,f_2)(\phi,\theta)=\frac{\psi\left(\big(r_{0,1}(\phi)+f_1(\phi,\theta)\big)e^{i\theta}, d_1\cos\phi\right)}{r_{0,1}(\phi)}\cdot
\end{equation}
The regularity of $\mathscr{G}_2(f_1)$ was already studied in \cite[Proposition 5.2, Step 1]{GHM}. It remains to consider $\mathscr{G}_1$ whose  proof will be divided into two steps. In step 1 we shall analyze  the regularity, however in step 2   we  will prove the persistence of the symmetry.

\medskip
\noindent
{{{\bf Step 1:} $(f_1,f_2)\mapsto \mathscr{G}_1(f_1,f_2)$ {\it  is well-defined}.} We will check that the functional $\mathscr{G}_1$  is symmetric with respect to $\phi=\frac{\pi}{2}$ and therefore it suffices to check the desired regularity in the range $\phi\in(0,\pi/2)$ { and check that the derivative is vanishing at $\pi/2$}. Let us emphasize  that   we need to check the regularity not for  $G_1$ but for its fluctuation with respect to the average, in terms of  $\mathscr{G}_1$. Notice that  we have to include the average in $\mathscr{G}_1$ in order to satisfy   the boundary conditions. In addition and we will see later at some steps of the proof,  some singular parts can be removed by subtracting the average.}
First, we shall check that $\mathscr{G}_1$ is bounded and satisfies the boundary condition $\mathscr{G}_1(0,\theta)=\mathscr{G}_1(\pi,\theta)=0$, for any $\theta\in(0,2\pi)$. 
 {Let us remark that since the kernel $\frac{1}{|x|}$ has homogeneity $-2,$ 
from the general potential theory, if $D$ is the characteristic function of a bounded domain, one has that $\psi=\frac{1}{|x|}\star \chi_D
\in \mathscr{C}^{1,\alpha}(\R^3).$ } Hence, we can write by virtue of Taylor's formula
\begin{align}\label{psi0X}
\forall x_h\in\R^2,\quad \psi(x_h,d_1\cos\phi)
=&\psi(0,0,d_1\cos\phi)+x_h\cdot\int_0^1\nabla_h\psi\big(\tau x_h,d_1\cos\phi\big)d\tau.
\end{align}
Making the substitution $x_h=(r_{0,1}(\phi)+f_1(\phi,\theta))e^{i\theta}$ and using \eqref{F1-X1} we infer
\begin{align*}
G_1(f_1,f_2)(\phi,\theta)=&\tfrac{\psi(0,0,d_1\cos\phi)}{r_{0,1}(\phi)}+\left(1+\tfrac{f_1(\phi,\theta)}{r_{0,1}(\phi)}\right)e^{i\theta}\cdot\int_0^1\nabla_h\psi\left(\tau(r_{0,1}(\phi)+f_1(\phi,\theta))e^{i\theta},d_1\cos\phi\right)d\tau\\
=:&\tfrac{\psi(0,0,d_1\cos\phi)}{r_{0,1}(\phi)}+\mathscr{G}_{1,1}(\phi,\theta).
\end{align*}
We observe that  the $\cdot$ denotes  the usual Euclidean  inner product  of $\R^2.$
Consequently, we obtain
\begin{equation}\label{Simplifi1}
\mathscr{G}_1(\phi,\theta)=\mathscr{G}_{1,1}(\phi,\theta)-{\frac{1}{2\pi}\int_0^{2\pi}\mathscr{G}_{1,1}(\phi,\theta)d\theta.
}
\end{equation}
Let us analyze the  term $ \mathscr{G}_{1,1}$ and check its continuity and the Dirichlet boundary condition. First we observe from the assumption ${\bf{(H2)}}$ that $0$ is a simple zero for $r_{0,1}$ and we know that $f_1(0,\theta)=0$, then    one may easily  obtain the bound 
$$
|\mathscr{G}_{1,1}(\phi,\theta)|\leq C(1+\|\partial_\phi f_1\|_{L^\infty})\|\nabla_h\psi\|_{L^\infty(\R^3)}.
$$
Notice that $\psi$ is $\mathscr{C}^{1,\alpha}$ in the physical variables, and then same tool gives the continuity of $ \mathscr{G}_{1,1}$ in $[0,\pi/2]\times[0,2\pi].$ Let us move now to the boundary conditions. According to  Lebesgue dominated convergence theorem  we infer
$$
\lim_{\phi\to 0}\mathscr{G}_{1,1}(\phi,\theta)=\left(1+\tfrac{\partial_\phi f_1(0,\theta)}{r_{0,1}^\prime(0)}\right)e^{i\theta}\cdot\nabla_h\psi\left(0,0,{d_1}\right),
$$
and this convergence is uniform in $\theta\in(0,2\pi)$. Now, applying  Lemma \ref{lemma-velocidad0} we get {$\nabla_h\psi\left(0,0,d_1\right)=(0,0)$}, and therefore
$$
\forall \,\theta\in(0,2\pi),\quad \lim_{\phi\to 0}\mathscr{G}_{1,1}(\phi,\theta)=\lim_{\phi\to 0}\langle\mathscr{G}_{1,1}\rangle_\theta=0.
$$
This implies that $\mathscr{G}_1$ is  continuous in $[0,\pi]\times[0,2\pi]$ and it  satisfies the required  Dirichlet  boundary condition $\mathscr{G}_1(0,\theta)=\mathscr{G}_1(\pi,\theta)=0$.

The next step is to establish that  $\partial_\theta\mathscr{G}_1$ and $\partial_\phi \mathscr{G}_1$ are $\mathscr{C}^\alpha$. We will relate such derivatives to the two-components { of the} velocity field $U=\nabla_h^\perp \psi$. Differentiating \eqref{F1-X1} with respect to $\theta$  yields
\begin{align}
 &\qquad\qquad\qquad\partial_\theta\mathscr{G}_1(\phi,\theta)=
\partial_\theta G_1(f_1,f_2)(\phi,\theta)\label{F1theta}\\
=&r_{0,1}^{-1}(\phi)\,\nabla_h \psi(r_1(\phi,\theta)e^{i\theta},d_1 \cos(\phi))\cdot\left(r_1(\phi,\theta)ie^{i\theta}+\partial_\theta r_1(\phi,\theta)e^{i\theta}\right)\nonumber\\
=&-\frac{r_1(\phi,\theta)}{r_{0,1}(\phi)} U(r_1(\phi,\theta)e^{i\theta}, d_1\cos(\phi))\cdot e^{i\theta}{{+}} \partial_\theta r_1(\phi,\theta) 
\frac{U(r_1(\phi,\theta)e^{i\theta},d_1\cos(\phi))}{r_{0,1}(\phi)}\cdot ie^{i\theta}.\nonumber
\end{align}
{As  to the partial derivative  in $\phi$, we achieve}
\begin{align}
\partial_\phi G_1(f_1,f_2)(\phi,\theta)=&-\tfrac{r_{0,1}'(\phi)}{r_{0,1}^2(\phi)}\psi(r_1(\phi,\theta)e^{i\theta}, d_1\cos(\phi))+
\tfrac{\partial_\phi r_1(\phi,\theta)}{r_{0,1}(\phi)}\nabla_h \psi(r_1(\phi,\theta)e^{i\theta},d_1\cos(\phi))\cdot e^{i\theta}\nonumber\\
&-\tfrac{{d_1}\sin(\phi)}{r_{0,1}(\phi)}\partial_z\psi(r_1(\phi,\theta)e^{i\theta},d_1\cos(\phi))\nonumber\\
=&-\tfrac{r_{0,1}'(\phi)}{r_{0,1}^2(\phi)}\psi(r_1(\phi,\theta)e^{i\theta}, d_1\cos(\phi))+
\partial_\phi r_1(\phi,\theta)\tfrac{U(r_1(\phi,\theta)e^{i\theta},d_1\cos(\phi))}{r_{0,1}(\phi)} \cdot ie^{i\theta}\nonumber\\
&-\tfrac{{d_1}\sin(\phi)}{r_{0,1}(\phi)}\partial_z\psi(r_1(\phi,\theta)e^{i\theta},d_1\cos(\phi)).\label{F1phi}
\end{align}
Define
\begin{equation}\label{J1-0}
\mathscr{J}_{1}(\phi,\theta):=\frac{r_{0,1}'(\phi)}{r_{0,1}^2(\phi)}\psi(r_1(\phi,\theta)e^{i\theta}, d_1\cos(\phi)),
\end{equation}
and
\begin{equation}\label{J2-0}
\mathscr{J}_{2}(\phi,\theta):=\frac{U(r_1(\phi,\theta)e^{i\theta},d_1\cos(\phi))}{r_{0,1}(\phi)} \cdot ie^{i\theta}.
\end{equation}
Then from \eqref{F1theta} and \eqref{F1phi} we may write
\begin{equation}\label{partialthetaG1}
\partial_\theta\mathscr{G}_1(\phi,\theta)= {-}\frac{r_1(\phi,\theta)}{r_{0,1}(\phi)} U(r_1(\phi,\theta)e^{i\theta}, d_1\cos(\phi))\cdot e^{i\theta}      {{+}}\partial_\theta r_1(\phi,\theta) \mathscr{J}_2(\phi,\theta),
\end{equation}
and
\begin{equation}\label{partialphiG1}
\partial_\phi G_1(f_1,f_2)(\phi,\theta)=-\mathscr{J}_{1}(\phi,\theta)+\partial_\phi r_1(\phi,\theta)\mathscr{J}_{2}(\phi,\theta)-\tfrac{{d_1}\sin(\phi)}{r_{0,1}(\phi)}\partial_z\psi(r_1(\phi,\theta)e^{i\theta},d_1\cos(\phi)).
\end{equation}
{Observe that the last term in \eqref{partialphiG1} belongs to $\mathscr{C}^\alpha((0,\pi)\times\T)$ (here we do not need to restrict ourselves to $\phi\in[0,\pi/2]$). Indeed, as  $(\phi,\theta)\mapsto \big(r_1(\phi,\theta)e^{i\theta},d_1\cos(\phi)\big)$ belongs to $\mathscr{C}^{1,\alpha}$ and $\partial_z\psi\in \mathscr{C}^{\alpha}(\R^3)$ then by composition we infer $(\phi,\theta)\mapsto \partial_z\psi\big(r_1(\phi,\theta)e^{i\theta},d_1\cos(\phi)\big)$ is in $\mathscr{C}^{\alpha}((0,\pi)\times{\T}\big)$. On the  other hand, the function  {$\phi\mapsto \frac{\sin(\phi)}{r_{0,1}(\phi)}$}  belongs to the algebra $ \mathscr{C}^\alpha$ obtaining the desired result. {
Since $\phi\mapsto \frac{r_1(\phi,\theta)}{r_{0,1}(\phi)}$ is a function in $\mathscr{C}^{\alpha},$} then to get the desired regularity it is enough   from \eqref{partialthetaG1}--\eqref{partialphiG1} to check that }
 $\mathscr{J}_1$,  $\mathscr{J}_2$  and the function 
\begin{equation}\label{U-partialtheta}
U(r_1(\phi,\theta)e^{i\theta}, d_1\cos(\phi))\cdot e^{i\theta}
\end{equation}
belong to $\mathscr{C}^\alpha([0,\pi/2]\times\T).$
Let us start with the last one  $U(r_1(\phi,\theta)e^{i\theta},d_1\cos(\phi))\cdot e^{i\theta}$. First, we note that $U$ can be written as \eqref{U-2},
\begin{align}\label{U1U2}
U(Re^{i\theta},z)=&\frac{d_1}{4\pi}\bigintsss_{0}^{\pi}\bigintsss_0^{2\pi}\frac{\sin(\varphi)(\partial_{\eta}r_1(\varphi,\eta)e^{i\eta}+ir_1(\varphi,\eta)e^{i\eta})}{|(Re^{i\theta},z)-(r_1(\varphi,\eta)e^{i\eta},d_1\cos(\varphi))|}d\eta d\varphi\\
&-\frac{d_2}{4\pi}\bigintsss_{0}^{\pi}\bigintsss_0^{2\pi}\frac{\sin(\varphi)(\partial_{\eta}r_2(\varphi,\eta)e^{i\eta}+ir_2(\varphi,\eta)e^{i\eta})}{|(Re^{i\theta},z)-(r_2(\varphi,\eta)e^{i\eta},d_2\cos(\varphi))|}d\eta d\varphi\nonumber\\
=&\mathcal{I}_1(R,z)-\mathcal{I}_2(R,z).\nonumber
\end{align}
Remark  that $\mathcal{I}_1$ refers to the induced effect generated by  the outer domain, which is similar to the velocity field  analyzed in \cite{GHM}. Actually, from \cite[Proposition 5.2, Step 2]{GHM} we have that the functional
$$
(\phi,\theta)\mapsto {\mathcal{I}_1(r_1(\phi,\theta), d_1\cos(\phi))}\cdot e^{i\theta},
$$
belong to $\mathscr{C}^\alpha([0,\pi/2]\times\T)$. It remains to study $\mathcal{I}_2$, where
\begin{equation}\label{U2}
\mathcal{I}_2(r_1(\phi,\theta),d_1\cos(\phi))=\frac{d_2}{4\pi}\bigintsss_{0}^{\pi}\bigintsss_0^{2\pi}\frac{\sin(\varphi)(\partial_{\eta}r_2(\varphi,\eta)e^{i\eta}+ir_2(\varphi,\eta)e^{i\eta})\,\,d\eta d\varphi}{|(r_1(\phi,\theta)e^{i\theta},d_1 \cos(\phi))-(r_2(\varphi,\eta)e^{i\eta},d_2\cos(\varphi))|}\cdot
\end{equation}
By Lemma \ref{lemma-estimdenominator}-(2) dealing with  $J_{12}(1,\phi,\theta,\varphi,\eta)$,  we have that
$$
|(r_1(\phi,\theta)e^{i\theta},d_1 \cos(\phi))-(r_2(\varphi,\eta)e^{i\eta},d_2\cos(\varphi))|\geqslant \delta,
$$
for some $\delta>0$.    Hence, the function  in $\mathcal{I}_2$ is not singular and, {since the denominator in the above integral is in $\mathscr{C}^\alpha,$}   we can easily prove that
$$
(\phi,\theta)\mapsto \mathcal{I}_2(r_1(\phi,\theta),d_1\cos(\phi))\in \mathscr{C}^\alpha.
$$
Let us now  move  to the term $\mathscr{J}_2$ in \eqref{J2-0}. From \cite[Proposition 5.2, Step 2]{GHM} we infer  that 
$$
(\phi,\theta)\mapsto  \frac{\mathcal{I}_1(sr_1(\phi,\theta), d_1\cos(\phi))}{r_{0,1}(\phi)}\cdot i\,e^{i\theta},
$$
belongs to $\mathscr{C}^\alpha([0,\pi/2]\times\T)$,  for any $s\in[0,1]$.  Therefore, in order to achieve that $\mathscr{J}_2$  
belongs to $\mathscr{C}^\alpha$ we need to prove
\begin{equation}\label{nonlinear-est-1}
(\phi,\theta)\mapsto  \frac{\mathcal{I}_2(r_1(\phi,\theta), d_1\cos(\phi))}{r_{0,1}(\phi)}\cdot i\,e^{i\theta}\in \mathscr{C}^\alpha((0,\pi/2)\times{\T}).
\end{equation}
For this aim,  we write by virtue of \eqref{U2}
\begin{align}\label{U2-est-1}
\frac{\mathcal{I}_2(s{r_1(\phi,\theta)},d_1\cos(\phi))}{r_{0,1}(\phi)}\cdot i e^{i\theta}=\frac{d_2}{4\pi r_{0,1}(\phi)}\bigintsss_0^\pi\bigintsss_0^{2\pi}\frac{\sin(\varphi)\partial_\eta (r_2(\varphi,\eta)\sin(\eta-\theta))}{J_{1,2}(s,\phi,\theta,\varphi,\eta)^\frac{1}{2}}d\eta d\varphi.
\end{align}
According to Lemma \ref{lemma-velocidad0} we have the identity
$$
\bigintsss_0^\pi\bigintsss_0^{2\pi}\frac{\sin(\varphi)\partial_\eta (r_2(\varphi,\eta)\sin(\eta-\theta))}{J_{1,2}(0,\phi,\theta,\varphi,\eta)^\frac{1}{2}}d\eta d\varphi{=0}.
$$
Hence, adding such term to \eqref{U2-est-1} we find
{\footnotesize\begin{align*}
&\qquad\qquad\qquad\qquad\qquad\qquad\frac{\mathcal{I}_2(s{ r_1(\phi,\theta)},d_1\cos(\phi))}{r_{0,1}(\phi)}\cdot i e^{i\theta}\\
&=-d_2\frac{r_1(\phi,\theta)}{4\pi r_{0,1}(\phi)}\int_0^\pi\int_0^{2\pi}\frac{\sin(\varphi)\partial_\eta (r_2(\varphi,\eta)\sin(\eta-\theta))s\left\{sr_1(\phi,\theta)-2r_2(\varphi,\eta)\cos(\theta-\eta)\right\}}{J_{1,2}(s,\phi,\theta,\varphi,\eta)^\frac{1}{2}J_{1,2}(0,\phi,\theta,\varphi,\eta)^\frac{1}{2}(J_{1,2}(s,\phi,\theta,\varphi,\eta)^\frac{1}{2}+J_{1,2}(0,\phi,\theta,\varphi,\eta)^\frac{1}{2})}d\eta d\varphi.\nonumber
\end{align*}}
To establish  \eqref{nonlinear-est-1}, it is enough to show that
{\footnotesize\begin{align*}
&\qquad\qquad\qquad\qquad\qquad\qquad\frac{\mathcal{I}_2({r_1(\phi,\theta)},d_1\cos(\phi))}{r_{0,1}(\phi)}\cdot i e^{i\theta}\\
&=-d_2\frac{r_1(\phi,\theta)}{4\pi r_{0,1}(\phi)}\int_0^\pi\int_0^{2\pi}\frac{\sin(\varphi)\partial_\eta (r_2(\varphi,\eta)\sin(\eta-\theta))\left\{r_1(\phi,\theta)-2r_2(\varphi,\eta)\cos(\theta-\eta)\right\}}{J_{1,2}(1,\phi,\theta,\varphi,\eta)^\frac{1}{2}J_{1,2}(0,\phi,\theta,\varphi,\eta)^\frac{1}{2}(J_{1,2}(1,\phi,\theta,\varphi,\eta)^\frac{1}{2}+J_{1,2}(0,\phi,\theta,\varphi,\eta)^\frac{1}{2})}d\eta d\varphi,\nonumber
\end{align*}}
belongs to $\mathscr{C}^\alpha((0,\pi/2)\times(0,2\pi))$. By Lemma \ref{lemma-estimdenominator}-(2), we have the estimate below
$$
J_{1,2}(1,\phi,\theta,\varphi,\eta)\geqslant \delta,
$$
and hence we can check the desired regularity since the denominator is not singular by applying classical  law products.

The term $\mathscr{J}_{1}$ in \eqref{J1-0} is more delicate. Fix $\phi_0$ as in Lemma \ref{lemma-estimdenominator}-(3) and let $\chi\in \mathscr{C}^\infty$ be a cut-off function such that 

\begin{equation*}
\chi(\phi)=\left\{
\begin{array}{ll}
1,&\phi\in[0,\phi_{0}/2]\\
0, &\phi\in[\phi_{0},\pi/2].
\end{array}
\right.
\end{equation*}
Then we may  obtain the  decomposition
$$
\mathscr{J}_1=\mathscr{J}_1\chi+\mathscr{J}_1(1-\chi).
$$
Let us analyze $\mathscr{J}_1(1-\chi)$. We have that { this new function} is vanishing in $[0,\phi_{0}/2]$ by definition and then the denominator $r_{0,1}^2$ is not vanishing. Consequently, using the classical law products combined with the fact that $\psi$ is {$\mathscr{C}^{1+\alpha}$} in the physical coordinates we obtain that  $\mathscr{J}_1(1-\chi)\in\mathscr{C}^\alpha((0,\pi/2)\times\T)$. As to the term 
 $\mathscr{J}_1\chi,$ we shall use Taylor formula for the stream function $\psi$ as in \eqref{psi0X} leading to 
\begin{align*}
&\qquad \qquad\qquad \mathscr{J}_{1}(\phi,\theta){\chi(\phi)}=\tfrac{r_{0,1}'(\phi)\psi(0,0,d_1\cos\phi)}{r_{0,1}^2(\phi)}\chi(\phi)\\
&+\chi(\phi)r_{0,1}'(\phi)r_{0,1}^{-1}(\phi)\left(1+\tfrac{f_1(\phi,\theta)}{r_{0,1}(\phi)}\right)\,\bigintsss_0^1\nabla_h\psi\big(s\, r_1(\phi,\theta)e^{i\theta},d_1\cos\phi\big)ds\cdot e^{i\theta}\\
=&\tfrac{r_{0,1}'(\phi)\psi(0,0,d_1\cos\phi)}{r_{0,1}^2(\phi)}\chi(\phi){+}\chi(\phi)r_{0,1}'(\phi)\left(1+\tfrac{f_1(\phi,\theta)}{r_{0,1}(\phi)}\right)\,r_{0,1}^{-1}(\phi) \bigintsss_0^1U\big(s\, r_1(\phi,\theta)e^{i\theta},d_1\cos\phi\big)ds\cdot ie^{i\theta}.
\end{align*}
We observe that the first term is singular at the poles but it  depends only on $\phi$ and therefore it does not contribute in the fluctuation $\mathscr{J}_{1}-\langle \mathscr{J}_{1}\rangle_\theta.$ In addition,  $(\phi,\theta)\mapsto \frac{f_1(\phi,\theta)}{r_{0,1}(\phi)}$ belongs to $\mathscr{C}^\alpha$ then to get  $\mathscr{J}_{1}-\langle \mathscr{J}_{1}\rangle_\theta\in \mathscr{C}^\alpha$ it suffices to prove that
\begin{equation}\label{CondXX1}
(\phi,\theta)\mapsto\chi(\phi)\bigintsss_0^1 \frac{U(s r_1(\phi,\theta)e^{i\theta}, d_1\cos(\phi))}{r_{0,1}(\phi)}\cdot ie^{i\theta} d s\in \mathscr{C}^\alpha((0,\pi/2)\times \T).
\end{equation}
We can once again decompose $U$ in $\mathcal{I}_1$ and $\mathcal{I}_2$. The term related to $\mathcal{I}_1$ was already studied in \cite[Proposition 5.2, Step 2]{GHM} and it remains to check the contribution of $\mathcal{I}_2$, that is, to check
\begin{equation}\label{nonlinear-est-2}
(\phi,\theta)\mapsto  \chi(\phi)\frac{\mathcal{I}_2(sr_1(\phi,\theta), d_1\cos(\phi))}{r_{0,1}(\phi)}\cdot i\,e^{i\theta}\in \mathscr{C}^\alpha((0,\pi/2)\times\T).
\end{equation}
For this purpose,  we write
{\footnotesize\begin{align*}
&\qquad\qquad \qquad\qquad \chi(\phi)\frac{\mathcal{I}_2(s{r_1(\phi,\theta)},d_1\cos(\phi))}{r_{0,1}(\phi)}\cdot i e^{i\theta}\\
&=-d_2\chi(\phi)\frac{r_1(\phi,\theta)}{4\pi r_{0,1}(\phi)}\int_0^\pi\int_0^{2\pi}\frac{\sin(\varphi)\partial_\eta (r_2(\varphi,\eta)\sin(\eta-\theta))s\left\{sr_1(\phi,\theta)-2r_2(\varphi,\eta)\cos(\theta-\eta)\right\}}{J_{1,2}(s,\phi,\theta,\varphi,\eta)^\frac{1}{2}J_{1,2}(0,\phi,\theta,\varphi,\eta)^\frac{1}{2}(J_{1,2}(s,\phi,\theta,\varphi,\eta)^\frac{1}{2}+J_{1,2}(0,\phi,\theta,\varphi,\eta)^\frac{1}{2})}d\eta d\varphi.\nonumber
\end{align*}}
Again by virtue of Lemma \ref{lemma-estimdenominator}-(3) and since the cut-off function $\chi$ is vanishing in $ [\phi_0,\frac\pi2]$ we have that the denominator is never vanishing,
\begin{align*}
J_{1,2}(s,\phi,\theta,\varphi,\eta)\geqslant \delta>0, \quad s\in[0,1], \phi\in[0,\phi_0],
\end{align*}
for any $\theta,\eta\in\T$ and $\varphi\in[0,\pi]$. In that way, since the integral is not singular we can easily find that it belongs to $\mathscr{C}^\alpha((0,\pi/2)\times\T)$ for any $s\in[0,1]$. 
This ends the proof of \eqref{nonlinear-est-1} and \eqref{nonlinear-est-2} concluding that 
$$
\partial_\theta \mathscr{G}_1,\partial_\phi \mathscr{G}_1\in \mathscr{C}^\alpha((0,\pi/2)\times\T).
$$
Finally, let us mention that $\partial_{f_i}\mathscr{G}_1(f_1,f_2)$ is continuous by using the results developed in \cite[Proposition 5.2]{GHM}.

\medskip
\noindent
{{{{\bf Step 2:} \it Symmetry persistence.}}} We will check the symmetry properties for $\tilde{F}_1$, and $\tilde{F}_2$ follows similarly. { By the decomposition of the stream function in \eqref{streamfunction-decomp} one has}
\begin{equation}\label{F1tilde}
\tilde{F}_1(\Omega,f_1,f_2)=\frac{1}{r_{0,1}(\phi)}\left\{\mathscr{T}_1(f_1,f_2)-\mathscr{T}_2(f_1,f_2)-\frac{\Omega}{2}r_1^2(\phi,\theta)-m_1(\Omega,f_1,f_2)(\phi)\right\}.
\end{equation}
To check the symmetry we proceed in different steps.

\medskip
\noindent
 \ding{232} We begin by checking the equatorial symmetry, that is,
$$
\tilde{F}(\Omega,f_1,f_2)(\pi-\phi,\theta)=\tilde{F}(\Omega,f_1,f_2)(\phi,\theta), \quad\forall (\phi,\theta)\in[0,\pi]\times\T.
$$
From the expression of $\tilde{F}$ in \eqref{F1tilde}, it  suffices to check the property for  $\mathscr{T}_i(f_1,f_2)$. One can easily verify using the symmetry of the functions $\cos$, $r_1$ and $r_2$ combined with the change of variables $\varphi\mapsto \pi-\varphi$
\begin{align*}
\mathscr{T}_i(f_1,f_2)\left(\pi-\phi,\theta\right)=&-\frac{d_i}{4\pi}\bigintsss_{0}^{\pi}\bigintsss_0^{2\pi}
\bigintsss_0^{r_i(\varphi,\eta)}\frac{r\sin(\varphi)drd\eta d\varphi}{|(re^{i\eta},d_i\cos(\varphi))-(r_1(\pi-\phi,\theta)e^{i\theta},d_1\cos(\pi-\phi))|}\\
=&-\frac{d_i}{4\pi}\bigintsss_{0}^{\pi}\bigintsss_0^{2\pi}\bigintsss_0^{r_i(\pi-\varphi,\eta)}
\frac{r\sin(\pi-\varphi)drd\eta d\varphi}{|(re^{i\eta},-d_i\cos(\varphi))-(r_1(\phi,\theta)e^{i\theta},-d_1\cos(\phi))|}\\
=&-\frac{d_i}{4\pi}\bigintsss_{0}^{\pi}\bigintsss_0^{2\pi}\bigintsss_0^{r_i(\varphi,\eta)}
\frac{r\sin(\varphi)drd\eta d\varphi}{|(re^{i\eta},d_i\cos(\varphi))-(r_1(\phi,\theta)e^{i\theta},d_1\cos(\phi))|}\\
=&\mathscr{T}_i(f_1,f_2)\left(\phi,\theta\right).
\end{align*}

\medskip
\noindent
 \ding{232} We shall check that $\tilde{F}$ can be written as a Fourier series in terms of $\cos$. In order to get  the desired structure, { using that $\frac{1}{2\pi}\int_0^{2\pi}\tilde{F}_i(\Omega,f_1,f_2)\left(\phi,\theta\right)d\theta=0$ and the definition of $r_j$ in \eqref{f}, } it suffices to check the following symmetry 
$$
\mathscr{T}_i(f_1,f_2)(\phi,-\theta)=\mathscr{T}_i(f_1,f_2)(\phi,\theta),\quad \forall \, (\phi,\theta)\in[0,\pi]\times\T.
$$
To do that, we use the symmetry of $r_i$, that is  $r_i(\varphi,-\theta)=r_i(\varphi,\theta),$ combined with  the change of variables $\eta\mapsto-\eta$ allowing to get
\begin{align*}
\mathscr{T}_i(f_1,f_2)(\phi,-\theta)=&-\frac{d_i}{4\pi}\bigintsss_{0}^{\pi}\bigintsss_0^{2\pi}\bigintsss_0^{r_i(\varphi,\eta)}\frac{r\sin(\varphi)drd\eta d\varphi}{|(re^{i\eta},d_i\cos(\varphi))-(r_1(\phi,-\theta)e^{-i\theta},d_1\cos(\phi))|}\\
=&-\frac{d_i}{4\pi}\bigintsss_{0}^{\pi}\bigintsss_0^{2\pi}\bigintsss_0^{r_i(\varphi,-\eta)}\frac{r\sin(\varphi)drd\eta d\varphi}{|(re^{-i\eta},d_i\cos(\varphi))-(r_1(\phi,\theta)e^{-i\theta},d_1\cos(\phi))|}\\
=&-\frac{d_i}{4\pi}\bigintsss_{0}^{\pi}\bigintsss_0^{2\pi}\bigintsss_0^{r_i(\varphi,\eta)}\frac{r\sin(\varphi)drd\eta d\varphi}{|(re^{i\eta},d_i\cos(\varphi))-(r_1(\phi,\theta)e^{i\theta},d_1\cos(\phi))|}\\
=&\mathscr{T}_i(f_1,f_2)(\phi,\theta).
\end{align*}

\medskip
\noindent
 \ding{232} Finally, we check the $m-$fold symmetry of $\tilde{F}$. Since the functions $r_i$ belong to $X_m^\alpha$, then  
 they satisfy the symmetry  $r_i(\varphi, \theta+\frac{2\pi}{m})=r_i(\varphi,\theta).$ Thus we get by the change of variables $\eta\mapsto \eta+\frac{2\pi}{m}$
\begin{align*}
&\mathscr{T}_i(f_1,f_2)\left(\phi,\theta+\tfrac{2\pi}{m}\right)
=-\frac{d_i}{4\pi}\int_{0}^{\pi}\int_0^{2\pi}\int_0^{r_i(\varphi,\eta)}\tfrac{r\sin(\varphi)drd\eta d\varphi}{|(re^{i\eta},d_i\cos(\varphi))-(r_1(\phi,\theta+\frac{2\pi}{m})e^{i(\theta+\frac{2\pi}{m})},d_1\cos(\phi))|}\\
&=-\frac{d_i}{4\pi}\int_{0}^{\pi}\int_0^{2\pi}\int_0^{r_i(\varphi,\eta+\frac{2\pi}{m})}\tfrac{r\sin(\varphi)drd\eta d\varphi}{|(re^{i(\eta+\frac{2\pi}{m})},d_i\cos(\varphi))-(r_1(\phi,\theta)e^{i(\theta+\frac{2\pi}{m})},d_1\cos(\phi))|}\\
&=-\frac{d_i}{4\pi}\int_{0}^{\pi}\int_0^{2\pi}\int_0^{r_i(\varphi,\eta)}\tfrac{r\sin(\varphi)drd\eta d\varphi}{|(re^{i\eta},d_i\cos(\varphi))-(r_1(\phi,\theta)e^{i\theta},d_1\cos(\phi))|}\\
&=\mathscr{T}_i(f_1,f_2)(\phi,\theta).
\end{align*}
Notice that we have used the fact that the Euclidean distance {in} $\C$ is invariant by the rotation action $z\mapsto e^{i\frac{2\pi}{m}} z$.
\end{proof}

\section{Main result and examples}\label{result}
{This section is devoted to  the main statement of this paper  together with the details of the proof. The results concerns specific domains subject to the  constraints {\bf (H)} and \eqref{assump1}. In Section \ref{sec-examples} we will give precise examples where all these conditions are satisfied and the main theorem can apply.}

\subsection{Main result}
This section aims to state the main result of this work together with its proof using all the previous results. Let us remind  that our main task is the search of time periodic solutions to the system \eqref{equation-model} taking the form 
\begin{equation}\label{rotating-sol}
q(t,x)=q_0(e^{-i\Omega t}(x_1,x_2),x_3), \quad q_0={\bf 1}_{D_1}-{\bf 1}_{D_2}={\bf 1}_{D_1\backslash D_2},
\end{equation}
for two bounded domains $D_1$ and $D_2$ of $\R^3$ such that $D_2\Subset D_1$ and each one is surrounded by a single surface. By Lemma \ref{Prop-stat} we have stationary solutions given by revolution shape domains $D_j=D_{0,j}$.  Hence, our purpose is to find non trivial (meaning non revolution shape domains) $D_j$ {\it close} to the trivial one $D_{0,j}$. They will be  described by means of the spherical type parametrization: 
$$
\partial D_j=\Big\{(r_j(\phi,\theta), d_j\cos(\phi));\, \theta\in\T, 0\leqslant\phi\leqslant\pi\Big\},
$$
with $r_j(0,\theta)=r_j(\pi,\theta)=0$, $d_1>d_2$, and 
$$
r_j(\phi,\theta)=r_{0,j}(\phi)+f_j(\phi,\theta),\quad f_j(\phi,\theta)=\sum_{j\geqslant 1}f_{j,n}(\phi)\cos(n\theta).
$$
 Hence, in the case $f_j=0$ we have a revolution shape domain and then a stationary patch. In that case, $D=D_1\backslash D_2$ describes a rotating solution with angular velocity $\Omega$ if and only if
$$
\tilde{F}(\Omega,f_1,f_2)(\phi,\theta)=0, \quad (\phi,\theta)=[0,\pi]\times\T,
$$
where $\tilde{F}$ is defined in \eqref{Ftilde}.

\begin{theo}\label{theorem}
Let $\alpha\in(0,1), j=1,2$ and $r_{0,j}:[0,\pi]\rightarrow \R$ satisfy   {\bf (H)} and \eqref{assump1}.
There exists $m_0\geqslant 2$ such that for any  $m\geqslant  m_0$ the following assertion occurs. There exists $\delta>0$ and one dimensional $\mathscr{C}^1$-curve $s\in(-\delta,\delta)\mapsto (\Omega(s), f_1(s),f_2(s))\in \R\times X_m^\alpha\times X_m^\alpha $, with
$$
f_j(0)=0, \quad f_j(s)\neq 0, \, \forall s\neq 0, \quad \textnormal{and} \quad \Omega(0)=\Omega_m,
$$
where $\Omega_m$ is defined in Proposition \ref{prop-kernel-onedim}, such that
$$
\tilde{F}(\Omega(s),f_1(s),f_2(s))=0, \quad \forall s\in(-\delta,\delta).
$$
\end{theo}
\begin{proof}
The main tool of the proof is the Crandall-Rabinowitz theorem which can be found in Appendix \ref{Ap-bif}. From Proposition \ref{prop-wellpos} we have that $\tilde{F}:\R\times X_m^\alpha\times X_m^\alpha$ is well-defined and is of class $\mathscr{C}^1$ for any $m\geqslant  2$ and $\alpha\in(0,1)$. Moreover, by Lemma \ref{Prop-stat} we get that $\tilde{F}(\Omega,0,0)=0$ for any $\Omega\in\R$. Thus,  it remains to check the required spectral properties of the linearized operator at $(\Omega,0,0)$. Notice that the expression of the linearized operator is given in Proposition \ref{Prop-MMlin1}. Furthermore, by Proposition \ref{prop-fredholmTT} we have that $\partial_{f_1,f_2}\tilde{F}(\Omega,0,0)$ is Fredholm of zero index, and hence the dimension of the kernel coincides with the codimension of the range. Moreover, again Proposition \ref{prop-fredholmTT} states that there exists $m_0\geqslant 2$ such that for $m\geqslant m_0$ and for $\Omega=\Omega_m$ defined in Proposition \ref{prop-kernel-onedim}, the dimension of the kernel of $\partial_{f_1,f_2}\tilde{F}(\Omega_m,0,0)$ equals to one. As to the transversality condition it is verified according to  Proposition \ref{prop-transversal}. This achieves the proof.
\end{proof}

\subsection{Spectral condition}\label{sec-examples}
In what follows we shall exhibit some explicit and implicit domains for which all the assumptions  {\bf (H)} and \eqref{assump1} are satisfied and Theorem \ref{theorem} may apply. The first discussion is relative to some conical domains. Remark  that any ellipsoid satisfies {\bf (H1)}--{\bf (H4)} since it agrees with $r_{0,j}(\phi)=a_j \sin(\phi)$ for some $a_j\in\R^+$. The case $a_j=d_j$ corresponds  to the sphere. The most subtle assumption is \eqref{assump1} and related to some nonlocal effects of the separation between the two domains. It was used in a crucial way along the spectral study developed in Section \ref{sec-spectral}. In the following proposition, we check that \eqref{assump1} is verified for the case that $D_{1}$ is an ellipsoid and $D_{2}$ is a sphere.

\begin{pro}\label{prop-ellipsoid}
Let $d_1>d_2>0,\, a>d_2$, $r_{0,1}(\phi)=a\sin(\phi)$ and $r_{0,2}(\phi)=d_2\sin(\phi)$. Denote by $D_2$ the ball centered at $0$ of radius $d_2$ and $D_1$ the ellipsoid with horizontal semiaxes $a$ and vertical semiaxis $d_1$. Then, the assumption  \eqref{assump1} holds true.
\end{pro}
\begin{proof}
In this particular case we can  compute $\nu_{j,\Omega}$   and hence find the explicit values of   $\overline{\Omega}_1$ and $\overline{\Omega}_2$ defined in \eqref{Omega1}--\eqref{Omega2}, which also verify \eqref{Omega1-2}--\eqref{Omega2-2}. In fact, by using \eqref{nu-streamfunction} we have that
$$
\nu_{1,\Omega}(\phi)=\frac{1}{r_{0,1}(\phi)}(\nabla_h \psi)(r_{0,1}(\phi)e^{i\theta},d_1\cos(\phi))\cdot e^{i\theta}-\Omega,
$$
and
$$
\nu_{2,\Omega}(\phi)=\Omega-\frac{1}{r_{0,2}(\phi)}(\nabla_h \psi)(r_{0,2}(\phi)e^{i\theta},d_2\cos(\phi))\cdot e^{i\theta}.
$$
Here $\psi$ satisfies $\Delta \psi={\bf 1}_{D_1}-{\bf 1}_{D_2}$ and  $D_2\subset D_1$ by the condition on the parameters. Since the stream function is linear with respect to the vorticity, we have that 
$$
\psi=\psi_1-\psi_2,
$$
where $\psi_i$ is the stream function associated to ${\bf 1}_{D_i}$, for $i=1,2$. Note that the stream function associated to an ellipsoid or sphere is well-known in the literature, see for instance \cite[Chapter 7, Section 6]{Kellog}.  Thus for the ball we find
\begin{align*}
\psi_2(x)=\frac16(x_1^2+x_2^2+x_3^2-3 d_2^2), \quad x\in D_2,\\
\psi_2(x)=-\frac{d_2^3}{3\sqrt{x_1^2+x_2^2+x_3^2}}, \quad x\in D_2^c,
\end{align*} 
and for the ellipsoid 
$$
\psi_1(x)=\alpha_1(a)(x_1^2+x_2^2)+\alpha_2(a)x_3^2+\alpha_3(a), \quad x\in D_1,
$$
with
\begin{align*}
\alpha_1(a)=\frac{a^2d_1}{4}\int_0^\infty \frac{ds}{(a^2+s)^2\sqrt{d_1^2+s}},\\
\alpha_2(a)=\frac{a^2d_1}{4}\int_0^\infty \frac{ds}{(a^2+s)(d_1^2+s)^\frac32},\\
\alpha_3(a)=-\frac{a^2d_1}{4}\int_0^\infty \frac{ds}{(a^2+s)\sqrt{d_1^2+s}}\cdot
\end{align*}
Therefore, we deduce that
\begin{align*}
\psi(x)=&\alpha_1(a)(x_1^2+x_2^2)+\alpha_2(a)x_3^2+\alpha_3(a)-\frac16(x_1^2+x_2^2+x_3^2-3 d_2^2) \quad x\in D_2,\\
\psi(x)=&\alpha_1(a)(x_1^2+x_2^2)+\alpha_2(a)x_3^2+\alpha_3(a)+\frac{\frac{d_2^3}{3}}{\sqrt{x_1^2+x_2^2+x_3^2}}, \quad x\in D_1\backslash D_2.
\end{align*} 
Computing the horizontal gradient we infer
\begin{align*}
\nabla_h\psi(x)=&\left(2\alpha_1(a)-\frac{1}{3}\right)(x_1,x_2), \quad x\in D_2,\\
\nabla_h\psi(x)=&2\alpha_1(a)(x_1,x_2)-\frac{\frac{d_2^3}{3}}{(x_1^2+x_2^2+x_3^2)^\frac32}(x_1,x_2), \quad x\in _1 D_1\backslash D_2,
\end{align*} 
which implies
\begin{align*}
\frac{1}{a\sin(\phi)}\nabla_h\psi(a\sin(\phi)e^{i\theta},d_1\cos(\phi))\cdot e^{i\theta}=&2\alpha_1(a)-\frac{\frac{d_2^3}{3}}{(d_1^2\cos^2(\phi)+a^2\sin^2(\phi))^\frac32},\\
\frac{1}{d_2\sin(\phi)}\nabla_h\psi(d_2\sin(\phi)e^{i\theta},d_2\cos(\phi))\cdot e^{i\theta}=&2\alpha_1(a)-\frac13\cdot
\end{align*} 
Then, $\nu_{1,\Omega}$ and $\nu_{2,\Omega}$ can be expressed as
\begin{align*}
\nu_{1,\Omega}=&2\alpha_1(a)-\frac{\frac{d_2^3}{3}}{(d_1^2\cos^2(\phi)+a^2\sin^2(\phi))^\frac32}-\Omega,\\
\nu_{2,\Omega}=&\Omega-2\alpha_1(a)+\frac13.
\end{align*}
We finally get that $\overline{\Omega}_2<\overline{\Omega}_1$ if 
$$
2\alpha_1(a)-\frac13<\inf_{\phi\in[0,\pi]} \left\{2\alpha_1(a)-\frac{d_2^3}{3}\frac{1}{(d_1^2\cos^2(\phi)+a^2\sin^2(\phi))^\frac32}\right\},
$$
or equivalently
$$
-\frac13<\inf_{\phi\in[0,\pi]} \left\{-\frac{d_2^3}{3}\frac{1}{(d_1^2\cos^2(\phi)+a^2\sin^2(\phi))^\frac32}\right\}.
$$
In the case of a prolate ellipsoid corresponding to $d_1>a$, we have
$$
-\frac13<\inf_{\phi\in[0,\pi]} \left\{-\frac{d_2^3}{3}\frac{1}{((d_1^2-a^2)\cos^2(\phi)+a^2)^\frac32}\right\}=-\frac{d_2^3}{3 a^3},
$$
which is satisfied since $a>d_2$. Otherwise if $d_1<a$ , which refers to an oblate ellipsoid, we find
$$
-\frac13<\inf_{\phi\in[0,\pi]} \left\{-\frac{d_2^3}{3}\frac{1}{(d_1^2+(a^2-d_1^2)\sin^2(\phi))^\frac32}\right\}=-\frac{d_2^3}{3 d_1^3},
$$
which hold true  since $d_1>d_2$.
\end{proof}
In the previous proposition we checked that \eqref{assump1} is satisfied for the case of the ellipsoid with a sphere. Moreover, since the inequality is stable under small perturbations, then  \eqref{assump1} holds for small perturbations of ellipsoids.
In the following, we will perform an asymptotic analysis to prove that when the  domains are  well-separated then  we achieve \eqref{assump1} together with ${\bf (H4)}$.

\begin{pro}\label{prop-d1}
Let $\gamma>0, d_2>0$ and $r_{0,2}$ satisfy ${\bf (H)}$. Assume that $r_{0,1}(\phi)=d_1^{1+\gamma}\overline{r}_{0,1}(\phi)$ satisfy  {\bf (H1)-(H3)}. Then, there exists $\overline{d}_1$ such that for $d_1\geqslant \overline{d}_1$, the conditions \eqref{assump1} and {\bf (H4)} are satisfied. 
\end{pro}

\begin{rem}
The condition $r_{0,1}(\phi)=d_1^{1+\gamma}\overline{r}_{0,1}(\phi)$ can be relaxed in the proof but we keep it for the sake of clarifying. Indeed, we just need $r_{0,1}(\phi)=d_1 f(d_1) \overline{r}_{0,1}(\phi)$, with 
$$
\lim_{d_1\rightarrow \infty} \frac{\log(d_1)}{f(d_1)}=0.
$$
\end{rem}

\begin{proof}
Let us start  with checking \eqref{assump1}. Recall  from \eqref{Omega1}--\eqref{Omega2} the expressions of $\overline{\Omega}_1$ and $\overline{\Omega}_2$,
\begin{align*}
\overline{\Omega}_1=&\inf_{\phi\in[0,\pi]}d_1\int_0^\pi H_{1,1}^1(\phi,\varphi)d\varphi-d_2\int_0^{\pi} H_{1,2}^1(\phi,\varphi)d\varphi,\\
\overline\Omega_2=&\sup_{\phi\in[0,\pi]}d_1\int_0^\pi H_{2,1}^1(\phi,\varphi)d\varphi-d_2\int_0^\pi H_{2,2}^1(\phi,\varphi)d\varphi,
\end{align*}
and let us analyze each term. Using the expression of $H_{i,j}^1$ in \eqref{HMM11} we find that
$$
H_{i,j}^1(\phi,\varphi)=\frac14\frac{\sin(\varphi)r_{0,j}^2(\varphi)}{R_{i,j}^\frac32(\phi,\varphi)}F_1\left(\frac{4 r_{0,i}(\phi)r_{0,j}(\varphi)}{R_{i,j}(\phi,\varphi)}\right),
$$
where
$$
R_{i,j}(\phi,\varphi)=(r_{0,i}(\phi)+r_{0,j}({\varphi}))^2+(d_i\cos(\phi)-d_j\cos{ (\varphi)})^2.
$$
Throughout all the proof, the parameter  $d_2>0$ is fixed and $d_1$ will grow to infinity. First, note that
\begin{align}\label{H22-limit}
\inf_{\phi\in(0,\pi)}d_2\int_0^\pi H_{2,2}^1(\phi,\varphi)d\varphi {>0,}
\end{align}
and this infimum  does not depend on $d_1$. Next, we shall  move to $d_1 H_{1,1}^1$ and observe  that
\begin{align*}
d_1 \frac{\sin(\varphi)r_{0,1}^2(\varphi)}{[R_{1,1}(\phi,\varphi)]^\frac32}=& \frac{d_1\sin(\varphi)r_{0,1}^2(\varphi)}{[ (r_{0,1}(\phi)+r_{0,1}({\varphi}))^2+(d_1\cos(\phi)-d_1\cos{ (\varphi)})^2 ]^\frac32}\\
=& \frac{d_1^{3+2\gamma}\sin(\varphi)\overline{r}_{0,1}^2(\varphi)}{[ (d_1^{1+\gamma}\overline{r}_{0,1}(\phi)+d_1^{1+\gamma}\overline{r}_{0,1}({\varphi}))^2+(d_1\cos(\phi)-d_1\cos{ (\varphi)})^2 ]^\frac32}\\
\leqslant &d_1^{3+2\gamma} \frac{\sin(\varphi)\overline{r}_{0,1}^2(\varphi)}{d_1^{3+3\gamma}\overline{r}_{0,1}^3(\varphi)}
\leqslant C d_1^{-\gamma},
\end{align*}
where the constant $C$ is uniform in $d_1$. Consequently,
\begin{align}\label{H11}
d_1\int_0^\pi H_{1,1}^1(\phi,\varphi)d\varphi\leq Cd_1^{-\gamma}\int_0^\pi F_1\left(\tfrac{4 r_{0,1}(\phi)r_{0,1}(\varphi)}{R_{1,1}(\phi,\varphi)}\right)d\varphi.
\end{align}
Moreover, from { \eqref{estimat-1}}, we have that for $x\in[0,1)$
$$
F_1(x)\leq C+C|\ln(1-x)|,
$$
and therefore
\begin{align}\label{F1-d1}
\nonumber F_1\left(\tfrac{4 r_{0,1}(\phi)r_{0,1}(\varphi)}{R_{1,1}(\phi,\varphi)}\right)\leqslant& C+C\log\left(\tfrac{d_1^{2+2\gamma}(\overline{r}_{0,1}(\phi)+\overline{r}_{0,1}({\varphi}))^2+d_1^2(\cos(\phi)-\cos{ (\varphi)})^2}{ d_1^{2+2\gamma}(\overline{r}_{0,1}(\phi)-\overline{r}_{0,1}({\varphi}))^2+d_1^2(\cos(\phi)-\cos{ (\varphi)})^2}\right)\\
\leqslant &C+C\log\left(1+\tfrac{d_1^{2+2\gamma}}{ d_1^{2+2\gamma}(\overline{r}_{0,1}(\phi)-\overline{r}_{0,1}({\varphi}))^2+d_1^2(\cos(\phi)-\cos{ (\varphi)})^2}\right).
\end{align}
Using \eqref{Chord} we find a constant $C>0$ independent of $d_1$ such that for any  $\phi,\varphi\in[0,\pi]$ 
$$
d_1^{2+2\gamma}(\overline{r}_{0,1}(\phi)-\overline{r}_{0,1}({\varphi}))^2+d_1^2(\cos(\phi)-\cos{ (\varphi)})^2\geqslant C d_1^2|\phi-\varphi|^2.
$$
Coming back to \eqref{F1-d1} we achieve
\begin{align}
\nonumber F_1\left(\frac{4 r_{0,1}(\phi)r_{0,1}(\varphi)}{R_{1,1}(\phi,\varphi)}\right)\leqslant &C+C\log\left(1+d_1^{2\gamma}|\phi-\varphi|^{-2}\right)\\
\leqslant &C+C\log(d_1)+C\log\left(1+|\phi-\varphi|^{-1}\right).\label{F1-d1-2}
\end{align}
Inserting \eqref{F1-d1-2} into  \eqref{H11} we finally obtain
\begin{align*}
d_1\int_0^\pi H_{1,1}^1(\phi,\varphi)d\varphi\leqslant Cd_1^{-\gamma}\int_0^\pi (\log(d_1)+|\log|\phi-\varphi||) d\varphi\leqslant C d_1^{-\gamma}\log(d_1),
\end{align*}
where $C$ is independent on  $d_1$. Hence, we conclude that
\begin{equation}\label{H11-limit}
\lim_{d_1\rightarrow \infty } d_1\sup_{\phi\in[0,\pi]}\int_0^\pi H_{1,1}^1(\phi,\varphi)d\varphi=0.
\end{equation}
To estimate $d_1 H_{2,1}^1$ we proceed as follows. First, we write
\begin{align*}
d_1 \frac{\sin(\varphi)r_{0,1}^2(\varphi)}{[R_{2,1}(\phi,\varphi)]^\frac32}=&d_1^{3+2\gamma} \frac{\sin(\varphi)\overline{r}_{0,1}^2(\varphi)}{[ ({r}_{0,2}(\phi)+d_1^{1+\gamma}\overline{r}_{0,1}({\varphi}))^2+(d_2\cos(\phi)-d_1\cos{ (\varphi)})^2 ]^\frac32}\\
\leqslant &d_1^{3+2\gamma} \frac{\sin(\varphi)\overline{r}_{0,1}^2(\varphi)}{d_1^{3+3\gamma}\overline{r}_{0,1}^3(\varphi)}\\
\leqslant &C d_1^{-\gamma},
\end{align*}
and then
\begin{align}\label{H21-d1-gamma}
d_1\int_0^\pi H_{2,1}^1(\phi,\varphi)d\varphi\leq Cd_1^{-\gamma}\int_0^\pi F_1\left(\tfrac{4 r_{0,2}(\phi)r_{0,1}(\varphi)}{R_{2,1}(\phi,\varphi)}\right)d\varphi.
\end{align}
Let us now work with the hypergeometric function. From straightforward computations we infer
\begin{align}\label{H21-hyp}
\nonumber F_1\left(\frac{4 r_{0,2}(\phi)r_{0,1}(\varphi)}{R_{2,1}(\phi,\varphi)}\right)\leq& C+C\log \left(\tfrac{({r}_{0,2}(\phi)+d_1^{1+\gamma}\overline{r}_{0,1}({\varphi}))^2+(d_2\cos(\phi)-d_1\cos{ (\varphi)})^2}{({r}_{0,2}(\phi)-d_1^{1+\gamma}\overline{r}_{0,1}({\varphi}))^2+(d_2\cos(\phi)-d_1\cos{ (\varphi)})^2}\right)\\\
\leqslant& C+C\log \left(1+\tfrac{d_1^{2+2\gamma}}{({r}_{0,2}(\phi)-d_1^{1+\gamma}\overline{r}_{0,1}({\varphi}))^2+(d_2\cos(\phi)-d_1\cos{ (\varphi)})^2}\right).
\end{align}
Take $\varphi\in[0,\pi/2]$ and let us separate in cases. If $\varphi\in[0,\pi/4)$ then for $d_1$ large enough we find that $\cos(\varphi)\geq \frac{\sqrt{2}}{2}$ and therefore
$$
d_1\cos(\varphi)-d_2\cos(\phi)= d_1\left(\cos(\varphi)-\frac{d_2}{d_1}\cos(\phi)\right)\geqslant\tfrac12 d_1,
$$
for $d_1$ large enough. On the other hand,  if $\varphi\in[\pi/4,\pi/2]$ then by   {\bf (H2)} we get  $\overline{r}_{0,1}(\varphi)>\delta$, for some constant $\delta$. That amounts to
$$
d_1^{1+\gamma}\overline{r}_{0,1}({\varphi})-{r}_{0,2}(\phi)= d_1^{1+\gamma}\left(\overline{r}_{0,1}({\varphi})-\frac{1}{d_1^{1+\gamma}}{r}_{0,2}(\phi)\right)\geqslant \frac{\delta}{2} d_1^{1+\gamma},
$$
for $d_1$ large enough. Then, combining the foregoing estimates and using the equatorial  symmetry of the profiles  we  find that for any $\phi,\varphi\in[0,\pi]$ 
\begin{equation}\label{H21-den}
({r}_{0,2}(\phi)-d_1^{1+\gamma}\overline{r}_{0,1}({\varphi}))^2+(d_2\cos(\phi)-d_1\cos{ (\varphi)})^2\geqslant C d_1^2,
\end{equation}
for some constant $C$ uniform in $d_1$. Plugging \eqref{H21-den} into \eqref{H21-hyp} we find
\begin{align*}
\nonumber F_1\left(\frac{4 r_{0,2}(\phi)r_{0,1}(\varphi)}{R_{2,1}(\phi,\varphi)}\right)\leqslant& C+C\log \left(1+d_1^{2\gamma}\right)\\
\leqslant& C+C \log d_1.
\end{align*}
Coming back to \eqref{H21-d1-gamma} we get
\begin{align*}
d_1\int_0^\pi H_{2,1}^1(\phi,\varphi)d\varphi\leqslant C d_1^{-\gamma }\log d_1.
\end{align*}
It follows that
\begin{align}\label{H21-limit}
\lim_{d_1\rightarrow \infty }d_1\sup_{\phi\in[0,\pi]}\int_0^\pi H_{2,1}^1(\phi,\varphi)d\varphi=0.
\end{align}
Let us  now finish with the estimate of  $H_{1,2}^1$. We write
\begin{align*}
 \frac{\sin(\varphi)r_{0,2}^2(\varphi)}{[R_{1,2}(\phi,\varphi)]^\frac32}=&\frac{\sin(\varphi)r_{0,2}^2(\varphi)}{[ (r_{0,1}(\phi)+r_{0,2}({\varphi}))^2+(d_1\cos(\phi)-d_2\cos{ (\varphi)})^2 ]^\frac32}\\
=&\frac{\sin(\varphi){r}_{0,2}^2(\varphi)}{[ (d_1^{1+\gamma}\overline{r}_{0,1}(\phi)+{r}_{0,2}({\varphi}))^2+(d_1\cos(\phi)-d_2\cos{ (\varphi)})^2 ]^\frac32}\cdot
\end{align*}
As before, we may use the equatorial symmetry of the integral and therefore reduce the analysis to $\phi\in[0,\pi/2]$.  In the case  $\phi\in(\pi/4,\pi/2]$, we may write
\begin{align*}
 \frac{\sin(\varphi)r_{0,2}^2(\varphi)}{[R_{1,2}(\phi,\varphi)]^\frac32}\leqslant &\frac{\sin(\varphi){r}_{0,2}^2(\varphi)}{d_1^{3+3\gamma}\overline{r}_{0,1}(\phi)}\leqslant C d_1^{-3-3\gamma},
\end{align*}
where the constant $C$ is uniform in $d_1$. On the other hand, if $\phi\in[0,\pi/4]$, we have that $\cos(\phi)\geq \frac{\sqrt{2}}{2}$, implying that
$$
d_1\cos(\phi)-d_2\cos(\varphi)=d_1 \left(\cos(\phi)-\frac{d_2}{d_1}\cos(\varphi)\right)\geqslant  \frac12 d_1,
$$
for $d_1$ large enough. Therefore, we deduce that
\begin{align*}
 \frac{\sin(\varphi)r_{0,2}^2(\varphi)}{[R_{1,2}(\phi,\varphi)]^\frac32}\leqslant C d_1^{-3}.
\end{align*}
Consequently, we obtain for any $\phi,\varphi\in[0,\pi]$
\begin{align*}
 \frac{\sin(\varphi)r_{0,2}^2(\varphi)}{[R_{1,2}(\phi,\varphi)]^\frac32}\leqslant C d_1^{-3}.
\end{align*}
Thus, we achieve
\begin{align*}
d_2\int_0^\pi H_{1,2}^1(\phi,\varphi)d\varphi\leqslant Cd_1^{-3}\int_0^\pi F_1\left(\tfrac{4 r_{0,1}(\phi)r_{0,2}(\varphi)}{R_{1,2}(\phi,\varphi)}\right)d\varphi.
\end{align*}
Using  the properties of Gauss hypergeometric function as for the $H_{2,1}^1$  we get
\begin{equation}\label{H12-limit}
\lim_{d_1\rightarrow \infty} d_2\sup_{\phi\in[0,\pi]}\int_0^\pi H_{1,2}^1(\phi,\varphi)d\varphi=0.
\end{equation}
Hence, by using \eqref{H22-limit}-\eqref{H11-limit}-\eqref{H21-limit}-\eqref{H12-limit} we find that as $d_1\rightarrow +\infty$, the condition $\overline{\Omega}_2<\overline{\Omega}_1$ agrees with
$$
-\inf_{\phi\in[0,\pi]} d_2\int_0^\pi H_{2,2}^1(\phi,\varphi)d\varphi<0,
$$
which is trivially satisfied. Therefore,  we find $\overline{d}_1$ large enough such that $\overline{\Omega}_2<\overline{\Omega}_1$ is verified  provided that $d_1\geq \overline{d}_1$.\\
It remains to check the validity of   ${\bf (H4)}$, that is, 
\begin{equation}\label{H4-2}
(d_1^{1+\gamma}\overline{r}_{0,1}(\phi)-r_{0,2}(\varphi))^2+(d_1\cos(\phi)-d_2\cos(\varphi))^2\geqslant \delta.
\end{equation}
If $\phi\in[0,\pi/4)$, then we easily get
\begin{align*}
(d_1^{1+\gamma}\overline{r}_{0,1}(\phi)-r_{0,2}(\varphi))^2+(d_1\cos(\phi)-d_2\cos(\varphi))^2\geqslant& (d_1\cos(\phi)-d_2\cos(\varphi))^2\\
\geqslant& d_1^2\left(\cos(\phi)-\frac{d_2}{d_1}\cos(\varphi)\right)\\
\geqslant & \delta.
\end{align*}
However, in the case $\phi\in[\pi/4,\pi/2]$ we may use the  first  term in the sum  in order to get
\begin{align*}
(d_1^{1+\gamma}\overline{r}_{0,1}(\phi)-r_{0,2}(\varphi))^2+(d_1\cos(\phi)-d_2\cos(\varphi))^2\geq& (d_1^{1+\gamma}\overline{r}_{0,1}(\phi)-r_{0,2}(\varphi))^2\\
\geq& d_1^{2+2\gamma}(\overline{r}_{0,1}(\phi)-d_1^{-1-\gamma}r_{0,2}(\varphi))^2\\
\geq & \delta,
\end{align*}
and achieving by this way the proof  \eqref{H4-2}. Thus, the proof of Proposition \ref{prop-d1} is now complete.
\end{proof}

\appendix

\section{Gauss Hypergeometric function}\label{Ap-spfunctions}
In this section, we intend to recall  the Gauss hypergeometric functions and collect  some of their basic properties. They were  needed before to recover  a compact and tractable form  of the linearized operator around the equilibrium state detailed  in Lemma \ref{nuOmega12}. Define for any real numbers $a,b\in \mathbb{R},\, c\in \mathbb{R}\backslash(-\mathbb{N})$ the Gauss hypergeometric function $z\mapsto F(a,b;c;z)$ on the open unit disc $\mathbb{D}$ by the power series
\begin{equation}\label{GaussF}
F(a,b;c;z)=\sum_{n=0}^{\infty}\frac{(a)_n(b)_n}{(c)_n}\frac{z^n}{n!}, \quad \forall z\in \mathbb{D},
\end{equation}
where the Pochhammer's  symbol $(x)_n$ is defined by
$$
(x)_n = \begin{cases}   1,   & n = 0, \\
 x(x+1) \cdots (x+n-1), & n \geq1,
\end{cases}
$$
and verifies
\begin{equation*}
(x)_n=x\,(1+x)_{n-1},\quad (x)_{n+1}=(x+n)\,(x)_n.
\end{equation*}
The series converges absolutely for all values of $|z|<1.$ For $|z|=1$ we have  the absolute convergence  if $\textnormal{Re} (a+b-c)<0$ and it diverges if $1\leqslant \textnormal{Re}(a+b-c)$. See \cite{Erdelyi} for more details. Moreover, let us recall the integral representation of the hypergeometric function, see for instance  \cite[p. 47]{Rainville}. Assume that  $ \textnormal{Re}(c) > \textnormal{Re}(b) > 0,$ then 
\begin{equation}\label{Ap-spfunctions-integ}
\hspace{1cm}F(a,b;c;z)=\frac{\Gamma(c)}{\Gamma(b)\Gamma(c-b)}\int_0^1 x^{b-1} (1-x)^{c-b-1}(1-zx)^{-a}~ dx,\quad \forall{z\in \C\backslash[1,+\infty)}.
\end{equation}

Next, we shall give the following classical result whose proof can be found in \cite{GHM}.
\begin{lem}\label{Lem-integral}
Let $n\in\N$, $\beta\geq0$ and $A>1$, then
$$
\bigintsss_0^{2\pi}\frac{\cos(n\theta)}{(A-\cos(\theta))^{\frac{\beta}{2}}}d\theta=\frac{2\pi}{(1+A)^{\frac{\beta}{2}+n}}\frac{\left(\frac{\beta}{2}\right)_n 2^n\left(\frac12\right)_n}{(2n)!}F\left(n+\frac{\beta}{2}, n+\frac12; 2n+1; \frac{2}{1+A}\right ).
$$
\end{lem}
Finally, let us describe the boundary behaviour of the Gauss Hypergeometric function $F(a,a;2a;x)$ at $1$, the proof can be found in \cite{GHM}.
\begin{pro}\label{Prop-behav} 
For $a>1$, there exists $C>0$ such that 
\begin{equation}\label{estimat-1}
\forall x\in[0,1),\quad F\left(a,a;2a; x\right)\le C\frac{|\ln(1-x)|}{x}\leq C+C|\ln(1-x)|.
\end{equation}
\end{pro}

\section{Bifurcation theory}\label{Ap-bif}
Bifurcation theory focuses on the topological transitions of the phase portrait through the variation of some parameters. More specifically, it deals with the stationary problem    $F(\lambda,x)=0,$ where $F:\R\times X\rightarrow Y$ is a  smooth  function between Banach spaces  $X$ and $Y$. 
Assuming that one has a trivial  solution,  $F(\lambda,0)=0$ for any $\lambda\in\R$, we would like to explore the bifurcation diagram in the neighborhood  of this elementary solution,  and see whether multiple branches of solutions may bifurcate  
from  a given point  $(\lambda_0,0)$, called  a bifurcation point. 

When the linearized operator around this point generates a Fredholm operator, then one  may  use Lyapunov--Schmidt 
reduction in order to reduce the infinite-dimensional problem to a finite-dimensional one, known as  the bifurcation equation. For this latter problem we need some specific  transversality conditions so that the    Implicit Function Theorem can be applied.  For more discussion in this subject, 
we refer to see \cite{Kato, Kielhofer}. Let us first recall some basic results on Fredholm operators.

\begin{defi}
Let $X$ and $Y$ be  two Banach spaces. A continuous linear mapping $T:X\rightarrow Y,$  is a  Fredholm operator if it fulfills the following properties,
\begin{enumerate}
\item $\textnormal{dim Ker}\,  T<\infty$,
\item $\textnormal{Im}\, T$ is closed in $Y$,
\item $\textnormal{codim Im}\,  T<\infty$.
\end{enumerate}
The integer $\textnormal{dim Ker}\, T-\textnormal{codim Im}\, T$ is called the Fredholm index of $T$.
\end{defi}
Next, we shall discuss  the index persistence through compact perturbations, see  \cite{Kato, Kielhofer}.
\begin{pro}
The index of a Fredholm operator remains unchanged under compact perturbations.
\end{pro}
Now, we recall the classical Crandall-Rabinowitz Theorem whose proof can be found  in \cite{CrandallRabinowitz}.

\begin{theo}[Crandall-Rabinowitz Theorem] 	\label{CR}
    Let $X, Y$ be two Banach spaces, $V$ be a neighborhood of $0$ in $X$ and $F:\mathbb{R}\times V\rightarrow Y$ be a function with the properties,
    \begin{enumerate}
        \item $F(\lambda,0)=0$ for all $\lambda\in\mathbb{R}$.
        \item The partial derivatives  $\partial_\lambda F_{\lambda}$, $\partial_fF$ and  $\partial_{\lambda}\partial_fF$ exist and are continuous.
        \item The operator $\partial_f F(0,0)$ is Fredholm of zero index and $\textnormal{Ker}(F_f(0,0))=\langle f_0\rangle$ is one-dimensional. 
                \item  Transversality assumption: $\partial_{\lambda}\partial_fF(0,0)f_0 \notin \textnormal{Im}(\partial_fF(0,0))$.
    \end{enumerate}
    If $Z$ is any complement of  $\textnormal{Ker}(\partial_fF(0,0))$ in $X$, then there is a neighborhood  $U$ of $(0,0)$ in $\mathbb{R}\times X$, an interval  $(-a,a)$, and two continuous functions $\Phi:(-a,a)\rightarrow\mathbb{R}$, $\beta:(-a,a)\rightarrow Z$ such that $\Phi(0)=\beta(0)=0$ and
    $$F^{-1}(0)\cap U=\{(\Phi(s), s f_0+s\beta(s)) : |s|<a\}\cup\{(\lambda,0): (\lambda,0)\in U\}.$$
\end{theo}


\begin{thebibliography}{99}


\bibitem{Andrews} G. R. Andrews, R. Askey, R. Roy, {\it Special Functions.} Cambridge University Press, 1999.


\bibitem{BHM22}{M.  Berti, Z. Hassainia, and N. Masmoudi, }Time quasi-periodic vortex patches, arXiv 
arXiv:2202.06215, 2022.

\bibitem{Erdelyi}{H. Bateman, } {\it Higher Transcendental Functions Vol. I--III.} McGraw-Hill Book Company, New York, 1953.

\bibitem{B-B} J. T. Beale, A. J. Bourgeois, {\it Validity of the quasi-geostrophic model for large-scale flow in the atmosphere and ocean,} SIAM J. Math. Anal. 25 (1994), no. 4, 1023–1068.

\bibitem{B-C} A. L. Bertozzi, P. Constantin, {\it Global regularity for vortex patches.} Comm. Math. Phys.
{\bf 152}(1) (1993), 19--28.





\bibitem{Burbea} { J. Burbea,} {\it Motions of vortex patches.} Lett. Math. Phys. {\bf 6} (1982), 1--16.

\bibitem{CMOV}{J.C Cantero, J. Mateu, J. Orobitg, J. Verdera, }{\it The regularity of the boundary of vortex patches for some non-linear transport equations. } 	arXiv:2103.05356, 2021.

\bibitem{Cas0-Cor0-Gom} {A. Castro, D. C\'ordoba, J. G\'omez-Serrano, }{\it Existence and regularity of rotating global solutions for the generalized surface quasi-geostrophic equations.} Duke Math. J. {\bf 165}(5) (2016), 935--984.
 
\bibitem{Cas-Cor-Gom} {A. Castro, D. C\'ordoba, J. G\'omez-Serrano, }{\it  Uniformly rotating analytic global patch solutions for active scalars}. J. Ann. PDE {\bf 2}(1) (2016),  Art. 1, 34.

\bibitem{CastroCordobaGomezSerrano} {A. Castro, D. C\'ordoba, J. G\'omez-Serrano, } {{\it Uniformly rotating smooth solutions for the incompressible 2D Euler equations.}} Arch. Ration. Mech. Anal. {bf 231}(2) (2019), 719--785.

{\bibitem{C-C-GS-2} A. Castro, D. C\'ordoba,  J. G\'omez-Serrano, {\it Global smooth solutions for the inviscid SQG equation,} Mem. Amer. Math. Soc. {\bf 266} (2020), no. 1292,
v+89.}

\bibitem{Charv}  F. Charve, {\it Convergence of weak solutions for the primitive system of the quasi-geostrophic equations.}  Asymptot. Anal. 42 (2005), no. 3-4, 173–209.
\bibitem{Chemin} {J.-Y. Chemin,} {{\it Persistance de structures g\'eometriques dans les  fluides incompressibles bidimensionnels. }} Ann. Sci. Ec. Norm. Sup. {\bf 26} (1993), 1--26.

\bibitem{CrandallRabinowitz} {M. G. Crandall, P. H. Rabinowitz, } {\it Bifurcation from simple eigenvalues.} J. Funct. Anal. {\bf 8} (1971), 321--340.




\bibitem{DeemZabusky} {G. S. Deem, N. J. Zabusky, } {\it Vortex waves: Stationary ``V-states'', Interactions, Recurrence, and Breaking.} Phys. Rev. Lett. {\bf 40} (1978), 859--862.


\bibitem{DelaHoz-Hassainia-Hmidi}{F. De la Hoz, Z. Hassainia, T. Hmidi, }{\it Doubly Connected V-States for the Generalized Surface Quasi-Geostrophic Equations.} Arch. Ration. Mech. Anal {\bf 220} (2016), 1209--1281. 

\bibitem{DelaHozHmidiMateuVerdera} {F. De la Hoz, T. Hmidi, J. Mateu, J. Verdera, } {\it Doubly connected V-states for the planar Euler equations. } SIAM J. Math. Anal. {\bf 48} (2016), 1892--1928.
 
{\bibitem{DHHM} {F. De la Hoz, Z. Hassainia, T. Hmidi, J. Mateu,} {\it An analytical and numerical study of steady patches in the disc.}  Anal. PDE {\bf 9}(7) (2016), 1609--1670.}
 
 
\bibitem{Desjardins-Grenier}{B. Desjardins, E. Grenier, }{\it Derivation of the Quasigeostrophic Potential Vorticity Equations, } Advances
in Differential Equations {\bf 3}(5) (1998), 715–752. 
 
\bibitem{D-S-R} D. G. Dritschel, R. K. Scott, J. N. Reinaud,  {\it  The stability of quasi-geostrophic ellipsoidal vortices.} J. Fluid Mech. 536 (2005), 401--421.

\bibitem{Dristchel2}{D. G. Dritschel, }{\it An exact steadily rotating surface quasi--geostrophic elliptical vortex. } Geophysical and Astrophysical Fluid Dynamics {\bf 105} (2011), 368--376.

 \bibitem{D-H-R} {D. G. Dritschel, T. Hmidi, C. Renault, } {\it
Imperfect bifurcation for the quasi-geostrophic shallow-water equations.} Arch. Ration. Mech. Anal. {\bf 231}(3) (2019), 1853--1915.

\bibitem{Dristchel}{D. G. Dritschel, J. N. Reinaud, W. J. McKiver, }{\it The quasi--geostrophic ellipsoidal vortex model. } J. Fluid Mech. {\bf 505} (2004), 201--223.
 
\bibitem{G-KVS}{C. Garc\'ia, }{\it K\'arm\'an Vortex Street in incompressible fluid models}. Nonlinearity {\bf 33}(4) (2020), 1625--1676. 

\bibitem{Gar-21}{C. Garc\'ia, }{\it Vortex patches choreography for active scalar equations, } J. Nonlinear Sci. {\bf 31} (2021), no. 5, Paper No.75, 31.

\bibitem{GH22}{C. Garc\'ia, S. V. Haziot, }{\it Global bifurcation for corotating and counter-rotating vortex pairs, }  arxiv:2204.11327, 2022.

\bibitem{GHM}{C. Garc\'ia, T. Hmidi, J. Mateu, }{\it Time periodic solutions for 3D quasi-geostrophic model. } Comm. Math. Phys. (2022) https://doi.org/10.1007/s00220-021-04290-w.
 

{ \bibitem{GHS}{C. Garc\'ia, T. Hmidi, J. Soler, }{\it Non uniform rotating vortices and periodic orbits for the two--dimensional Euler equations. } Arch. Ration. Mech. Anal. {\bf 238} (2020), 929--1086.}

\bibitem{G-S-19}{J. G\'omez-Serrano, }{\it On the existence of stationary patches. } Advances in Mathematics {\bf 343} (2019), 110–140.

\bibitem{GPSY}{J.  G\'omez-Serrano, J. Park, J. Shi, Y.Yao, } {\it Symmetry in stationary and uniformly-rotating solutions of active scalar equations}. Duke Math. J. {\bf 170} (2021), no. 13, 2957-3038.

\bibitem{Hassa-Hmi}  Z. Hassainia, T. Hmidi, {\it  On the V-states for the generalized quasi-geostrophic equations.} Comm. Math. Phys. {\bf 337}(1) (2015), 321--377.

\bibitem{HH21}{Z. Hassainia, T. Hmidi, } {\it Steady asymmetric vortex pairs for Euler equations, } American Institut of Mathematical Science {\bf 41} (2021), no. 4,
1939--1969.

\bibitem{HHH-18}{Z. Hassainia, T. Hmidi, F. de la Hoz,}{\it  Doubly connected V-states for the generalized surface quasi-geostrophic equations. } Cambridge University Press {\bf 439} (2018), 90–117.

\bibitem{HHHM-15}{Z. Hassainia, T. Hmidi, F. de la Hoz, J. Mateu, }{\it An analytical and numerical study of steady patches in the disc.} Analysis and PDE {\bf 9} (2015), no. 10.



\bibitem{HHM21}{Z. Hassainia, T. Hmidi, N. Masmoudi,} {\it KAM theory for active scalar equations, } arXiv 
arXiv:2110.08615, 2021.

\bibitem{HMW} Z. Hassainia, N. Masmoudi, M. H. Wheeler, {\it Global bifurcation of rotating vortex patches}. 	Comm. Pure Appl. Math.. doi:10.1002/cpa.21855, 2019.

\bibitem{HR22}{Z. Hassainia, E. Roulley, }{\it Boundary effects on the emergence of quasi-periodic solutions for Euler
equations, } arXiv:2202.10053, 2022.

\bibitem{HW21}{Z. Hassainia, M. H. Wheeler, }{\it Multipole vortex patch equilibria for active scalar equations, } arXiv:2103.06839, 2021.



\bibitem{HmidiHozMateuVerdera} {T. Hmidi, F. De la Hoz, J. Mateu, J. Verdera, }{\it  Doubly connected V-states for the planar Euler equations. } SIAM J. Math. Anal. {\bf 48}(3), 1892--1928.

{\bibitem{HmidiMateu}{T. Hmidi, J. Mateu, } {\it Bifurcation of rotating patches from Kirchhoff vortices.} Discret. Contin. Dyn. Syst. {\bf 36} (2016), 5401--5422.}
 
 \bibitem{H-M} T. Hmidi, J. Mateu,  {\it Existence of corotating and counter-rotating vortex pairs for active
scalar equations.} Comm. Math. Phys. {\bf 350}(2) (2017), 699--747.

\bibitem{HM16} T. Hmidi, J. Mateu, {\it Degenerate bifurcation of the rotating patches. } Adv. Math. {\bf 302} (2016), 799–850.

\bibitem{HmidiMateuVerdera} {T. Hmidi, J. Mateu, J. Verdera, } {\it Boundary regularity of rotating vortex patches. } Arch. Ration. Mech. Anal {\bf 209} (2013), 171--208. 

\bibitem{HMV15}{T. Hmidi, J. Mateu, J. Verdera, }{\it On rotating doubly connected vortices. } J. Differential Equations {\bf 258} (2015), no. 4, 1395-1429.

\bibitem{HR21}{T. Hmidi, E. Roulley, }{\it  Time quasi-periodic vortex patches for quasi-geostrophic shallow-water
equations, } arXiv:2110.13751, 2021.



\bibitem{Iftimie2}{D. Iftimie, }{\it Approximation of the quasigeostrophic system with primitive systems. } Asymptotic Analysis {\bf 21} (1999), 89--97.

\bibitem{Kato}{ T. Kato, }{\it Perturbation Theory for Linear Operators. }  Springer-Verlag, Berlin-Heidelberg-New York, 1995.

\bibitem{Kellog}{O.D. Kellog, }{\it Foundations of Potential Theory} Springer-Verlag, Berlin, 1967. 

\bibitem{Kielhofer}{ H. Kielh\"ofer, }{\it Bifurcation Theory: An Introduction with Applications to PDEs. }  Springer-Verlag, Berlin-Heidelberg-New York, 2004.

\bibitem{Kirchhoff}{G. R. Kirchhoff, }{\it  Vorlesungenber mathematische Physik. Mechanik.} Teubner, Leipzig, 1876.


\bibitem{Kiselev}{A. Kiselev, X. Luo, }{\it Illposedness of $C^2$ vortex patches. } 	arXiv:2204.06416, 2022.

{\bibitem{MB}{A. J. Majda, A. L. Bertozzi, }{\it Vorticity and incompressible flow,} volume 27 of Cambridge Texts in Applied Mathematics. Cambridge University Press, Cambridge, 2002.} 
 



{\bibitem{NPS}{J. Nieto, F. Poupaud, J. Soler, }{\it High-field limit for the Vlasov-Poisson-Fokker-Planck system, } Arch. Ration. Mech. Anal., 158(1), (2001).}




\bibitem{Ped} J. Pedlosky, {\it  Geophysical Fluid Dynamics,} second ed., Springer-Verlag, New York, 1987, pp. 1--710.

\bibitem{Rainville}{E. D. Rainville, }{\it Special Functions. } The Macmillan Co., 1973.


\bibitem{Rein-Drit} J. N. Reinaud, D. G. Dritschel, {\it The stability and nonlinear evolution of quasi-geostrophic toroidal vortices.} J. Fluid Mech. 863 (2019), 60--78.

\bibitem{Rein}  J. N. Reinaud, {\it Three-dimensional quasi-geostrophic vortex equilibria with m-fold symmetry}. J. Fluid Mech. 863 (2019), 32--59.

\bibitem{Serf} P. Serfati, {\it Une preuve directe d'existence globale des vortex patches 2D.}
C. R. Acad. Sci. Paris S\'er. I Math. {\bf 318}(6) (1994), 515--518.


\bibitem{Yudovich}{Y. Yudovich, }{\it Nonstationary flow of an ideal incompressible liquid. } Zh. Vych. Mat. {\bf 3} (1963), 1032--1066.

\end{thebibliography}
\end{document}